\documentclass[11pt]{article}
\usepackage[leqno]{amsmath}
\usepackage{amssymb,mathrsfs,bm}
\usepackage{hyperref}
\usepackage{mathtools}
\usepackage{fullpage}
\usepackage{epic}
\usepackage{eepic}
\usepackage[english]{babel}
\usepackage[all]{xy}
\usepackage{enumerate}
\usepackage[T1]{fontenc}
\usepackage[latin1]{inputenc}
\usepackage{color}
\usepackage{soul}
\usepackage{mathdots}
\usepackage{mathabx}
\usepackage{dsfont}
\usepackage{amsthm}

\newtheorem{theointro}{Theorem}

\newtheorem{theo}{Theorem}[subsection]
\newtheorem{lem}{Lemma}[subsection]

\newtheorem{prop}{Proposition}[subsection]
\newtheorem{cor}{Corollary}[subsection]
\newtheorem{rem}{Remark}[subsection]
\newtheorem*{finalremark}{Final remark}

\DeclareMathOperator{\Hom}{Hom}
\DeclareMathOperator{\Tr}{Trace}
\DeclareMathOperator{\Tra}{Tr}

\DeclareMathOperator{\Temp}{Temp}

\DeclareMathOperator{\Ad}{Ad}

\DeclareMathOperator{\vol}{vol}

\DeclareMathOperator{\cusp}{cusp}

\DeclareMathOperator{\Res}{Res}

\DeclareMathOperator{\Irr}{Irr}
\DeclareMathOperator{\geom}{geom}
\DeclareMathOperator{\spec}{spec}

\DeclareMathOperator{\cA}{\mathcal{A}}
\DeclareMathOperator{\bR}{\mathbb{R}}

\DeclareMathOperator{\bN}{\mathbb{N}}
\DeclareMathOperator{\bA}{\mathbb{A}}

\DeclareMathOperator{\fg}{\mathfrak{g}}
\DeclarePairedDelimiter\ceil{\lceil}{\rceil}
\DeclarePairedDelimiter\floor{\lfloor}{\rfloor}
\DeclareMathOperator{\cC}{\mathcal{C}}

\DeclareMathOperator{\cU}{\mathcal{U}}
\DeclareMathOperator{\cV}{\mathcal{V}}
\DeclareMathOperator{\cO}{\mathcal{O}}
\DeclareMathOperator{\cZ}{\mathcal{Z}}
\DeclareMathOperator{\cB}{\mathcal{B}}
\DeclareMathOperator{\cW}{\mathcal{W}}

\DeclareMathOperator{\Sm}{Sm}
\DeclareMathOperator{\fl}{fl}
\DeclareMathOperator{\Whitt}{Whitt}

\DeclareMathOperator{\Norm}{Norm}

\DeclareMathOperator{\unit}{unit}

\DeclareMathOperator{\disc}{disc}

\DeclareMathOperator{\Pet}{Pet}

\DeclareMathOperator{\GL}{\mathrm{GL}}

\DeclareMathOperator{\gen}{gen}

\DeclareMathOperator{\Cusp}{\mathrm{Cusp}}

\DeclareMathOperator{\cI}{\mathcal{I}}
\DeclareMathOperator{\ess}{ess}
\DeclareMathOperator{\Herm}{Herm}
\DeclareMathOperator{\bC}{\mathbb{C}}
\DeclareMathOperator{\Sqr}{\mathrm{Sqr}}

\DeclareMathOperator{\unr}{\mathrm{unr}}
\DeclareMathOperator{\alg}{\mathrm{alg}}
\DeclareMathOperator{\PGL}{\mathrm{PGL}}
\DeclareMathOperator{\RTF}{\mathrm{RTF}}
\DeclareMathOperator{\STF}{\mathrm{STF}}
\DeclareMathOperator{\BC}{\mathrm{BC}}

\numberwithin{equation}{subsection}
\newcounter{keepeqno}
\newenvironment{num}
 {\setcounter{keepeqno}{\value{equation}}%
  \begin{list}{(\theequation)}{\usecounter{equation}}%
  \setcounter{equation}{\value{keepeqno}}}
 {\end{list}}
 
\newcommand{\eq}[1][r]
   {\ar@<-3pt>@{-}[#1]
    \ar@<-1pt>@{}[#1]|<{}="gauche"
    \ar@<+0pt>@{}[#1]|-{}="milieu"
    \ar@<+1pt>@{}[#1]|>{}="droite"
    \ar@/^2pt/@{-}"gauche";"milieu"
    \ar@/_2pt/@{-}"milieu";"droite"}

\title{Multiplicities and Plancherel formula for the space of nondegenerate Hermitian matrices}
\author{Rapha\"el Beuzart-Plessis \protect\footnote{The project leading to this publication has received funding from Excellence Initiative of Aix-Marseille University-A*MIDEX, a French ``Investissements d'Avenir" programme.}}

\begin{document}

\maketitle

\begin{abstract}
This paper contains two results concerning the spectral decomposition, in a broad sense, of the space of nondegenerate Hermitian matrices over a local field of characteristic zero. The first is an explicit Plancherel decomposition of the associated $L^2$ space thus confirming a conjecture of Sakellaridis-Venkatesh in this particular case. The second is a formula for the multiplicities of generic representations in the $p$-adic case that extends previous work of Feigon-Lapid-Offen. Both results are stated in terms of Arthur-Clozel's quadratic local base-change and the proofs are based on local analogs of two relative trace formulas previously studied by Jacquet and Ye and known as (relative) Kuznetsov trace formulas.
\end{abstract}

\tableofcontents

\section{Introduction}

Let $E/F$ be a quadratic extension of local fields and let $n\geqslant 1$ be a positive integer. Set $G=\GL_n(E)$ and let $X=X_n$ be the space of nondegenerate Hermitian matrices i.e.
$$\displaystyle X=\left\{x\in G\mid {}^tx^c=x \right\}$$
where $c$ is the non-trivial Galois involution of $E/F$. There is a natural right action of $G$ on $X$ and $X$ carries an (unique up to a scalar) invariant measure for this action. We also set $G'=\GL_n(F)$ and $\BC:\Irr(G')\to \Irr(G)$ to be Arthur-Clozel's base-change map \cite{AC} between the smooth duals of $G'$ and $G$. The image of $\BC$ is then the set of irreducible smooth representations $\pi$ of $G$ that are Galois invariant i.e. satisfying $\pi\simeq \pi^c$.

The main theme of this paper is, roughly speaking, the ``spectrum'' of the space $X$. More precisely, we will consider the following two specific questions:
\begin{enumerate}[(1)]
\item $L^2$ version: give an explicit decomposition of $L^2(X)$ into a direct integral of unitary irreducible representations (Plancherel decomposition);

\item Smooth version: compute the multiplicity function $\pi\in \Irr(G)\mapsto m(\pi)=\dim \Hom_G(\pi,C^\infty(X))$ where $C^\infty(X)$ is the space of smooth functions on $X$ and $\Hom_G(.,.)$ stands for the space of $G$-equivariant (continuous)\footnote{The continuity requirement is only meaningful in the Archimedean case where $\pi$ should run over the {\em Casselman-Wallach globalizations} of irreducible Harish-Chandra modules and these naturally come with a (Fr\'echet) topology. However, in this paper we will only consider the multiplicities $m(\pi)$ when $F$ is non-Archimedean in which case these subtleties will not intervene.} linear maps.
\end{enumerate}

Note that for $x\in X$, the stabilizer $G_x$ is the unitary group preserving the Hermitian form naturally associated to $x$ and, by Frobenius reciprocity, we have
$$\displaystyle m(\pi)=\sum_{x\in X/G} \dim \Hom_{G_x}(\pi,\bC)$$
where $x$ runs overs $G$-orbits in $X$ (or, equivalently, equivalence classes of Hermitian forms on $E^n$) and $\Hom_{G_x}(\pi,\bC)$ is the space of (continuous) $G_x$-invariant linear forms on $\pi$ (so-called {\em local unitary periods}). The second problem has first been considered by Jacquet \cite{Jac01} who proved, using a global method, that when $n=3$ and $\pi$ is supercuspidal, $m(\pi)\neq 0$ if and only if $\pi\simeq \pi^c$ (i.e. $\pi$ is in the image of $\BC$) in which case each of the space $\Hom_{G_x}(\pi,\bC)$ is one-dimensional (so that $m(\pi)=2$ sinc $X$ has two $G$-orbits in this case). Following the same global approach and combining it with local methods, Feigon-Lapid-Offen \cite{FLO} have obtained extremely fine results on the multiplicities $m(\pi)$. In this paper, we will only propose a modest improvement on their work for generic representations. On the other hand, our solution to problem (1) seems new as it hasn't been adressed in the litterature yet but, again, to work it out we will make an extensive use of the work \cite{FLO} (which is again a generalization, and refinement, of Jacquet's work for $n=3$ \cite{Jac01}). The answers we obtain for both problems rely heavily on the base-change map $\BC$.

\subsection{Plancherel decomposition}

Our main result on problem (1) (Theorem \ref{theo PLancherel formula}) can be stated as follows.

\begin{theointro}\label{Planch intro}
There is a (natural) isomorphism of unitary $G$-representations
$$\displaystyle L^2(X)\simeq \int_{\Temp(G')}^{\oplus} \BC(\sigma) d\mu_{G'}(\sigma)$$
where $\Temp(G')\subset \Irr(G')$ is the tempered dual of $G'$ and $d\mu_{G'}$ the Plancherel measure for the group $G'$.
\end{theointro}

This theorem confirms, in the particular case at hand, a general conjecture of Sakellaridis-Venkatesh on the $L^2$-spectrum of spherical varieties \cite[Conjecture 16.2.2]{SV}. More precisely, Sakellaridis and Venkatesh associate to $X$ a dual group $\check{G}_X=\GL_n(\bC)=\check{G}'$ together with a ``distinguished morphism'' $\check{G}_X\to \check{G}$ to the Langlands dual group of $G$ (seen as an algebraic group over $F$). In \cite{SV}, only splits groups are considered so that there is no need to consider $L$-groups. This is not precisely the case here (since the group $G$ is not split over $F$) but the distinguished morphism naturally extends to the base-change map between $L$-groups ${}^L G'\to {}^L G$ and an obvious extrapolation\footnote{That the ``$L$-group'' of $X$ should really be ${}^L G'$ equipped with the base-change map ${}^L G'\to {}^L G$ is also consistent with a conjecture of Jacquet on distinction of irreducible representations by unitary groups. A refined version of this conjecture, due to Feigon-Lapid-Offen, will be discussed below.} of \cite[Conjecture 16.2.2]{SV} predicts a decomposition like the one of Theorem \ref{Planch intro}.

An immediate consequence of Theorem \ref{Planch intro} is to the determination of the so-called ``relative discrete series'' for $X$ i.e. of the unitary representations of $G$ that embed in the space $L^2(X,\chi)$ for some character $\chi$ of the center: these are precisely the base-change of discrete series of $G'$ (see Corollary \ref{cor discrete series}). Note that these representations are always tempered but not necessarily discrete series of the group $G$. It was already shown by Jerrod Smith \cite{Smith} that these representations are indeed relative discrete series but he didn't prove that they actually exhaust all of them.

The proof of Theorem \ref{Planch intro} actually gives more information. Namely, we define $G$-invariant semi-definite scalar products $\langle .,.\rangle_{X,\sigma}$ on $C_c^\infty(X)$, that are indexed by the irreducible tempered representations $\sigma$ of $G'$ and factorize through a quotient isomorphic to $\BC(\sigma)^\vee$ (for technical reasons, we prefer to take the smooth contragredient of the base-change), such that
\begin{equation}
\displaystyle \langle \varphi_1,\varphi_2\rangle_X=\int_{\Temp(G')} \langle \varphi_1,\varphi_2\rangle_{X,\sigma} d\mu_{G'}(\sigma)
\end{equation}
for every $\varphi_1,\varphi_2\in C_c^\infty(X)$ where $\langle .,.\rangle_X$ stands for the $L^2$-scalar product on $X$. That such a formula implies a decomposition like the one of Theorem \ref{Planch intro} follows from Bernstein \cite{Ber3} interpretation of abstract Plancherel decompositions. The scalar products $\langle .,.\rangle_{X,\sigma}$ are built on certain canonical $G$-equivariant embeddings $\cW(\BC(\sigma))\to C^\infty(X)$, where $\cW(\BC(\sigma))$ denotes the Whittaker model of $\BC(\sigma)$ (for a certain choice of Whittaker datum), that have been introduced by Feigon-Lapid-Offen \cite{FLO} in their work on the factorization of global unitary periods. By Frobenius reciprocity, these embeddings are equivalent to the data of $G_x$-invariant functionals $\alpha^\sigma_x: \cW(\BC(\sigma))\to \bC$ for $x\in X$ satisfying $\alpha^\sigma_{xg}=\alpha^\sigma_x \circ \BC(\sigma)(g)$ for $g\in G$. We call the $\alpha^\sigma_x$, $x\in X$, the {\em FLO functionals} associated to $\sigma$. The definition of those functionals by Feigon-Lapid-Offen is actually implicit: these are characterized by a series of identities between {\em relative Bessel distributions} through a certain transfer of functions $\varphi\in C_c^\infty(X)\mapsto f'\in C_c^\infty(G')$ that was established by Jacquet \cite{Jac}. One of the main result of \cite{FLO} is that these functionals give a factorization of global unitary periods of (cuspidal) automorphic forms on $\GL_n$ (thus generalizing a result of Jacquet \cite{Jac01} in the case $n=3$). In Section \ref{sect global periods}, we will reinterpret their result in a form that make the relation to the local scalar products $\langle .,.\rangle_{X,\sigma}$ more transparent. This simple cosmetic exercise has the pleasant feature of being remarkably aligned with certain general speculations of Sakellaridis-Venkatesh on relations between global automorphic periods and local Plancherel formulas \cite[\S 17]{SV}.

\subsection{Multiplicities}

As already said, the multiplicity $m(\pi)$ has already been extensively studied by Jacquet \cite{Jac01} and Feigon-Lapid-Offen \cite{FLO}. Their most complete result are for generic representations: when $\pi$ is generic, \cite[Theorem 0.2]{FLO} gives a lower bound for $m(\pi)$ which is attained for ``almost all'' generic $\pi$. We henceforth assume that $F$ is a $p$-adic field. In order to state the result of \cite{FLO} and our (small) improvement on it, we find it convenient to equip the sets $\Irr(G')$ and $\Irr(G)$ with structures of algebraic varieties over $\bC$. This construction is surely well-known, it is simply based on Langlands classification, but in lack of a proper reference we explain it in Section \ref{section structure of alg var} (see however \cite{Pras} for a similar construction on the Galois side). For these extra structures, the map $\BC$ is a finite morphism of algebraic varieties and we denote by $\deg \BC: \Irr(G)\to \bN$ the associated degree function (it sends a representation $\pi\in \Irr(G)$ to the sum of the degrees of $\BC$ at the elements in the fiber $\BC^{-1}(\pi)$). Since we are in the $p$-adic case, $G$ has two orbits in $X$ (corresponding to the two isomorphism classes of Hermitian spaces of dimension $n$). The following result is a restatement of \cite[Theorem 0.2]{FLO}.

\begin{theointro}[Feigon-Lapid-Offen]
Suppose that $\pi\in \Irr(G)$ is generic. Then, we have $m(\pi)\geqslant \deg \BC(\pi)$. More precisely, for each $x\in X$ we have
\begin{equation}\label{ineq FLO1}
\displaystyle \dim \Hom_{G_x}(\pi,\bC)\geqslant \left\{\begin{array}{ll}
\ceil{\frac{\deg \BC(\pi)}{2}} & \mbox{ if } G_x \mbox{ is quasi-split}, \\
\\
\floor{\frac{\deg \BC(\pi)}{2}} & \mbox{ otherwise.}
\end{array} \right.
\end{equation}
Moreover, if $\BC$ is unramified at (every point in the fiber of) $\pi$ then equality holds in \eqref{ineq FLO1}.
\end{theointro}

Our main result is that the above lower bound is actually always attained. More precisely, we show.

\begin{theointro}\label{theo multiplicities}
Let $\pi\in \Irr(G)$ be generic. Then, we have $m(\pi)=\deg \BC(\pi)$. In particular, equality always holds in \eqref{ineq FLO1}.
\end{theointro}

This result has been conjectured Feigon-Lapid-Offen \cite[Conjecture 13.17]{FLO} and it also confirms (in this particular case) a general conjecture of Prasad for Galois pairs \cite{Pras}.

\subsection{Tools: local trace formulas and Whittaker Paley-Wiener theorem}

The main new tools we introduce to prove Theorems \ref{Planch intro} and \ref{theo multiplicities} are certain local analogs of relative trace formulas first introduced in a global setting by Jacquet and Ye \cite{Ye88}, \cite{Ye89}, \cite{JY90}, \cite{JY92}, \cite{JY}, \cite{Jac01}. A comparison of these global relative trace formulas, that was established in general by Jacquet \cite{Jac}, \cite{Jac04}, \cite{Jac05}, led to a complete characterization of cuspidal automorphic representations that are {\em distinguished} by a given unitary group by Jacquet \cite{Jac10} (for the quasi-split unitary group) and then in general by Feigon-Lapid-Offen \cite{FLO}. This can be seen as a solution to the global analog of problem (2) above. Roughly speaking, we will deduce Theorems \ref{Planch intro} and \ref{theo multiplicities} through a similar local comparison. We note that local versions of the Jacquet-Ye trace formulas have been developed by Feigon \cite{Fe} in the case $n=2$ so that our treatment can be seen as a generalization of her work to arbitrary rank.

To be more precise, we develop local analogs of both the Kuznetsov trace formula (for an arbitrary quasi-split group) and of the relative Kuznetsov trace formula for $X$: these are identities relating so-called (relative) {\em Bessel distributions} (the spectral side) to (relative) orbital integrals (the geometric side). We refer the reader to the core of the text for details and precise statements (see in particular Theorems \ref{theo local Kuznetsov} and \ref{theo local TF for X}). We content ourself here to mention that these relative trace formulas are easy to establish. Namely, contrary to other formulas of the same sort, we can completely avoid analytic difficulties by using a regularization process of certain divergent oscillatory integrals due to Sakellaridis-Venkatesh \cite[Corollary 6.3.3]{SV} and generalized by Lapid-Mao in \cite[Proposition 2.11]{LM2} (this last result roughly says that integration over a maximal unipotent subgroup against a generic character of the latter behaves, in some respect, as a compact integration).

Another result that we will need to establish Theorem \ref{theo multiplicities} is a certain scalar Whittaker Paley-Wiener theorem describing, in the case of a quasi-split reductive $p$-adic group $G$, the image by some ``Bessel transform'' of the space of test functions $C_c^\infty(G)$ (see Section \ref{section PW theorem} for a precise statement). The result is far simpler to state than for the usual trace Paley-Wiener theorem \cite{BDK} and it is moreover an easy consequence of the theory of Jacquet's functionals. However, we have not seen this theorem stated elsewhere in the litterature (maybe because of simplicity).

%%\subsection{Outline}

%%We finish this introduction by a brief description of the content of the paper. In Chapter \ref{Chap Kuznetsov}, 
\subsection{Acknowledgement}

The results of this paper have been announced by the author in his ``Cours Peccot" in April-May 2017. I would like to thank the Coll\`ege de France for giving me the opportunity to give this course and I apologize for the delay in writing the details of this work.

I am very grateful to Jean-Loup Waldspurger for a thorough reading of a first version of this manuscript allowing for many improvements and in particular for catching a fatal error in the proof of Proposition \ref{prop local isom lambda}.

The project leading to this publication has received funding from Excellence Initiative of Aix-Marseille University-A*MIDEX, a French ``Investissements d'Avenir" programme.

\subsection{General notation}

\begin{itemize}
\item In the whole paper, $F$ denotes a local field of characteristic zero (Archimedean or non-Archimedean). In some specific sections (in particular, in the whole of Chapter \ref{Chapter multiplicities}), $F$ will be assumed to be $p$-adic but such restriction will always be explicitely stated.
\item For a smooth manifold $X$, we denote by $C_c^\infty(X)$ the usual space of test functions on $X$. For a totally disconnected locally compact space $X$, we denote by $C_c^\infty(X)$ the space of locally constant compactly supported complex functions on $X$.
\item If $f$ and $g$ are two positive functions on a set $X$, we write
$$\displaystyle f(x)\ll g(x),\;\; x\in X,$$
to mean that there exists a constant $C>0$ such that $f(x)\leqslant Cg(x)$ for every $x\in X$. If we want to emphasize that the implicit constant depends on auxilliary parameters $y_1,\ldots,y_k$ we write $f(x)\ll_{y_1,\ldots,y_k} g(x)$ instead.
\item The symbol $\widehat{\otimes}$ stands for the projective completed tensor product of locally convex topological vector spaces (cf.\, \cite[Chap. 43]{Tr}; this will only be used for Fr\'echet spaces).
\item When a group $G$ acts on the right (resp. on the left) of a set $X$, we denote by $R$ (resp. $L$) the corresponding action by translation on the space of functions on $X$.
\item If $G$ is a group and $S$ a subset of it, we write $\Norm_G(S)$ for the normalizer of $S$ in $G$.
\item For every integer $n\geqslant 0$, we denote by $\mathfrak{S}_n$ the symmetric group in $n$ letters. 
\item If $G$ is a Lie group, we write $\fg$ for its Lie algebra and $\cU(\fg)$ for the corresponding enveloping algebra.
\item Let $G$ be a real or $p$-adic reductive group. By a {\em smooth representation} of $G$ we mean a representation over a complex vector space with open stabilizers in the $p$-adic case, a smooth admissible Fr\'echet representation of moderate growth in the sense of Casselman-Wallach in the real case \cite{Cas}, \cite[Chap. 11]{WallII}. If $\pi$ is a smooth irreducible representation of $G$, we denote by $\pi^\vee$ its smooth contragredient (that is the Casselman-Wallach globalization of the admissible dual of the underlying Harish-Chandra module in the real case).
\item We denote the set of isomorphism classes of smooth irreducible representations of $G$ by $\Irr(G)$ and we write $\Temp(G)\subset \Irr(G)$ for the subset of tempered representations.
\item If $G$ is a $p$-adic reductive group, $H$ is a closed subgroup and $\pi$, $\sigma$ are smooth representations of $G$ and $H$ respectively, we write $\Hom_H(\pi,\sigma)$ for the space of $H$-equivariant linear maps $\pi\to \sigma$.
\item Still in the $p$-adic case, if $P$ is a parabolic subgroup of $G$ and $\sigma$ a smooth representation of one of its Levi component, we denote by $I_P^G(\sigma)$ the normalized smoth parabolic induction of $\sigma$.
\end{itemize}

\section{Local Kuznetsov trace formula and a scalar Whittaker Paley-Wiener theorem}\label{Chap Kuznetsov}

Let $F$ be a local field of characteristic zero (Archimedean or $p$-adic) and $\underline{G}$ be a quasi-split connected reductive group defined over $F$. The main goal of this chapter is to develop a local Kuznetsov trace formula for $\underline{G}(F)$ in the spirit of the work of Feigon \cite{Fe} for the group $\PGL_2(F)$.

More precisely, let $\underline{B}=\underline{T}\underline{N}$ be a Borel subgroup of $\underline{G}$ (defined over $F$) and $\underline{B}^-=\underline{T}\underline{N}^-$ be the opposite Borel subgroup (with respect to $\underline{T}$). We set $G=\underline{G}(F)$, $B=\underline{B}(F)$, $T=\underline{T}(F)$, $N=\underline{N}(F)$ and $N^-=\underline{N}^-(F)$. We denote by $\delta_B$ the modular character of $B$ and we fix an element $w\in G$ such that $N^-=w^{-1}Nw$. Let $\xi:N\to \mathbb{S}^1$ be a non-degenerate character (i.e. whose stabilizer in $T$ is reduced to the center of $G$). We define a non-degenerate unitary character $\xi^-:N^-\to \mathbb{S}^1$ by $\xi^-(u^-)=\xi(wu^-w^{-1})$ for every $u^-\in N^-$.

For $f_1,f_2\in C_c^\infty(G)$, we consider the kernel $K_{f_1,f_2}$ of the biregular action of $f_1\otimes \overline{f_2}$ on $L^2(G)$. Then, the distribution of interest is obtained, formally, by integrating this kernel over $N^-\times N$ against the character $(u^-,u)\mapsto \xi^-(u^-)^{-1}\xi(u)$. This expression is usually divergent and needs to be suitably regularized (see Section \ref{section Kuznetsov geom}). Once this is done, the resulting distribution admits two natural and distinct expansions: one geometric, in terms of relative orbital integrals, and one spectral, in terms of Bessel distributions also called relative characters. The equality between the two expansions is the aforementioned local Kuznetsov trace formula (cf.\, Theorem \ref{theo local Kuznetsov}).

The statements and proofs of these two expansions are given in Sections \ref{section Kuznetsov geom} and \ref{section Kusspec} respectively. For technical reasons, it will be more convenient to work with the Harish-Chandra Schwartz space $\cC(G)$ rather than $C_c^\infty(G)$. We recall the definition as well as basic properties of $\cC(G)$ and related function spaces in Section \ref{section HCS}. Finally, in Section \ref{section PW theorem} we give a scalar Paley-Wiener theorem for Bessel distributions in the $p$-adic case whose proof is an easy consequence of the theory of Jacquet's functionals (although we will rather work with the more convenient tool of the regularized $\xi$-integral introduced by Lapid-Mao \cite{LM2}).

We equip $N$ and $N^-$ with Haar measures such that the isomorphism $N\simeq N^-$, $u\mapsto w^{-1}uw$, is measure-preserving. We also endow $G$ and $T$ with Haar measures such that the following integration formula
\begin{align}\label{int formula}
\displaystyle \int_G f(g)dg=\int_{N^-\times T \times N} f(u^-tu)\delta_B(t)du^-dtdu
\end{align}
is satisfied for every $f\in L^1(G)$.

\subsection{Reminder on Harish-Chandra Schwartz space}\label{section HCS}

Let $\Xi^G$ be the Harish-Chandra basic spherical function of $G$ (see \cite[\S II.1]{Wald1}, \cite[\S II.8.5]{Var}). It depends on the choice of a maximal compact subgroup $K$ of $G$ that we assume fixed from now on. The function $\Xi^G$ is $K$-biinvariant and we have \cite[Lemme II.1.3]{Wald1}, \cite[Proposition 16(iii) p.329]{Var}
\begin{align}\label{eq 1 HCS}
\displaystyle \int_K\Xi^G(g_1kg_2)dk=\Xi^G(g_1)\Xi^G(g_2)
\end{align}
for every $g_1,g_2\in G$ and where the Haar measure on $K$ is normalized to have total mass $1$.

Let $\sigma_G$ be a log-norm on $G$ (see \cite[\S 1.2]{BeuGGP}). We assume that $\sigma_G$ is bi-$K$-invariant and satisfies $\sigma_G(g^{-1})=\sigma_G(g)$. There exists $d_0>0$ such that (\cite[Lemme II.1.5, Proposition II.4.5]{Wald1}, \cite[Proposition 31 p.340, Theorem 23 p.360]{Var})
\begin{align}\label{eq 2 HCS}
\displaystyle \int_G \Xi^G(g)^2 \sigma_G(g)^{-d_0}dg<\infty
\end{align}
and
\begin{align}\label{eq 3 HCS}
\displaystyle \int_{N^-} \Xi^G(u^-)\sigma_G(u^-)^{-d_0}du^-<\infty.
\end{align}

Let $\cC(G)$ be the {\em Harish-Chandra Schwartz space} of $G$. It is the space of functions $f:G\to \mathbb{C}$ which are $C^\infty$ in the Archimedean case, biinvariant by a compact-open subgroup in the $p$-adic case, and satisfy inequalities
$$\displaystyle \lvert f(g)\rvert \ll_d \Xi^G(g) \sigma_G(g)^{-d},\;\; g\in G$$
for every $d>0$ in the $p$-adic case;
$$\displaystyle \lvert (R(X)L(Y)f)(g)\rvert \ll_{d,X,Y} \Xi^G(g) \sigma_G(g)^{-d},\;\; g\in G$$
for every $d>0$ and every $X,Y\in \cU(\fg)$ in the Archimedean case.

There is a natural topology on $\cC(G)$ making it into a Fr\'echet space in the Archimedean case and a strict LF space in the $p$-adic case \cite[\S 1.5]{BeuGGP}. The Harish-Chandra Schwartz space $\cC(G\times G)$ of $G\times G$ is defined similarly. We will need the following, probably well-known, result.

%%and we have a separately continuous bilinear map

%%The composition of this bilinear map with the inclusion $\cC(G\times G)\subseteq \cF(G\times G)$ in the space of all complex-valued functions on $G\times G$ is continuous for the topology of pointwise convergence on $\cF(G\times G)$. Therefore, we obtain a continuous linear map $\cC(G)\widehat{\otimes}\cC(G)\to \cF(G\times G)$.

\begin{lem}\label{lem proj tensor prod}
Assume that $F$ is Archimedean. Then, there is a topological isomorphism $\cC(G)\widehat{\otimes}\cC(G)\simeq \cC(G\times G)$ sending a pure tensor $f_1\otimes f_2$ to the function $(g_1,g_2)\mapsto f_1(g_1)f_2(g_2)$.
\end{lem}

\begin{proof}
The bilinear map
$$\displaystyle \cC(G)\times \cC(G)\to \cC(G\times G)$$
$$\displaystyle (f_1,f_2)\mapsto \left( (g_1,g_2)\mapsto f_1(g_1)f_2(g_2)\right)$$
is continuous and therefore induces a continuous linear map
\begin{equation}\label{eq map completed tensor product}
\cC(G)\widehat{\otimes}\cC(G)\to \cC(G\times G).
\end{equation}
By \cite[end of Section 3.5]{Ber3}, $\cC(G)$ is nuclear. Hence, by Grothendieck's weak-strong principle \cite[th\'eor\`eme 13, Chap. II \S 3, n.3]{Gro}, the map \eqref{eq map completed tensor product} is injective with image the space of all functions $f: G\times G\to \bC$ satisfying the following condition:
\begin{equation*}
\mbox{For every }g\in G, T\in \cC(G)'\mbox{ the functions }g'\mapsto f(g,g')\mbox{ and }g\mapsto \langle f(g,.),T\rangle\mbox{ belong to }\cC(G).
\end{equation*}
But it is easy to see that every $f\in \cC(G\times G)$ satisfies this condition. Therefore, the linear map \eqref{eq map completed tensor product} is bijective and thus, by the open mapping theorem \cite[Theorem 17.1]{Tr}, a topological isomorphism.
\end{proof}

\begin{rem}
Assume that $F$ is non-Archimedean case. Let $J$ be a compact-open subgroup of $G$ and denote by $\cC(J\backslash G/J)$, $\cC(J\times J\backslash G\times G/ J\times J)$ the subspaces of $J$ and $J\times J$ biinvariant functions in $\cC(G)$ and $\cC(G\times G)$ respectively. We can show similarly the existence of a natural topological isomorphism $\cC(J\backslash G/J)\widehat{\otimes}\cC(J\backslash G/J)\simeq \cC(J\times J\backslash G\times G/ J\times J)$ but such isomorphism does no longer exist without fixing ``the level''. Indeed, there is a natural algebraic isomorphism $\cC(G)\widehat{\otimes}\cC(G)\simeq \cC(G\times G)$ which is however not topological. We refer the reader to \cite[Exemple 4, Chap. II \S 3 n.3 p.84]{Gro} for a detailed discussion of a similar issue for the projective tensor product $C_c^\infty(M)\widehat{\otimes}C_c^\infty(N)$ where $M$ and $N$ are infinitely differentiable real manifolds.
\end{rem}

We let $\cC^w(G)$ be the {\em weak Harish-Chandra Schwartz space} of $G$ that is the space of functions $f:G\to \mathbb{C}$ which are $C^\infty$ in the Archimedean case, biinvariant by a compact-open subgroup in the $p$-adic case, and for which there exists $d>0$ such that
$$\displaystyle \lvert f(g)\rvert \ll \Xi^G(g) \sigma_G(g)^{d},\;\; g\in G$$
in the $p$-adic case;
$$\displaystyle \lvert (R(X)L(Y)f)(g)\rvert \ll_{X,Y} \Xi^G(g) \sigma_G(g)^{d},\;\; g\in G$$
for every $X,Y\in \cU(\fg)$ in the Archimedean case. The space $\cC^w(G)$ is naturally equipped with a structure of LF space for which the subspace $C_c^\infty(G)$ is dense.

By \cite[Corollary 6.3.3]{SV}, \cite[Proposition 7.1.1]{BeuGGP}\footnote{Strictly speaking in {\it loc. cit.} only the case of unitary groups is treated but the arguments extend verbatim to the general case.} the linear form
$$\displaystyle f\in C_c^\infty(G)\mapsto \int_N f(u)\xi(u)du$$
extends continuously to $\cC^w(G)$. As in \cite[\S 7.1]{BeuGGP}, we denote by
$$\displaystyle f\in \cC^w(G)\mapsto \int_N^* f(u)\xi(u)du$$
this unique continuous extension that we will call the {\em $(N,\xi)$-regularized integral}. Let $\varphi\in C_c^\infty(T)$ and $f\in \cC^w(G)$. Define $\Ad(\varphi)f\in \cC^w(G)$ by
$$\displaystyle (\Ad(\varphi)f)(g)=\int_T \varphi(t) f(t^{-1}gt)dt,\;\;\; g\in G.$$
We also set
$$\displaystyle \widehat{\varphi}(u)=\int_T \varphi(t) \delta_B(t) \xi(tut^{-1})dt,\;\;\; u\in N.$$
Note that $\widehat{\varphi}$ is invariant by the derived subgroup $N'$ of $N$ and that it is ``rapidly decreasing'' (and even compactly supported in the non-Archimedean case) on $N/N'$ by usual properties of the Fourier transform. By the same argument as \cite[Lemma 7.1.2(ii)]{BeuGGP} we have
\begin{align}\label{eq 4 HCS}
\displaystyle \int_N^* (\Ad(\varphi)f)(u)\xi(u)du=\int_N f(u) \widehat{\varphi}(u)du
\end{align}
where the second integral is absolutely convergent. More precisely, for every $d>0$ we have
\begin{align}\label{eq 5 HCS}
\displaystyle \int_N \Xi^G(u) \sigma_G(u)^d \lvert \widehat{\varphi}(u)\rvert du<\infty.
\end{align}
Actually \eqref{eq 4 HCS} can be taken as a definition of the $(N,\xi)$-regularized integral since, by Dixmier-Malliavin \cite{DM}, any function of $\cC^w(G)$ is a finite sum of functions of the form $\Ad(\varphi)f$.

\subsection{Geometric expansion}\label{section Kuznetsov geom}

Let $f_1,f_2\in \cC(G)$. We set
$$\displaystyle K_{f_1,f_2}(x,y):=\int_G f_1(x^{-1}gy)\overline{f_2(g)}dg,\;\; x,y\in G.$$
Note that this expression is absolutely convergent by \eqref{eq 2 HCS}. More precisely, let $d_0>0$ be such that \eqref{eq 2 HCS} is satisfied. Then, from \eqref{eq 1 HCS}, \eqref{eq 2 HCS} and the inequality $\sigma_G(g_1g_2)\ll \sigma_G(g_1)\sigma_G(g_2)$ for every $g_1,g_2\in G$, it is easy to infer that
\begin{align}\label{eq 1 Kusgeom}
\displaystyle \lvert K_ {f_1,f_2}(x,y)\rvert\ll_d \Xi^G(x)\Xi^G(y)\sigma_G(x)^{-d}\sigma_G(y)^{d},\;\;\; x,y\in G,
\end{align}
for every $d>0$. Therefore by \eqref{eq 3 HCS} the expression
$$\displaystyle K_{f_1,f_2}^{N^-,\xi^-}(x):=\int_{N^-} K_{f_1,f_2}(u^-,x)\xi^-(u^-)^{-1}du^-$$
is absolutely convergent for any $x\in G$. We have
\begin{align}\label{eq 2 Kusgeom}
\displaystyle K_{f_1,f_2}^{N^-,\xi^-}\in \cC^w(G).
\end{align}
Indeed, in the $p$-adic case it is clear as $K_{f_1,f_2}^{N^-,\xi^-}$ is biinvariant by a compact-open subgroup and by \eqref{eq 1 Kusgeom} it satisfies
$$\displaystyle \lvert K_{f_1,f_2}^{N^-,\xi^-}(x)\rvert \ll \Xi^G(x)\sigma_G(x)^{d_0},\;\;\; x\in G$$
where $d_0$ is chosen such that the integral \eqref{eq 3 HCS} converges. In the Archimedean case, by differentiating under the integral sign (which is justified here by the absolute convergence of the resulting expression), we see that $K_{f_1,f_2}^{N^-,\xi^-}$ is $C^\infty$ and that
$$\displaystyle R(X)L(Y)K_{f_1,f_2}^{N^-,\xi^-}=K_{R(X)f_1,R(\overline{Y})f_2}^{N^-,\xi^-}$$
for every $X,Y\in \cU(\fg)$. Thus, by \eqref{eq 1 Kusgeom}, we have
$$\displaystyle \lvert R(X)L(Y)K_{f_1,f_2}^{N^-,\xi^-}(x)\rvert \ll_{X,Y} \Xi^G(x)\sigma_G(x)^{d_0},\;\;\; x\in G$$
for every $X,Y\in \cU(\fg)$ where $d_0$ is again chosen such that the integral \eqref{eq 3 HCS} converges. This proves the claim \eqref{eq 2 Kusgeom}.

By \eqref{eq 2 Kusgeom}, we can now define the following expression
\[\begin{aligned}
\displaystyle I(f_1,f_2) & :=\int^*_N K_{f_1,f_2}^{N^-,\xi^-}(u)\xi(u)du=\int^*_N \int_{N^-} K_{f_1,f_2}(u^-,u)\xi^-(u^-)^{-1}du^-\xi(u)du.
\end{aligned}\]

\begin{rem}
By being slightly more careful, we can show that $K_{f_1,f_2}^{N^-,\xi^-}\in \cC(G)$ so that the integral over $N$ above is actually absolutely convergent. However, the final expression is only convergent as an iterated double integral and we will not use this fact in the sequel.
\end{rem}

For $t\in T$ and $f\in \cC(G)$ we set
$$\displaystyle O(t,f)=\int_{N\times N^-} f(u^-tu) \xi(u)\xi^-(u^-)du^-du.$$

\begin{lem}\label{lem 1 Kusgeom}
The expression defining $O(t,f)$ is absolutely convergent locally uniformly in $t$ and $f$.
\end{lem}

\begin{proof}
After the change of variable $u\mapsto tut^{-1}$, we see that it suffices to show the existence of $d>0$ such that
$$\displaystyle \int_{N\times N^-} \Xi^G(u^-u)\sigma_G(u^-u)^{-d}du^-du<\infty.$$
By the Iwasawa decomposition, there exist functions $t_B:G\to T$, $u_B:G\to N$ and $k_B:G\to K$ such that $g=k_B(g)t_B(g)u_B(g)$ for every $g\in G$. As $\Xi^G$ and $\sigma_G$ are $K$-invariant, we have
\[\begin{aligned}
\displaystyle \int_{N\times N^-} \Xi^G(u^-u)\sigma_G(u^-u)^{-d}du^-du=\int_{N\times N^-} \Xi^G(t_B(u^-)u)\sigma_G(t_B(u^-)u)^{-d}du^-du.
\end{aligned}\]
By \cite[Proposition II.4.5]{Wald1} and \cite[Theorem 23 p.360]{Var} for any $d'>0$ we can choose $d$ such that the above expression is essentially bounded by
$$\displaystyle \int_{N^-} \delta_B(t_B(u^-))^{-1/2} \sigma_G(t_B(u^-))^{-d'}du^-.$$
Finally by \cite[Lemme II.3.4, Lemme II.4.2]{Wald1} and \cite[Theorem 4.5.4]{Wall} for $d'$ sufficiently large the last integral above converges. This proves the lemma.
\end{proof}

Set
$$\displaystyle I_{\geom}(f_1,f_2)=\int_T O(t,f_1)\overline{O(t,f_2)}\delta_B(t)dt.$$
The main result of this section is the following.

\begin{theo}\label{theo Kusgeom}
The expression defining $I_{\geom}(f_1,f_2)$ is absolutely convergent and moreover we have
$$\displaystyle I(f_1,f_2)=I_{\geom}(f_1,f_2).$$
\end{theo}

\begin{proof}
We extend the definition of $K_{f_1,f_2}$ to any $F\in \cC(G\times G)$ by
$$\displaystyle K_F(x,y)=\int_G F(x^{-1}gy,g)dg,\;\;\; x,y\in G.$$
We have $K_{f_1,f_2}=K_{f_1\otimes \overline{f_2}}$ where $f_1\otimes \overline{f_2}\in \cC(G\times G)$ is the function given by $(f_1\otimes \overline{f_2})(g_1,g_2)=f_1(g_1)\overline{f_2(g_2)}$. The same argument as before shows that
\begin{align}\label{eq 3 Kusgeom}
\displaystyle \displaystyle \lvert K_ {F}(x,y)\rvert\ll_{d,F} \Xi^G(x)\Xi^G(y)\sigma_G(x)^{-d}\sigma_G(y)^{d},\;\;\; x,y\in G
\end{align}
for any $d>0$ and $F\in \cC(G\times G)$. Therefore, we can define
$$\displaystyle K_F^{N^-,\xi^-}(x):=\int_{N^-} K_{F}(u^-,x)\xi^-(u^-)^{-1}du^-$$
for any $x\in G$ and $F\in \cC(G\times G)$ and by the same argument as for \eqref{eq 2 Kusgeom} we have $K_F^{N^-,\xi^-}\in \cC^w(G)$. Denote by $R^\Delta$ the right diagonal action of $T$ on $\cC(G\times G)$. 

In the $p$-adic case, we choose a compact-open subgroup $K_T$ of $T$ by which both $f_1$ and $f_2$ are right-invariant and we set $\varphi=\vol(K_T)^{-1}\mathbf{1}_{K_T}\in C_c^\infty(T)$, $F=f_1\otimes \overline{f_2}\in \cC(G\times G)$. Then, we have $f_1\otimes \overline{f_2}=R^\Delta(\varphi)F$. In the Archimedean case, by Dixmier-Malliavin \cite{DM}, $f_1\otimes \overline{f_2}$ is a finite sum of functions of the form $R^\Delta(\varphi)F$ where $\varphi\in C_c^\infty(T)$ and $F\in \cC(G\times G)$. For notational simplicity we will assume that $f_1\otimes \overline{f_2}=R^\Delta(\varphi)F$ for some functions $(\varphi,F)\in C_c^\infty(T)\times \cC(G\times G)$, the modifications needed to treat the general case are obvious.

In both cases, we have $K_{f_1,f_2}^{N^-,\xi^-}=K_{R^\Delta(\varphi)F}^{N^-,\xi^-}$ and a simple change of variable shows that $K_{R^\Delta(\varphi)F}^{N^-,\xi^-}=\Ad(\varphi)K_{F}^{N^-,\xi^-}$ (where the operator $\Ad(\varphi)$ was introduced in Section \ref{section HCS}). Hence, by \eqref{eq 4 HCS} we have
\[\begin{aligned}
\displaystyle I(f_1,f_2)=\int_N^* (\Ad(\varphi)K_{F}^{N^-,\xi^-})(u)\xi(u)du=\int_N K_{F}^{N^-,\xi^-}(u) \widehat{\varphi}(u)du
\end{aligned}\]
where the function $\widehat{\varphi}$ is defined as in Section \ref{section HCS}. Unfolding all the definitions, we arrive at the following equality:
\[\begin{aligned}
\displaystyle I(f_1,f_2)=\int_N \int_{N^-} \int_G F(u^-gu,g)dg \xi^-(u^-)du^-\widehat{\varphi}(u)du.
\end{aligned}\]
As follows readily from \eqref{eq 3 Kusgeom}, \eqref{eq 3 HCS} and \eqref{eq 5 HCS} this last expression is absolutely convergent. By \eqref{int formula}, we have
\[\begin{aligned}
\displaystyle I(f_1,f_2) & =\int_{N\times N^-} \int_{N^-\times T\times N}F(u^-v^-tvu,v^-tv)\delta_B(t)dvdtdv^-\xi^-(u^-)du^-\widehat{\varphi}(u)du \\
 & =\int_T \int_{N^2\times (N^-)^2} F(u^-tu,v^-tv)\xi^-(u^-)\xi^-(v^-)^{-1}\widehat{\varphi}(v^{-1}u)du^-dv^-dudv\delta_B(t)dt.
\end{aligned}\]
Set
$$\displaystyle O(t,F)=\int_{N^2\times (N^-)^2} F(u^-tu,v^-tv)\xi^-(u^-)\xi^-(v^-)^{-1} \xi(u)\xi(v)^{-1}du^-dv^-dudv$$
for every $t\in T$ and $F\in \cC(G\times G)$. By the same argument as for Lemma \ref{lem 1 Kusgeom}, this expression is absolutely convergent locally uniformly in $t$ and $F$. Note that
$$\displaystyle O(t,f_1\otimes \overline{f_2})=O(t,f_1)\overline{O(t,f_2)},\;\;\; t\in T.$$
We have (where all the manipulations are justified since $O(t,F)$ converges locally uniformly in $t$ and $F$)
\[\begin{aligned}
\displaystyle & \int_{N^2\times (N^-)^2} F(u^-tu,v^-tv)\xi^-(u^-)\xi^-(v^-)^{-1}\widehat{\varphi}(v^{-1}u)du^-dv^-dudv \\
 & =\int_{N^2\times (N^-)^2} F(u^-tu,v^-tv)\xi^-(u^-)\xi^-(v^-)^{-1}\int_T \varphi(a) \delta_B(a) \xi(av^{-1}ua^{-1})dadu^-dv^-dudv \\
 & =\int_T \varphi(a)\delta_B(a)^{-1}\int_{N^2\times (N^-)^2} F(u^-ta^{-1}ua,v^-ta^{-1}va)\xi^-(u^-)\xi^-(v^-)^{-1}\xi(u)\xi(v)^{-1}du^-dv^-dudvda \\
 & =\int_T \varphi(a)\delta_B(a)^{-1} O(a^{-1}t,R^\Delta(a)F)da.
\end{aligned}\]
Thus, the above computations show that the expression
\begin{align}\label{eq 4 Kusgeom}
\displaystyle \int_T\int_T \varphi(a)\delta_B(a)^{-1} O(a^{-1}t,R^\Delta(a)F)da\delta_B(t)dt
\end{align}
is convergent as an iterated integral for any $F\in \cC(G\times G)$ and $\varphi\in C_c^\infty(T)$ and moreover that
\begin{align}\label{eq 5 Kusgeom}
\displaystyle I(f_1,f_2)=\int_T\int_T \varphi(a)\delta_B(a)^{-1} O(a^{-1}t,R^\Delta(a)F)da\delta_B(t)dt
\end{align}
whenever $f_1\otimes \overline{f_2}=R^\Delta(\varphi)F$. We are now going to show that this last expression is absolutely convergent. In the $p$-adic case it is clear when $f_1=f_2$ as the integrand is nonnegative and the general case follows by Cauchy-Schwarz. In the Archimedean case, the argument is essentially the same but slightly less direct. We actually show the following:
\begin{num}
\item\label{eq 6 Kusgeom} The expression \eqref{eq 4 Kusgeom} converges absolutely for any $F\in \cC(G\times G)$ and $\varphi\in C_c^\infty(T)$. 
\end{num}

Let $\varphi\in C_c^\infty(T)$. As $\lvert \varphi\rvert$ is bounded by $\varphi'$ for some $\varphi'\in C_c^\infty(T)$, we may assume that $\varphi\geqslant 0$. Let $(T_n)_n$ be an increasing sequence of compact subsets of $T$ such that $T=\bigcup_n T_n$. It suffices to show that for every $\phi\in L^\infty(T\times T)$ the sequence of continuous linear forms
$$\displaystyle L_{n,\phi}:F\in \cC(G\times G)\mapsto \int_{T_n\times T_n} \phi(a,t)\varphi(a)O(a^{-1}t,R^\Delta(a)F)\delta_B(a^{-1}t)dadt$$
converges pointwise. By Lemma \ref{lem proj tensor prod} and \cite[(A.5.3)]{BeuGGP}, it suffices to show that for any $f_1,f_2\in \cC(G)$ the sequence $(L_{n,\phi}(f_1\otimes \overline{f_2}))_n$ converges for all $\phi\in L^\infty(T\times T)$ or what amounts to the same that the integral
$$\displaystyle \int_{T\times T} \varphi(a)O(a^{-1}t,R(a)f_1\otimes R(a)\overline{f_2})\delta_B(a^{-1}t)dadt=\int_{T\times T} \varphi(a)O(a^{-1}t,R(a)f_1)\overline{O(a^{-1}t,R(a)f_2)}\delta_B(a^{-1}t)dadt$$
is absolutely convergent. By Cauchy-Schwartz again, we just need to check that
$$\displaystyle \int_{T\times T} \varphi(a) \lvert O(a^{-1}t,R(a)f)\rvert^2 \delta_B(a^{-1}t)dadt<\infty$$
for every $f\in \cC(G)$. Letting $F=f\otimes \overline{f}$, we have $O(a^{-1}t,R^\Delta(a)F)=\lvert O(a^{-1}t,R(a)f)\rvert^2$.
Thus, for this particular choice of $F$ and $\varphi$ the integrand in \eqref{eq 4 Kusgeom} is nonnegative hence this expression, which is the same as above, is absolutely convergent. This proves the claim.

By \eqref{eq 5 Kusgeom} and \eqref{eq 6 Kusgeom}, we now have
\[\begin{aligned}
\displaystyle I(f_1,f_2) & =\int_T\varphi(a)\int_T O(a^{-1}t,R^\Delta(a)F)\delta_B(a^{-1}t)dtda=\int_T\varphi(a)\int_T O(t,R^\Delta(a)F)\delta_B(t)dtda \\
 & =\int_T\int_T \varphi(a) O(t,R^\Delta(a)F)da\delta_B(t)dt=\int_T O(t,R^\Delta(\varphi)F)\delta_B(t)dt \\
 & =\int_T O(t,f_1\otimes \overline{f_2})\delta_B(t)dt=I_{\geom}(f_1,f_2)
\end{aligned}\]
where all the above expressions are absolutely convergent. This proves the theorem.
\end{proof}

\subsection{Spectral expansion}\label{section Kusspec}

Let $\Temp(G)$ denote the set of isomorphism classes of irreducible tempered representations of $G$. This set carries a natural topology (see \cite[Section 2.6]{BeuPlanch}). Let $\pi\in \Temp(G)$. The representation $\pi$ is unitary and we fix an invariant scalar product $(.,.)$ on its space. Then, to every $f\in \cC(G)$ we can associate an operator $\pi(f)$ such that for $u,v$ smooth vectors in the space of $\pi$ we have
$$\displaystyle (\pi(f)u,v)=\int_G f(g)(\pi(g)u,v) vdg$$
where the integral converges absolutely. This operator is of trace-class (it is even of finite rank in the $p$-adic case) and the function
$$\displaystyle f_\pi:g\in G\mapsto \Tr(\pi(g^{-1})\pi(f))$$
belongs to $\cC^w(G)$ \cite[(2.2.5)]{BeuGGP}. According to Harish-Chandra \cite{H-C}, \cite{Wald1} (see also \cite{Ber3}) there exists a unique measure $d\mu_G(\pi)$ on $\Temp(G)$ such that
$$\displaystyle f(g)=\int_{\Temp(G)}f_\pi(g) d\mu_G(\pi)$$
for every $f\in \cC(G)$ and $g\in G$ where the right-hand side is an absolutely convergent integral.

For any $\pi\in \Temp(G)$ we define a {\em Bessel distribution} by
$$\displaystyle f\in \cC(G)\mapsto I_\pi(f):=\int_N^* f_\pi(w^{-1}u)\xi(u)du=\int_N^* \Tr(\pi(w)\pi(f)\pi(u^{-1})) \xi(u)du.$$

Let $f_1,f_2\in \cC(G)$. We set
$$\displaystyle I_{\spec}(f_1,f_2):=\int_{\Temp(G)} I_\pi(f_1)\overline{I_\pi(f_2)}d\mu_G(\pi).$$
The main result of this section is the following.

\begin{theo}\label{theo Kusspec}
The expression defining $I_{\spec}(f_1,f_2)$ is absolutely convergent and moreover we have
$$\displaystyle I(f_1,f_2)=I_{\spec}(f_1,f_2).$$
\end{theo}

\begin{proof}
First we consider the convergence of $I_{\spec}(f_1,f_2)$. By \cite[Proposition 2.131]{BeuPlanch} the functions $\pi\in \Temp(G)\mapsto I_\pi(f_1)$ and $\pi\mapsto I_\pi(f_2)$ are continuous and compactly supported in the $p$-adic case whereas there are continuous and essentially bounded by $N(\pi)^{-k}$ for any $k>0$ in the Archimedean case where $N(.)$ is the ``norm'' on $\Temp(G)$ introduced in \cite[\S 2.6]{BeuPlanch}. Combining this with \cite[(2.7.4)]{BeuPlanch} we see that the integral defining $I_{\spec}(f_1,f_2)$ is absolutely convergent. Actually, using the full strength of \cite[Proposition 2.131]{BeuPlanch} we even have that $(f_1,f_2)\in \cC(G)^2\mapsto I_{\spec}(f_1,f_2)$ is a continuous sesquilinear form. By making the arguments for \eqref{eq 1 Kusgeom} and \eqref{eq 2 Kusgeom} effective, we have similarly that $(f_1,f_2)\in \cC(G)^2\mapsto I(f_1,f_2)$ is a (separately) continuous sesquilinear form. Therefore we just need to show the equality of the theorem for a dense subset of $\cC(G)$. In particular, we may assume that the operator-valued Fourier transform $\pi\in \Temp(G)\mapsto \pi(f_1)$ is compactly supported \cite[Theorem 2.6.1]{BeuGGP}. In this case the identity of the theorem is just a reformulation of \cite[Lemma 7.2.2(v)]{BeuGGP}\footnote{Once again only the case of unitary groups was considered in {\it loc. cit.} but the proof works equally well in the more general situation considered here.}.
\end{proof}

Combining Theorem \ref{theo Kusgeom} with Theorem \ref{theo Kusspec} we arrive at the following.

\begin{theo}[Local Kuznetsov trace formula]\label{theo local Kuznetsov}
For any $f_1,f_2\in \cC(G)$ we have
$$\displaystyle I_{\geom}(f_1,f_2)=I_{\spec}(f_1,f_2).$$
\end{theo}

\begin{rem}
Although not transparent from the notation, both sides depend on the choice of $w$: this dependence is quite transparent for $I_{\spec}(f_1,f_2)$ from the definition whereas for $I_{\geom}(f_1,f_2)$ the dependence is hidden in the definition of $\xi^-$ (given at the beginning of this chapter).
\end{rem}

\subsection{A scalar Whittaker Paley-Wiener theorem}\label{section PW theorem}

In this subsection we assume that $F$ is a $p$-adic field. Let $\widehat{\cZ}(G)$ be the Bernstein center of $G$ \cite{Ber2}. Then $\widehat{\cZ}(G)$ is a direct product of integral domains indexed by the Bernstein components of $G$. We let $\cZ(G)$ be the corresponding direct sum. Let $\Cusp(G)$ be the set of pairs $(M,\sigma)$ where $M$ is a semi-standard Levi subgroup of $G$ and $\sigma$ is the isomorphism class of a supercuspidal representation of $M$. There is a natural action of the Weyl group $W=\Norm_G(T)/T$ on $\Cusp(G)$ and the maximal spectrum of $\cZ(G)$ is in natural bijection with the quotient $\Cusp(G)/W$.

A smooth representation $\pi$ of $G$ is said to be {\em $(N,\xi)$-generic} if $\Hom_N(\pi,\xi)\neq 0$. For $M$ a semi-standard Levi subgroup, we define similarly the notion of $(N^M,\xi^M)$-generic smooth representation of $M$ where $N^M=N\cap M$ and $\xi^M$ denotes the restriction of $\xi$ to $N^M$. We let $\Cusp_{\gen}(G)$ be the subset of $(M,\sigma)\in \Cusp(G)$ such that $\sigma$ is $(N^M,\xi^M)$-generic. It is known that a pair $(M,\sigma)\in \Cusp(G)$ belongs to $\Cusp_{\gen}(G)$ if and only if for one, or equivalently every, parabolic subgroup $P$ with Levi component $M$ the normalized smooth induction $I_P^G(\sigma)$ is $(\xi,N)$-generic in which case it contains an unique $(N,\xi)$-generic irreducible subquotient. Moreover, $\Cusp_{\gen}(G)$ is stable by the action of $W$ and $\Cusp_{\gen}(G)/W$ is a disjoint union of connected components of $\Cusp(G)/W$. We denote by $\cZ_{\gen}(G)$ the algebra of regular functions on $\Cusp_{\gen}(G)/W$ (thus, it is a direct factor of $\cZ(G)$).

Let $C^\infty(G)$ be the space of functions $G\to \mathbb{C}$ which are bi-invariant by some compact-open subgroup of $G$. It has a natural topology of LF space (for every compact-open subgroup $J$ we endow $C(J\backslash G/J)$ with the topology of pointwise convergence) for which the subspace $C_c^\infty(G)$ is dense. We will use the following very nice extension of \cite[Corollary 6.3.3]{SV} which is due to Lapid and Mao \cite[Proposition 2.11]{LM2}: the linear form
$$\displaystyle f\in C_c^\infty(G)\mapsto \int_N f(u)\xi(u)du$$
extends continuously to $C^\infty(G)$. As in Section \ref{section HCS}, we denote by
$$\displaystyle f\in C^\infty(G)\mapsto \int_N^* f(u)\xi(u)du$$
this unique continuous extension. Note that its restriction to $\cC^w(G)$ coincides with the $(N,\xi)$-regularized integral of Section \ref{section HCS} as the embedding $\cC^w(G)\subset C^\infty(G)$ is continuous.

Let $\Sm^{\fl}(G)$ be the category of smooth complex representations of $G$ which are of finite length. Let $\pi\in \Sm^{\fl}(G)$. To $f\in C_c^\infty(G)$ we associate the operator $\pi(f)$ such that for every vectors $v$, $v^\vee$ in the spaces of $\pi$ and $\pi^\vee$ (the smooth contragredient of $\pi$) we have
$$\displaystyle \langle \pi(f)v,v^\vee\rangle=\int_G f(g)\langle \pi(g)v,v^\vee\rangle dg.$$
This operator is of finite rank and the function $g\in G \mapsto \Tr(\pi(g)\pi(f))$ belongs to $C^\infty(G)$. We define the {\em Bessel distribution} $I_\pi$ by
$$\displaystyle I_\pi(f):=\int_N^* \Tr(\pi(w)\pi(f)\pi(u^{-1})) \xi(u)du,\;\;\; f\in C_c^\infty(G).$$
Obviously, when $\pi\in \Temp(G)$ this definition coincides with the restriction to $C_c^\infty(G)$ of the distribution defined in Section \ref{section Kusspec}. Note that $I_\pi$ only depends on the semi-simplification of $\pi$ (as it only depends on the distributional character of $\pi$). Thus, for $(M,\sigma)\in \Cusp_{\gen}(G)$ we can set $I_{M,\sigma}=I_{I_P^G(\sigma)}$ where $P$ is any parabolic subgroup with Levi component $M$.

Let $\cI_{\gen}(G)$ be the space of functions on $\Cusp_{\gen}(G)$ of the form $(M,\sigma)\mapsto I_{M,\sigma}(f)$ where $f\in C_c^\infty(G)$. The main result of this section is the following.

\begin{theo}\label{theo Whittaker-Paley-Wiener theorem}
We have
$$\displaystyle \cI_{\gen}(G)=\cZ_{\gen}(G).$$
\end{theo}

\begin{proof}
The inclusion $\cI_{\gen}(G)\subset \cZ_{\gen}(G)$ follows from \cite[Proposition 2.8]{LM2} and usual properties of the Jacquet functionals. Moreover, the action of the Bernstein center on $C_c^\infty(G)$ shows that $\cI_{\gen}(G)$ is an ideal of $\cZ_{\gen}(G)$. On the other hand, for any $(M,\sigma)\in \Cusp_{\gen}(G)$ the functional $I_{M,\sigma}$ is nonzero by \cite[Proposition 2.10]{LM2}. Hence, $\cI_{\gen}(G)$ is an ideal of $\cZ_{\gen}(G)$ which is not contained in any maximal ideal so that finally $\cI_{\gen}(G)=\cZ_{\gen}(G)$.
\end{proof}

\section{The symmetric space $X$ and FLO invariant functionals}\label{chapter invt funct}

\subsection{Groups and normalization of measures}\label{section groups, measures and the symmetric space X}

In this chapter we let $E/F$ be a quadratic extension of local fields of characteristic zero. We denote by $\Tra_{E/F}:E\to F$ the trace map and by $\eta$ be the quadratic character of $F^\times$ associated to this extension. We also fix a non-trivial unitary additive character $\psi':F\to \mathbb{S}^1$ and we let $\psi=\psi'\circ \Tra_{E/F}$.

Let $n\geqslant 1$. We set $G=\GL_n(E)$ and $G'=\GL_n(F)$. Let $T_n$, $N_n$ and $B_n$ be the algebraic subgroups of diagonal, unipotent upper triangular and upper triangular matrices of $\GL_n$ respectively. We set $T=T_n(E)$, $T'=T_n(F)$, $N=N_n(E)$, $N'=N_n(F)$, $B=B_n(E)$, $B'=B_n(F)$ and we denote by $\delta_B$, $\delta_{B'}$ the modular characters of $B$ and $B'$ respectively.

We denote by $c$ the non-trivial Galois automorphism of $E$ over $F$ and by $g\mapsto g^c$ the natural extension of $c$ to $G$. For $g\in G$, we also write ${}^t g$ for the transpose of $g$.

Using $\psi'$ and $\psi$ we define in the usual way non-degenerate characters $\psi'_n$ and $\psi_n$ of $N'$ and $N$ respectively: for every $u=(u_{i,j})_{1\leqslant i,j\leqslant n}\in N'$ we have
$$\displaystyle \psi'_n(u)=\psi'(\sum_{i=1}^{n-1}u_{i,i+1})$$
and similarly for $\psi_n$. Set
$$\displaystyle w=\begin{pmatrix} & & 1 \\ & \iddots & \\ 1 & &\end{pmatrix}.$$
Then we have $\psi'_n(wu^-w^{-1})=\psi_n'({}^t u^-)$ for every $u^-\in {}^t N'$.

We denote by $\Irr^{\gen}(G)\subseteq \Irr(G)$ (resp. $\Irr^{\gen}(G')\subseteq \Irr(G')$) the subset of generic irreducible representations and for $\pi\in \Irr^{\gen}(G)$ (resp. $\sigma\in \Irr^{\gen}(G')$) by $\cW(\pi,\psi_n)$ (resp. $\cW(\sigma,\psi'_n)$) the corresponding Whittaker model.

We equip $N'$, $T'$ and $G'$ with Haar measures such that the following integration formula
$$\displaystyle \int_{G'} f(g)dg=\int_{N'\times T'\times N'}f({}^t u_1tu_2)\delta_{B'}(t)du_1dtdu_2$$
is valid for every $f\in L^1(G')$.

Let $P_n$ be the mirabolic subgroup of $\GL_n$ (i.e. the subgroup of matrices with last row $(0,\ldots,0,1)$) and set $P=P_n(E)$, $P'=P_n(F)$. We equip $P$ (resp. $P'$) with a right Haar measure normalized such that setting
$$\displaystyle W_f(g_1,g_2)=\int_N f(g_1^{-1}ug_2)\psi_n(u)^{-1}du,\;\;\; g_1,g_2\in G$$
$$(\mbox{resp. } W_{f'}(g_1,g_2)=\int_{N'} f'(g_1^{-1}ug_2)\psi'_n(u)^{-1}du,\;\;\; g_1,g_2\in G'),$$
we have the Fourier inversion formulas
\begin{align}\label{eq1 Fourier inversion}
\displaystyle f(1)=\int_{N\backslash P} W_f(p,p)dp\; (\mbox{resp. } f'(1)=\int_{N'\backslash P'} W_{f'}(p,p)dp)
\end{align}
for every $f\in C_c^\infty(G)$ (resp. $f'\in C_c^\infty(G')$) see \cite[Lemma 4.4]{LM2}. Actually, the definition of $W_f$ and $W_{f'}$ extend to any $f\in \cC^w(G)$ and $f'\in \cC^w(G')$ by replacing the integrals over $N$ and $N'$ by the {\em regularized} one introduced in Section \ref{section HCS}. Then, the right-hand side of \eqref{eq1 Fourier inversion} is still absolutely convergent (this follows from \cite[Lemma 2.14.1 and Lemma 2.15.1]{BeuPlanch} in the degenerate case $E=F\times F$) and defines a continuous linear form on $\cC^w(G)$ or $\cC^w(G')$. Therefore, by density of $C_c^\infty(G)$ or $C_c^\infty(G')$ in $\cC^w(G)$ or $\cC^w(G')$, the inversion formula \eqref{eq1 Fourier inversion} continues to hold for every $f\in \cC^w(G)$ and $f'\in \cC^w(G')$.

For every $\sigma\in \Temp(G')$, the expression
\begin{align}\label{def pairing Whittaker}
\displaystyle \langle W,W'\rangle_{\Whitt}=\int_{N'\backslash P'} W(p)\overline{W'(p)}dp,\;\; W,W'\in \cW(\sigma,\psi'_n),
\end{align}
is absolutely convergent and defines a nonzero $G'$-invariant inner product on $\cW(\sigma,\psi'_n)$ by \cite{Ber1}, \cite{Bar}. This pairing allows to identify $\overline{\cW(\sigma,\psi'_n)}=\cW(\sigma^\vee,{\psi'_n}^{-1})$ with the smooth contragredient of $\cW(\sigma,\psi'_n)$. With our normalization of Haar measures, we have
\begin{align}\label{eq2 Fourier inversion}
\displaystyle \int_{N'}^* \langle R(u)W,W'\rangle_{\Whitt}\psi'_n(u)^{-1}du=W(1)\overline{W'(1)}
\end{align}
for every $\sigma\in \Temp(G')$ and $W,W'\in \cW(\sigma,\psi'_n)$ where the above regularized integral is taken in the sense of Section \ref{section HCS}. Indeed, the function $f(g)=\langle R(g)W,W'\rangle_{\Whitt}$, being a smooth matrix coefficient of a tempered representation, belongs to $\cC^w(G')$ and by unicity of the Whittaker model, there exists a constant $c$ (independent of $W$ and $W'$) such that $W_f(g_1,g_2)=cW(g_1)\overline{W'(g_2)}$. Applying the inversion formula \eqref{eq1 Fourier inversion}, we get
\[\begin{aligned}
\displaystyle \langle W,W'\rangle_{\Whitt}=f(1)=c\int_{N'\backslash P'} W(p)\overline{W'(p)}dp=c \langle W,W'\rangle_{\Whitt}.
\end{aligned}\]
As this is true for every $W,W'\in \cW(\sigma,\psi'_n)$, this shows that $c=1$ and the claim \eqref{eq2 Fourier inversion} is proved. Of course, a similar formula is valid for $G$.

\subsection{The symmetric space $X$}\label{S:symm X}

Let $h_V:E^n \times E^n\to E$ be a nondegenerate Hermitian form (our convention is that Hermitian forms are always linear in the first variable and antilinear in the second one). We denote by $V=(E^n,h_V)$ the associated Hermitian space and by $U(V)\subseteq G$ the corresponding unitary group defined by
$$\displaystyle U(V)=\{g\in G\mid h_V(gv,gv')=h_V(v,v')\; \forall v,v'\in E^n \}.$$
We also set $X_V=U(V)\backslash G$. Let $\mathcal{V}$ be a set of representatives of the isomorphism classes of Hermitian spaces of dimension $n$ over $E$ with underlying space $E^n$ (this set is finite and has two elements if $F$ is $p$-adic, $n+1$ if $F=\mathbb{R}$). \begin{equation}\label{def X}
\displaystyle X=\bigsqcup_{V\in \mathcal{V}} X_V.
\end{equation}

Let
$$\displaystyle \Herm_n^*=\{h\in G\mid {}^th^c=h \}$$
be the variety of invertible Hermitian matrices of size $n$. For each $V\in \cV$ we identify $h_V$ with the unique element of $\Herm_n^*$ such that
$$\displaystyle h_V(v,v')={}^t{v'}^c h_V v,\;\;\; v,v'\in E^n.$$
Then, there is an isomorphism $X\simeq \Herm_n^*$ given by
$$\displaystyle x\in X_V\mapsto {}^tx^c h_V x,\;\;\; V\in \cV.$$
This isomorphism sends the action by right translations of $G$ on $X$ to the right action of $G$ on $\Herm_n^*$ given by $h\cdot g={}^tg^c hg$. Besides, $\Herm_n^*$ admits a commuting left $F^\times$-action simply given by scalar multiplication. We denote by $(\lambda,x)\in F^\times\times X\mapsto \lambda x$ the corresponding action on $X$. Note that, when $n$ is odd or $F=\bR$, this extra action permutes certain components of the decomposition \eqref{def X}.

Note that $T'\subseteq \Herm_n^*$. We let $T_X$ be the subvariety of $X$ corresponding to $T'$ by the above isomorphism and we endow this set with the image of the Haar measure that we have fixed on $T'$. We also denote by $\delta_X$ the composition of the isomorphism $T_X\simeq T'$ with the modular character $\delta_{B'}$. Note that $T_X$ is invariant by translation by $T$ and consists of finitely many $T$-orbits. We equip $N$ with a Haar measure and $X$ with a $G$-invariant measure such that the following integration formula
\begin{align}\label{int formula X}
\displaystyle \int_X \varphi(x)dx=\int_N \int_{T_X} \varphi(tu)\delta_X(t)dtdu
\end{align}
is valid for every $\varphi\in L^1(X)$.

Whenever convergent, we denote by
$$\displaystyle \langle \varphi,\varphi'\rangle_X=\int_X \varphi(x)\overline{\varphi'(x)}dx$$
the $L^2$-inner product of two functions $\varphi,\varphi'\in C^\infty(X)$.

By \cite[Corollary 1.2]{GO}, for every $V\in \cV$ the pair $(G,U(V))$ is {\em tempered} in the sense of \cite[\S 2.7]{BeuGalP} that is:
\begin{equation}\label{X tempered}
\displaystyle\mbox{There exists }d>0\mbox{ such that the integral }\int_{U(V)} \Xi^G(h)\sigma_G(h)^{-d}dh\mbox{ is convergent.}
\end{equation}
As in the proof of \cite[Proposition 1.7.1]{BeuGalP}, this implies the following:
\begin{num}
\item\label{eq2 X tempered} For every $\varphi,\varphi'\in C_c^\infty(X)$ the function
$$\displaystyle g\in G\mapsto \langle R(g)\varphi,\varphi'\rangle_X$$
belongs to $\cC^w(G)$.
\end{num}

\subsection{Jacquet-Ye's transfer}\label{section JY transfer}

For $\varphi\in C_c^\infty(X)$, $f'\in C_c^\infty(G')$, $t\in T_X$ and $a\in T'$ we define the orbital integrals
$$\displaystyle O(t,\varphi)=\int_N \varphi(tu)\psi_n(u)du \;\; \mbox{ and } \;\; O(a,f')=\int_{N'\times N'} f'({}^tu_1a u_2) \psi'_n(u_1u_2)du_1du_2.$$
Note that these integrals are absolutely convergent as the integrand are compactly supported. For every $a\in T'$, we set
$$\displaystyle \gamma(a):=\prod_{k=1}^{n-1}\eta(a_k)^k$$
where $a_1,\ldots,a_n$ denote the diagonal entries of $a$. We say that the functions $\varphi\in C_c^\infty(X)$ and $f'\in C_c^\infty(G')$ {\em match} and we will write $\varphi\leftrightarrow f'$ if
$$\displaystyle \gamma(a)O(a,f')=O(t,\varphi)$$
whenever $t\in T_X$ maps to $a\in T'$ via the isomorphism $T_X\simeq T'$.

The following theorem is due to Jacquet \cite{Jac} (in the $p$-adic case) and Aizenbud-Gourevitch \cite{AG} (in the Archimedean case).
\begin{theo}[Jacquet, Aizenbud-Gourevitch]\label{theo JY transfer}
Every $\varphi\in C_c^\infty(X)$ matches a function $f'\in C_c^\infty(G')$. Conversely, every $f'\in C_c^\infty(G')$ matches a function $\varphi\in C_c^\infty(X)$.
\end{theo}

\subsection{Feigon-Lapid-Offen's functionals}\label{section FLO funct}

Let $\pi\in\Temp(G)$. We denote by $\mathcal{E}_G(X,\cW(\pi,\psi_n)^*)$ the set of all maps
$$\displaystyle \alpha: X\times \cW(\pi,\psi_n)\to \mathbb{C}$$
which are $G$-invariant for the diagonal $G$-action i.e. satisfying $\alpha(xg,R(g^{-1})W)=\alpha(x,W)$ for every $x\in X$, $W \in \cW(\pi, \psi_n)$ and $g\in G$, and such that $W\in \cW(\pi,\psi_n)\mapsto \alpha(x,W)$ is a continuous linear functional for every $x\in X$ (the continuity condition is only for the Archimedean case). Let $x_1, \ldots,x_k$ be a family of representatives for the $G$-orbits in $X$, then we have an isomorphism
$$\displaystyle \mathcal{E}_G(X,\cW(\pi,\psi_n)^*)\simeq  \bigoplus_{i=1}^k  \Hom_{G_{x_i}}(\cW(\pi,\psi_n),\bC),$$
$$\displaystyle \alpha\mapsto (\alpha(x_i,.))_{1\leqslant i\leqslant k}.$$
To any $\alpha\in \mathcal{E}_G(X,\cW(\pi,\psi_n)^*)$ we associate a {\em relative Bessel distribution} $J_{\pi}^\alpha:C_c^\infty(X)\to \mathbb{C}$ by
$$\displaystyle J_\pi^\alpha(\varphi):= \langle \varphi\cdot \alpha,\lambda^\vee_1\rangle,\;\; \varphi\in C_c^\infty(X),$$
where $\varphi\cdot \alpha$ is the smooth functional
$$\displaystyle W\in \cW(\pi,\psi_n)\mapsto \int_{X} \varphi(x) \alpha(x,W) dx$$
that we identify with an element of $\overline{\cW(\pi,\psi_n)}=\cW(\pi^\vee,\psi_n^{-1})$ via the invariant inner product $\langle .,.\rangle_{\Whitt}$ defined by \eqref{def pairing Whittaker} and $\lambda^\vee_1$ denotes the functional $W^\vee \mapsto W^\vee(1)$ on $\cW(\pi^\vee,\psi_n^{-1})$. Similarly for any $\sigma\in \Temp(G')$, we define a {\em Bessel distribution} $I_\sigma$ on $C_c^\infty(G')$ by
$$\displaystyle I_\sigma(f'):=\langle f'\cdot \lambda_w,\lambda^\vee_1\rangle,\;\; f'\in C_c^\infty(G'),$$
where $f'\cdot \lambda_w$ is the smooth functional
$$\displaystyle W\in \cW(\sigma,\psi_n')\mapsto \int_{G'} f'(g)W(wg)dg$$
that we again identify with an element of $\cW(\sigma^\vee,{\psi'_n}^{-1})$ via the pairing $\langle .,.\rangle_{\Whitt}$ and $\lambda^\vee_1$ denotes the functional $W^\vee \mapsto W^\vee(1)$ on $\cW(\sigma^\vee,{\psi'_n}^{-1})$. We have
\begin{num}
\item The above Bessel distribution $I_\sigma$ coincides with the one defined in Section \ref{section Kusspec}.
\end{num}
Indeed, since both functionals are continuous on $C_c^\infty(G')$ we just need to show the equality between them for functions $f'\in C_c^\infty(G')$ which are right-$K'$-finite. Let $f'\in C_c^\infty(G')$ which transforms for the right action according to a finite dimensional representation $\rho$ of $K'$. Let $\cB[\rho^\vee]$ be a basis of the $\rho^\vee$-isotypic component $\cW(\sigma,\psi'_n)[\rho^\vee]$ that is orthonormal with respect to the inner product $\langle .,.\rangle_{\Whitt}$. Then, denoting temporarily by $I'_\sigma$ the Bessel functional defined in Section \ref{section Kusspec}, by \eqref{eq2 Fourier inversion} we have
\[\begin{aligned}
\displaystyle I'_\sigma(f') & =\int_{N'}^* \sum_{W\in \cB[\rho^\vee]} \langle R(uw)R(f')W,W\rangle_{\Whitt} \psi'_n(u)^{-1}du \\ & =\sum_{W\in \cB[\rho^\vee]} (f'\cdot \lambda_w)(W) \lambda_1^\vee(\overline{W})=I_\sigma(f').
\end{aligned}\]

The following is \cite[Theorem 12.4]{FLO}.

\begin{theo}[Feigon-Lapid-Offen]\label{theo FLO funct}
Let $\sigma\in \Temp(G')$. Then, there exists an unique element
$$\displaystyle \alpha^\sigma\in \mathcal{E}_G(X,\cW(\BC(\sigma),\psi_n)^*)$$
such that we have the identity
$$\displaystyle J_{\BC(\sigma)}^{\alpha^\sigma}(\varphi)=I_\sigma(f')$$
for every pair of matching test functions $(\varphi,f')\in C_c^\infty(X)\times C_c^\infty(G')$.
\end{theo}

Let $\sigma \in \Temp(G')$ and $\alpha^\sigma\in \mathcal{E}_G(X,\cW(\BC(\sigma),\psi_n)^*)$ be as in the theorem above. We set $\alpha^\sigma_x=\alpha^\sigma(x,.)\in \Hom_{G_x}(\cW(\pi,\psi_n),\mathbb{C})$ for every $x\in X$ and we call them {\em the FLO functionals} associated to $\sigma$. By abuse of language, we shall also call $\alpha^\sigma$ the FLO functional associated to $\sigma$. For notational simplicity, we set
$$\displaystyle J_\sigma:=J_{\BC(\sigma)}^{\alpha^\sigma}$$
and call it {\em the FLO relative character} associated to $\sigma$.

Let $\lambda\in F^\times$ and $\varphi\in C_c^\infty(X)$. Then, for any matching test function $f'\in C_c^\infty(G')$ it is easy to see that the left translates $L(\lambda)\varphi=\varphi(\lambda^{-1}.)$ and $L(\lambda)f'=f'(\lambda^{-1}.)$ also match. From this and the characterization of the FLO functional, we readily infer that
\begin{equation}\label{equivariance FLO functionals center}
\displaystyle (L(\lambda)\varphi)\cdot \alpha^\sigma=\omega_\sigma(\lambda) \varphi\cdot \alpha^\sigma,\;\; \mbox{ for every } \sigma\in \Temp(G').
\end{equation}

\subsection{Harish-Chandra Schwartz and weak Harish-Chandra Schwartz spaces on X}\label{section HCS space for X}

\textbf{In this section and the next, we assume that $F$ is a $p$-adic field.}

For every $x\in X$ we set
$$\displaystyle \Xi^X(x)=\vol_X(xK)^{-1/2}.$$
Let $\sigma_X$ be a log-norm on $X$ (see \cite[\S 1.2]{BeuGGP}). We define the {\em Harish-Chandra Schwartz space} $\cC(X)$ as the space of functions $\varphi:X\to \bC$ which are right invariant by a compact-open subgroup of $G$ and such that for every $d>0$ we have
\begin{equation}\label{ineq HCS}
\displaystyle \lvert \varphi(x)\rvert\ll \Xi^X(x) \sigma_X(x)^{-d},\;\;\; x\in X.
\end{equation}
For every compact-open subgroup $J\subset G$, the subspace $\cC(X)^J\subset \cC(X)$ of right $J$-invariant functions is naturally a Fr\'echet space and therefore $\cC(X)=\bigcup_J \cC(X)^J$ is a strict LF space (that is a countable inductive limit of Fr\'echet spaces with closed embeddings as connecting morphisms). We have:

\begin{num}
\item\label{density HCS} The subspace $C_c^\infty(X)$ is dense in $\cC(X)$.
\end{num}

Indeed, let $J\subset G$ be a compact-open subgroup and $\varphi\in \cC(X)^J$. Let $(X_k)_{k\geqslant 1}$ be an increasing and exhausting sequence of $J$-invariant compact subsets of $X$. Then, the sequence $\varphi_k=\mathbf{1}_{X_k}\varphi$ belongs to $C_c^\infty(X)^J$ and converges to $\varphi$ in the Fréchet space $\cC(X)^J$ as can easily be seen from the fact that $\sigma_X(x)\to \infty$ as $x\to \infty$.

We define similarly the {\em weak Harish-Chandra Schwartz space} $\cC^w(X)$ as the space of functions $\varphi:X\to \bC$ which are right invariant by a compact-open subgroup of $G$ and satisfying the inequality
\begin{equation}\label{ineq HCSw}
\displaystyle \lvert \varphi(x)\rvert\ll \Xi^X(x) \sigma_X(x)^d,\;\;\; x\in X,
\end{equation}
for some $d>0$. For every compact-open subgroup $J\subset G$ and $d>0$, the subspace $\cC^w_d(X)^J\subset \cC^w(X)$ of right $J$-invariant functions which satisfies the estimates \eqref{ineq HCSw} for the given exponent $d$ is naturally a Fr\'echet space and therefore $\cC^w(X)=\bigcup_{J, d>0} \cC^w_d(X)^J$ is a LF space (that is a countable inductive limit of Fr\'echet spaces).

By \cite[Proposition 3.1.1(iii)]{BeuGalP}, for every $\varphi\in \cC(X)$ and $\varphi'\in \cC^w(X)$ the inner product $\langle \varphi,\varphi'\rangle_X$ converges absolutely.

%\begin{num}
%\item\label{eq1 HCS X} For every $\varphi\in \cC(X)$ and $\varphi'\in \cC^w(X)$ the function
%$$\displaystyle g\in G\mapsto \langle R(g)\varphi,\varphi'\rangle_X$$
%belongs to $\cC^w(G)$. Moreover, the resulting sesquilinear map $\cC(X)\times \cC^w(X)\to \cC^w(G)$ is separately continuous.
%\end{num}

%\begin{lem}\label{lem1 HCS X}
%Assume that $F$ is $p$-adic. Let $\pi\in\Temp(G)$ and $\iota: \pi\to C^\infty(X)$ be a $G$-equivariant linear map. Then, the image of $\iota$ is included in $\cC^w(X)$.
%\end{lem}

\begin{prop}\label{prop HCS}
\begin{enumerate}[(i)]
\item For every $(\varphi,\varphi')\in \cC(X)\times \cC^w(X)$ the function
$$\displaystyle g\in G\mapsto \langle R(g)\varphi,\varphi'\rangle_X$$
belongs to $\cC^w(G)$ and the resulting sesquilinear map $\cC(X)\times \cC^w(X)\to \cC^w(G)$ is separately continuous.

\item The action by right convolution
$$\displaystyle C_c^\infty(G)\times \cC(X)\to \cC(X)$$
$$\displaystyle (f,\varphi)\mapsto R(f)\varphi$$
extends to a separately continuous bilinear map $\cC(G)\times \cC(X)\to \cC(X)$.

\item Let $\pi\in \Temp(G)$ and $\iota: \pi\to C^\infty(X)$ be a $G$-equivariant linear map. Then, the image of $\iota$ lands in $\cC^w(X)$.
\end{enumerate}
\end{prop}

\begin{proof}
\begin{enumerate}[(i)]
\item According to \cite[Key Lemma, \S 3.4]{Ber3} we have equalities of topological vector spaces
\begin{equation}\label{eq1 Bernstein key lemma}
\displaystyle \cC(X)=\bigcap_{d>0} L^2(X,\sigma_X(x)^d dx)^\infty \mbox{ and } \cC^w(X)=\bigcup_{d>0} L^2(X, \sigma_X(x)^{-d}dx)^\infty
\end{equation}
where for $d\in \bR$, $L^2(X,\sigma_X(x)^d dx)$ stands for the space of smooth (that is right-invariant by a compact-open subgroup) square-integrable functions on $X$ with respect to the measure $\sigma_X(x)^d dx$. Let $\lVert .\rVert_{X,d}$ be the Hilbert norm on $L^2(X,\sigma_X(x)^d dx)$ and set $\lVert .\rVert_X=\lVert .\rVert_{X,0}$. We may assume, without loss in generality, that the log-norm $\sigma_X$ is right $K$-invariant.

Recall that for every $V\in \cV$, the pair $(G,U(V))$ is {\em tempered} in the sense of \cite[\S 2.7]{BeuGalP} (see \eqref{X tempered}). Hence, by \cite[Proposition 2.7.1]{BeuGalP}, the unitary $G$-representation $L^2(X)$ is tempered meaning that its Plancherel support is included in the set of irreducible tempered representations. From \cite[Theorem 2]{CHH}, it follows that for every compact-open subgroup $J\subset K$, there exists a constant $C_J>0$ such that
\begin{equation}\label{eq1 CHH}
\displaystyle \langle R(g)\varphi_1,\varphi_2\rangle_X\leqslant C_J \Xi^G(g) \lVert \varphi_1\rVert_{X} \lVert \varphi_2\rVert_X
\end{equation}
for every $\varphi_1,\varphi_2\in L^2(X)^J$ and $g\in G$.

Let now $d>0$, $J\subset K$ be a compact-open subgroup and $(\varphi_1,\varphi_2)\in L^2(X,\sigma_X(x)^d dx)^J\times L^2(X,\sigma_X(x)^{-d} dx)^J$. Then, we have $\sigma_X^{d/2}\lvert \varphi_1\rvert \in L^2(X)^J$ and $\sigma_X^{-d/2}\lvert \varphi_2\rvert \in L^2(X)^J$. Moreover, there exists a constant $C_0>0$ such that $\sigma_X(x)\leqslant C_0\sigma_X(xg) \sigma_G(g)$ for every $(x,g)\in X\times G$. Therefore, using \eqref{eq1 CHH}, we obtain
\[\begin{aligned}
\displaystyle \left\lvert \langle R(g)\varphi_1,\varphi_2\rangle_X \right\rvert & \leqslant \int_X \lvert \varphi_1\rvert(xg) \lvert \varphi_2\rvert(x) dx=\int_X \sigma_X(x)^{d/2}\lvert \varphi_1\rvert(xg) \sigma_X(x)^{-d/2}\lvert \varphi_2\rvert(x) dx \\
& \leqslant C_0 \sigma_G(g)^{d/2} \int_X \sigma_X(xg)^{d/2}\lvert \varphi_1\rvert(xg) \sigma_X(x)^{-d/2}\lvert \varphi_2\rvert(x) dx \\
& =C_0\sigma_G(g)^{d/2} \langle R(g)\sigma_X^{d/2}\lvert \varphi_1\rvert, \sigma_X^{-d/2}\lvert \varphi_2\rvert\rangle_X \\
 & \leqslant C_0C_J\Xi^G(g)\sigma_G(g)^{d/2} \lVert \varphi_1\rVert_{X,d} \lVert \varphi_2\rVert_{X,-d}
\end{aligned}\]
for every $g\in G$. Combined with \eqref{eq1 Bernstein key lemma}, this implies part (i) of the proposition.

\item Let $\varphi\in \cC(X)$. We need to show that the linear map $f\in C_c^\infty(G)\mapsto R(f)\varphi\in \cC(X)$ extends continuously to $\cC(G)$. The equalities \eqref{eq1 Bernstein key lemma} imply that, through the integration pairing $\langle .,.\rangle_X$, $\cC(X)$ gets identified with the space of smooth continuous anti-linear forms on $\cC^w(X)$. Let $f\in \cC(G)$. By (i), the anti-linear form
$$\displaystyle \varphi'\in \cC^w(X)\mapsto \int_G f(g)\langle R(g)\varphi,\varphi'\rangle_X dg$$
is well-defined and continuous. It is also smooth as $f$ is biinvariant by a compact-open subgroup. Therefore, there exists an unique element $R(f)\varphi\in \cC(X)$ such that
$$\displaystyle \int_G f(g)\langle R(g)\varphi,\varphi'\rangle_X dg=\langle R(f)\varphi,\varphi'\rangle_X$$
for every $\varphi'\in \cC^w(X)$. Moreover, this definition is easily seen to coincides with the action by right convolution when $f\in C_c^\infty(G)$. Finally, the linear map $f\in \cC(G)\mapsto R(f)\varphi\in \cC(X)$ is continuous by the closed graph theorem \cite[Corollary 4, \S 17]{Tr} since, by definition, for every $\varphi'\in \cC^w(X)$ the linear form $f\in \cC(G)\mapsto \langle R(f)\varphi,\varphi'\rangle_X$ is continuous.

\item The argument is similar to the proof of \cite[Lemma 4.2.1]{BeuGalP} so we only sketch it. The idea, which goes back to Lagier \cite{Lag} and Kato-Takano \cite{KT}, is to relate functions in the image of $\iota$ to smooth matrix coefficients of $\pi$ and then deduce the result from the known asymptotics for smooth matrix coefficients of tempered representations. More precisely, for each $V\in \cV$, denoting by $x_V\in X_V=U(V)\backslash G$ the canonical base-point, using the weak Cartan decomposition of Benoist-Oh \cite{BO} and Delorme-S\'echerre \cite{DS} (see also \cite[Lemma 5.3.1]{SV} for a different proof) we can construct as in \cite[Corollary 5.3.2]{SV} a subset $G_V^+\subset G$ such that
\begin{equation}\label{eq1 wavefront}
\displaystyle X_V=x_V G_V^+
\end{equation}
and (the so-called ``wave-front lemma'')
\begin{num}
\item\label{eq wavefront lemma} For every compact-open subgroup $J_1\subset G$, there exists another compact-open subgroup $J_2\subset G$ such that
$$\displaystyle x_VJ_2g\subset x_VgJ_1$$
for every $g\in G_V^+$.
\end{num}
Moreover, by \cite[Proposition 3.3.1 (ii)]{BeuGalP} (which holds as the pair $(G,U(V))$ is tempered in the sense of \cite[\S 2.7]{BeuGalP}, see \eqref{X tempered}) there also exists $d>0$ such that
\begin{equation}\label{eq2 wavefront}
\displaystyle \Xi^G(g)\ll \Xi^X(x_Vg)\sigma_X(x_Vg)^{d}, \; g\in G_V^+.
\end{equation}

Let $e\in \pi$ and $J_1\subset G$ be a compact-open subgroup leaving $e$ invariant. Let $J_2\subset G$ be as in \eqref{eq wavefront lemma} (for every $V\in \cV$). Then, by equivariance of $\iota$, for every $k_2\in J_2$ there exists $k_1\in J_1$ such that
$$\displaystyle \iota(e)(x_Vg)=\iota(e)(x_Vgk_1)=\iota(e)(x_Vk_2g)=\iota(\pi(k_2g)e)(x_V)$$
for every $V\in \cV$ and $g\in G_V^+$. Therefore,
$$\displaystyle \iota(e)(x_Vg)=\int_{K_2}\iota(\pi(k_2g)e)(x_V)dk_2=\langle \pi(g)e,e_V^\vee\rangle \mbox{ for } V\in \cV, g\in G_V^+$$
where $e_V^\vee$ is a certain vector in the smooth contragredient of $\pi$. By the asymptotic of smooth coefficients of tempered representations \cite{CHH}, we have $\lvert \langle \pi(g)e,e_V^\vee\rangle\rvert\ll \Xi^G(g)$ for $g\in G$, hence by \eqref{eq1 wavefront} and \eqref{eq2 wavefront} we get
\begin{equation*}
\displaystyle \lvert \iota(e)(x)\rvert\ll \Xi^X(x) \sigma_X(x)^d
\end{equation*}
for every $x\in X=\bigsqcup_{V\in \cV} X_V$. As the function $\iota(e)$ is also smooth, this shows that $\iota(e)\in \cC^w(X)$ and the proposition is proved.

\end{enumerate}
\end{proof}

\subsection{Abstract tempered relative characters}\label{section abstract rel chars}

In this section, we continue to assume that $F$ is a $p$-adic field. Let $\pi\in \Temp(G)$. We denote by $C_c^\infty(X)_\pi$ the $\pi^\vee$-coinvariant space of $C_c^\infty(X)$ i.e. the maximal quotient which is $G$-isomorphic to a direct sum of copies of $\pi^\vee$. We define the space of {\em abstract relative characters supported on} $\pi$ as the space
$$\displaystyle \Hom_N(C_c^\infty(X)_\pi,\psi_n)$$
of $(N,\psi_n)$-equivariant functionals on $C_c^\infty(X)_\pi$. Note that $J_\sigma\in \Hom_N(C_c^\infty(X)_{\BC(\sigma)},\psi_n)$ for every $\sigma\in \Temp(G')$.

\begin{lem}\label{lem1 abstract rc}
Let $J\in \Hom_N(C_c^\infty(X)_\pi,\psi_n)$. Then, $J$ extends by continuity to $\cC(X)$ and moreover there exists a function $F\in \cC^w(X)$ such that
\begin{align}\label{eq1 abstract rc}
\displaystyle J(\varphi)=\int_N^* \langle R(u)\varphi,F\rangle_X \psi_n(u)^{-1}du
\end{align}
for every $\varphi\in \cC(X)$.
\end{lem}

\begin{rem}
Note that by Proposition \eqref{prop HCS}(i) the above ``regularized" integral makes sense for every $\varphi\in \cC(X)$ and $F\in \cC^w(X)$.
\end{rem}

\begin{proof}
By Frobenius reciprocity and unicity of the Whittaker model, $J$ induces a $G$-equivariant linear map
$$\displaystyle W_J: C_c^\infty(X)\to \cW(\pi^\vee,\psi_n)$$
satisfying that $J(\varphi)=W_J(\varphi)(1)$ for every $\varphi\in C_c^\infty(X)$. Let $W_J^*: \cW(\pi^\vee,\psi_n)\to C^\infty(X)$ be the smooth adjoint of $W_J$ with respect to the invariant inner products $\langle .,.\rangle_X$ and $\langle .,.\rangle_{\Whitt}$. By \eqref{eq2 Fourier inversion}, we have
$$\displaystyle J(\varphi)\overline{w(1)}=W_J(\varphi)(1)\overline{w(1)}=\int_N^* \langle R(u) W_J(\varphi),w\rangle_{\Whitt} \psi_n(u)^{-1}du$$
for every $\varphi\in C_c^\infty(X)$ and $w\in \cW(\pi^\vee,\psi_n)$. Choose $w\in \cW(\pi^\vee,\psi_n)$ such that $w(1)=1$ and set $F=W_J^*(w)$. By Proposition \eqref{prop HCS}(iii), we have $F\in \cC^w(X)$. On the other hand, by adjunction we have $\langle R(u) W_J(\varphi),w\rangle_{\Whitt}=\langle R(u)\varphi, F\rangle_X$ for every $\varphi\in C_c^\infty(X)$ and $u\in N$. Therefore the function $F$ satisfies \eqref{eq1 abstract rc} for every $\varphi\in C_c^\infty(X)$. That $J$ extends continuously to $\cC(X)$ and \eqref{eq1 abstract rc} is still satisfied for $\varphi\in \cC(X)$ now follows from Proposition \eqref{prop HCS}(i).
\end{proof}

\section{Jacquet-Ye's local trace formula}\label{chapter JY TF}

In this chapter, we develop a local trace formula for the symmetric variety $X$. More precisely, we consider a {\em relative local Kuznetsov trace formula} for $X$ which is obtained by applying the $(N,\psi_n)$-regularized integral of Section \ref{section HCS} to a matrix coefficient for $L^2(X)$. The resulting `distribution' (a sesquilinear form on $C_c^\infty(X)$) admits both a geometric expansion, in terms of relative orbital integrals, and a spectral expansion, in terms of the FLO relative characters of Section \ref{section FLO funct}. The equality between the two expansions is the aforementioned local trace formula (Theorem \ref{theo local TF for X}). It will be applied in Chapters \ref{Chapter multiplicities} and \ref{chapter Plancherel formula} to finish the computation of multiplicities of generic representations with respect to $X$ and to the Plancherel decomposition of $X$ respectively. In Section \ref{section JY TF geom}, we define the relevant distribution on $C_c^\infty(X)$ and we establish a geometric expansion for it. In Section \ref{section JY TF spec}, we state and prove the spectral expansion and the resulting trace formula identity (Theorem \ref{theo local TF for X}).

We note here that a similar formula has been developed by Feigon \cite[Sect. 4]{Fe} in the context of the symmetric variety $X=\PGL_2(F)\backslash \PGL_2(E)$. One main difference between the two formulas is that the spectral side of Feigon's identity is given in terms of explicit invariant linear forms on tempered representations whereas the spectral side of Theorem \ref{theo local TF for X} is given in terms of the FLO functionals $J_\sigma$ (see the definition at the beginning of \S \ref{section JY TF spec}) which are in turn only defined implicitely through the Jacquet-Ye transfer (see \S \ref{section FLO funct}).

\subsection{Geometric expansion}\label{section JY TF geom}

Let $\varphi_1,\varphi_2\in C_c^\infty(X)$. By \eqref{eq2 X tempered}, we can define the following expression
\begin{equation}\label{def J}
\displaystyle J(\varphi_1,\varphi_2)=\int_N^* \langle R(u)\varphi_1,\varphi_2\rangle_X \psi_n(u)^{-1}du
\end{equation}
where the right-hand side is an $(N,\psi_n^{-1})$-regularized integral as defined in Section \ref{section HCS}.

%%Still by \eqref{eq1 HCS X}, we have:
%%\begin{num}
%%\item\label{eq 0 JY TF geom} The hermitian form $(\varphi_1,\varphi_2)\in \cC(X)^2\mapsto J(\varphi_1,\varphi_2)$ is separately continuous.
%%\end{num}

For $t\in T_X$ and $\varphi\in C_c^\infty(X)$ we set
$$\displaystyle O(t,\varphi)=\int_N \varphi(tu) \psi_n(u)^{-1}du.$$

\begin{lem}\label{lem1 JY TF geom}
The expression defining $O(t,\varphi)$ is absolutely convergent locally uniformly in $t$ and $\varphi$.
\end{lem}

\begin{proof}
This follows from the fact that the morphism $T_X\times N \to T_X\times X$, $(t,u)\mapsto (t,tu)$ is a closed embedding (hence proper). \end{proof}

Set
$$\displaystyle J_{\geom}(\varphi_1,\varphi_2)=\int_{T_X} O(t,\varphi_1)\overline{O(t,\varphi_2)}\delta_X(t)dt.$$
The main result of this section is the following.

\begin{theo}\label{theo JY TF geom}
The expression defining $J_{\geom}(\varphi_1,\varphi_2)$ is absolutely convergent and we have
$$\displaystyle J(\varphi_1,\varphi_2)=J_{\geom}(\varphi_1,\varphi_2).$$
\end{theo}

\begin{proof}
The proof is very similar to that of Theorem \ref{theo Kusgeom} so we will be brief and not give all the details. First we extend the definition of $J(\varphi_1,\varphi_2)$ to $\Phi\in C_c^\infty(X\times X)$ by
$$\displaystyle J(\Phi):=\int_N^* \int_X \Phi(xu,x)dx \psi_n(u)^{-1}du.$$
Note that this expression makes sense since we can show similarly to \eqref{eq2 X tempered} that the function
$$\displaystyle K_\Phi: g\in G\mapsto \int_X \Phi(xg,x)dx$$
belongs to $\cC^w(G)$. We have $J(\varphi_1,\varphi_2)=J(\varphi_1\otimes \overline{\varphi_2})$ where $\varphi_1\otimes \overline{\varphi_2}\in C_c^\infty(X\times X)$ is the function given by $(\varphi_1\otimes \overline{\varphi_2})(x_1,x_2)=\varphi_1(x_1)\overline{\varphi_2(x_2)}$.

Let $R^\Delta$ be the right diagonal action of $T$ on $C_c^\infty(X\times X)$. In the $p$-adic case, we choose a compact-open subgroup $K_T$ of $T$ by which both $\varphi_1$ and $\varphi_2$ are right-invariant and we set $\phi=\vol(K_T)^{-1}\mathbf{1}_{K_T}$, $\Phi=\varphi_1\otimes \overline{\varphi_2}$ so that $\varphi_1\otimes \overline{\varphi_2}=R^\Delta(\phi) \Phi$. In the Archimedean case, by Dixmier-Malliavin \cite{DM}, we may assume that $\varphi_1\otimes \overline{\varphi_2}=R^\Delta(\phi) \Phi$ for some $\phi\in C_c^\infty(T)$ and $\Phi\in C_c^\infty(X\times X)$. Then, by \eqref{eq 4 HCS}, in both cases we have
\begin{align}\label{eq 1bis JY TF geom}
\displaystyle J(\varphi_1,\varphi_2) & =\int_N^* K_{R^\Delta(\phi)\Phi}(u)\psi_n(u)^{-1}du=\int_N^* (\Ad(\phi)K_\Phi)(u)\psi_n(u)^{-1}du \\
\nonumber & =\int_N K_\Phi(u)\widehat{\phi}(u)du=\int_N\int_X \Phi(xu,x)dx \widehat{\phi}(u)du
\end{align}
where
$$\displaystyle \widehat{\phi}(u)=\int_T \phi(a) \psi_n(aua^{-1})^{-1}\delta_B(a)da,\;\;\; u\in N.$$
It follows readily from \eqref{eq2 X tempered} and \eqref{eq 5 HCS} that the last expression in \eqref{eq 1bis JY TF geom} is absolutely convergent. By \eqref{int formula X}, we have
\[\begin{aligned}
\displaystyle \int_N\int_X \Phi(xu,x)dx \widehat{\phi}(u)du & =\int_N\int_{T_X}\int_N \Phi(tvu,tv)dv \delta_X(t)dt \widehat{\phi}(u)du \\
 & =\int_{T_X} \int_{N^2} \Phi(tvu,tv) \widehat{\phi}(u)dudv\delta_X(t) dt.
\end{aligned}\]
Set
$$\displaystyle O(t,\Phi)=\int_{N^2} \Phi(tu,tv)\psi_n(u)^{-1}\psi_n(v)dudv$$
for every $\Phi\in C_c^\infty(X\times X)$ and $t\in T_X$. The same arguments as for Lemma \ref{lem1 JY TF geom} show that this expression is absolutely convergent locally uniformly in $t$ and $\Phi$. Note that $O(t,\varphi_1\otimes \overline{\varphi_2})=O(t,\varphi_1) \overline{O(t,\varphi_2)}$ for every $t\in T_X$. Simple manipulations (which are justified by the absolute convergence of $O(t,\Phi)$ uniformly in $t$ and $\Phi$) show that
\[\begin{aligned}
\displaystyle \int_{N^2} \Phi(tvu,tv) \widehat{\phi}(u)dudv=\int_T \phi(a)\delta_B(a)^{-1} O(ta^{-1},R^\Delta(a)\Phi)da
\end{aligned}\]
for every $t\in T_X$. Thus, the above computations imply that the expression
\begin{align}\label{eq 2 JY TF geom}
\displaystyle \int_{T_X}\int_T \phi(a)\delta_B(a)^{-1} O(ta^{-1},R^\Delta(a)\Phi)da\delta_X(t)dt
\end{align}
is convergent as an iterated integral for every $\phi\in C_c^\infty(T)$ and $\Phi\in C_c^\infty(X\times X)$ and moreover that
\begin{align}\label{eq 3 JY TF geom}
\displaystyle J(\varphi_1,\varphi_2)=\int_{T_X} \int_T \phi(a)\delta_B(a)^{-1} O(ta^{-1},R^\Delta(a)\Phi)da\delta_X(t)dt
\end{align}
whenever $\varphi_1\otimes \overline{\varphi_2}=R^\Delta(\phi)\Phi$. The argument at the end of the proof of Theorem \ref{theo Kusgeom}, in particular for the claim \eqref{eq 6 Kusgeom}, adapts almost verbatim to this situation to show that \eqref{eq 2 JY TF geom} is actually absolutely convergent. (Here, we recall that, in the Archimedean case for any compact subset $L\subset X$, denoting by $C^\infty_L(X)$ the subspace of smooth functions supported in $L$, we have $C^\infty_L(X)\widehat{\otimes}C^\infty_L(X)\simeq C_{L\times L}^\infty(X\times X)$ \cite[Exemple 1, Chap. II \S 3 n.3]{Gro}). Using \eqref{eq 3 JY TF geom}, simple manipulations now allow to get the identity
$$\displaystyle J(\varphi_1,\varphi_2)=J_{\geom}(\varphi_1,\varphi_2)$$
and the fact that the expression defining $J_{\geom}(\varphi_1,\varphi_2)$ is absolutely convergent.
\end{proof}

%%\begin{rem}
%%The above theorem implies in particular that $J$ is a positive semi-definite hermitian form on $\cC(X)$. As separately continuous bilinear forms on a Fr\'echet space are automatically continuous and $\cC(X)$ is either a LF (in the $p$-adic case) or Fr\'echet (in the Archimedean case) space, by \eqref{eq 0 JY TF geom} this shows that $\varphi\mapsto J(\varphi,\varphi)^{1/2}$ is a continuous semi-norm on $\cC(X)$. 
%%\end{rem}

\subsection{Spectral expansion}\label{section JY TF spec}

Recall from Section \ref{section FLO funct} that to every $\sigma\in \Temp(G')$ is associated a relative character $J_\sigma$ which is a functional on $C_c^\infty(X)$. For every $\varphi_1,\varphi_2\in C_c^\infty(X)$ we set
\begin{equation}\label{def Jspec}
\displaystyle J_{\spec}(\varphi_1,\varphi_2)=\int_{\Temp(G')} J_\sigma(\varphi_1)\overline{J_\sigma(\varphi_2)} d\mu_{G'}(\sigma)
\end{equation}
where $\mu_{G'}$ denotes the Plancherel measure of $G'$ (see Section \ref{section Kusspec}).

\begin{theo}\label{theo JY TF spec}
For every $\varphi_1,\varphi_2\in C_c^\infty(X)$, the expression defining $J_{\spec}(\varphi_1,\varphi_2)$ is absolutely convergent and we have
$$\displaystyle J(\varphi_1,\varphi_2)=J_{\spec}(\varphi_1,\varphi_2).$$
\end{theo}

\begin{proof}
Let $f_1,f_2\in C_c^\infty(G')$ be test functions matching $\varphi_1$, $\varphi_2$ respectively in the sense of Section \ref{section JY transfer}. By Theorem \ref{theo JY TF geom}, the definition of the transfer, and the fact that the isomorphism $T_X\simeq T'$ is measure preserving, we have
$$\displaystyle J(\varphi_1,\varphi_2)=\int_{T_X} O(t,\varphi_1)\overline{O(t,\varphi_2)} \delta_X(t)dt=\int_{T'} O(a,f_1) \overline{O(a,f_2)} \delta_{B'}(a) da$$
where the ``transfer factors'' disappear as $\gamma(a)^2=1$. By Theorem \ref{theo local Kuznetsov}, this last expression is equal to
\begin{align}\label{eq1 explicit spec exp}
\displaystyle \int_{\Temp(G')}I_\sigma(f_1)\overline{I_\sigma(f_2)} d\mu_{G'}(\sigma).
\end{align}
By definition of the FLO relative characters $J_\sigma$, this is further equal to
$$\displaystyle \int_{\Temp(G')}J_\sigma(\varphi_1)\overline{J_\sigma(\varphi_2)} d\mu_{G'}(\sigma)=J_{\spec}(\varphi_1,\varphi_2).$$
Moreover, as \eqref{eq1 explicit spec exp} is absolutely convergent (by Theorem \ref{theo Kusgeom}), the above expression is also convergent and this proves the theorem.
\end{proof}

From Theorem \ref{theo JY TF spec} and Theorem \ref{theo JY TF geom}, we deduce:

\begin{theo}[Local Kuznetsov trace formula for $X$]\label{theo local TF for X}
For every $\varphi_1,\varphi_2\in C_c^\infty(X)$, we have
$$\displaystyle J_{\geom}(\varphi_1,\varphi_2)=J_{\spec}(\varphi_1,\varphi_2).$$
\end{theo}

Assume now that $F$ is a $p$-adic field. By Proposition \ref{prop HCS}(i), the definition \eqref{def J} of $J(\varphi_1,\varphi_2)$ extends to any $\varphi_1,\varphi_2\in \cC(X)$ and moreover, $J$ is a separately continuous Hermitian form on $\cC(X)$.  On the other hand, by Lemma \ref{lem1 abstract rc} the FLO relative characters $J_\sigma$, $\sigma\in \Temp(G')$, extend by continuity to $\cC(X)$. Hence, the definition \eqref{def Jspec} of $J_{\spec}(\varphi_1,\varphi_2)$ still makes sense, formally, for every $\varphi_1,\varphi_2\in \cC(X)$. In this context, Theorem \ref{theo JY TF spec} admits the following extension.

\begin{theo}\label{theo JY TF spec2}
For every $\varphi_1,\varphi_2\in \cC(X)$, the expression defining $J_{\spec}(\varphi_1,\varphi_2)$ is absolutely convergent and we have
$$\displaystyle J(\varphi_1,\varphi_2)=J_{\spec}(\varphi_1,\varphi_2).$$
\end{theo}

\begin{proof}
Let $J\subset G$ be a compact-open subgroup and $\varphi\in \cC(X)^J$. Let $(\varphi_k)_{k\geqslant 1}$ a sequence in $C_c^\infty(X)^J$ converging to $\varphi$ in $\cC(X)^J$ (such sequence exists by \eqref{density HCS}). Since separately continuous bilinear forms on FR\'echet spaces are automatically continuous \cite[Corollary 34.2]{Tr}, by Theorem \ref{theo JY TF spec} and the continuity of $J$ we deduce that the sequence
$$\displaystyle J(\varphi_k,\varphi_k)=\int_{\Temp(G')} \lvert J_\sigma(\varphi_k)\rvert^2d\mu_{G'}(\sigma)$$
converges to $J(\varphi,\varphi)$. Hence, by Fatou's lemma and the continuity of $J_\sigma$ on $\cC(X)$, the integral
$$\displaystyle \int_{\Temp(G')} \lvert J_\sigma(\varphi)\rvert^2d\mu_{G'}(\sigma)$$
converges and is bounded by $J(\varphi,\varphi)$. By Cauchy-Schwarz, it follows that $J_{\spec}(\varphi_1,\varphi_2)$ is absolutely convergent and defines a continuous sesquilinear form on $\cC(X)^J$. The theorem follows by the continuity of $J$ and the density of $C_c^\infty(X)$ in $\cC(X)$.
\end{proof}

\section{Multiplicities}\label{Chapter multiplicities}

In this chapter, we keep the notation introduced in the Chapters \ref{chapter invt funct} and \ref{chapter JY TF} and we moreover assume that:
$$F \mbox{ is a } p-\mbox{adic field.}$$
The goal of this chapter is to complement results of Feigon-Lapid-Offen on the computations of the {\em multiplicity}
$$\displaystyle m(\pi)=\dim \Hom_G(\pi,C^\infty(X))$$
for $\pi\in \Irr(G)$ {\em generic}. This multiplicity is always finite by a general result of Delorme \cite[Theorem 4.5]{Del} and naturally decomposes as a sum over $V\in \cV$ of individual multiplicities
$$\displaystyle m_V(\pi)=\dim \Hom_G(\pi, C^\infty(U(V)\backslash G))=\dim \Hom_{U(V)}(\pi,\mathbb{C})$$
where the last equality follows from Frobenius reciprocity.

In \cite[Theorem 0.2]{FLO}, Feigon, Lapid and Offen gives a lower bound for $m_V(\pi)$ in terms of the (cardinality of the) general fibers of Arthur and Clozel's base-change map $\BC:\Irr(G')\to \Irr(G)$ \cite{AC}. They moreover show that this lower bound is actually equal to the multiplicity when $\BC$ is ``unramified at $\pi$'' (in a sense that will be made precise in the next section). The new result obtained here is that equality {\em always} holds as conjectured by Feigon-Lapid-Offen \cite[Conjecture 13.17]{FLO}. The main ingredients entering into the proof are the local trace formula for $X$ developed in the last chapter as well as the scalar Whittaker-Paley-Wiener theorem of Section \ref{section PW theorem} for the group $G'$. 

In order to state the main result in the appropriate context, in Section \ref{section structure of alg var} we explain how to endow $\Irr(G)$ and $\Irr(G')$ with natural structures of algebraic varieties and we study related properties of the base-change map $\BC$ and the map $\lambda$ associating to an irreducible representation its cuspidal support. Using these extra structures, we state in Section \ref{section result multiplicities} the main result whose proof occupies Sections \ref{section first step multiplicities} to \ref{section end of proof multiplicities}. More precisely, in Section \ref{section first step multiplicities}, we make a reduction to tempered representations following \cite[\S 6]{FLO}. In Section \ref{section second step multiplicities}, we relate the multiplicity $m(\pi)$ to the FLO functionals of Section \ref{section FLO funct} via the local trace formula developed in the previous chapter. Once this relation is establish, the theorem readily follows from the scalar Whittaker Paley-Wiener theorem and the necessary arguments are given in Section \ref{section end of proof multiplicities}.

Here is a list of notation and conventions that we shall use in this chapter (besides the one introduced in previous sections):
\begin{itemize}
\item A {\em semi-standard Levi} of $G$ (resp. $G'$) means a Levi subgroup containing $T$ (resp. $T'$). Similarly, a {\em standard parabolic subgroup} of $G$ (resp. $G'$) is a parabolic subgroup containing $B$ (resp. $B'$) and a standard Levi subgroup is the unique semi-standard Levi component of a standard parabolic subgroup.

\item For $M$ a Levi subgroup of $G$ or $G'$, we denote by $X(M)$, $X_{\unit}(M)$, $X_{\unr}(M)$ and $X_{\alg}(M)$ the groups of smooth, unitary, unramified and algebraic (defined over $F$) characters of $M$ respectively. Recall that $X_{\unr}(M)$ is a complex torus whose index in $X(M)$ is countable. Therefore, $X(M)$ has a natural structure of algebraic variety over $\bC$ (with countably many components). We set $\cA_M^*=X_{\alg}(M)\otimes \bR$. There is an injective homomorphism $\cA_M^*\to X(M)$ sending $\lambda\otimes x$ to the character $m\in M\mapsto \lvert \lambda(m)\rvert_F^x$. The image of this homomorphism is the subgroup of positive valued characters of $M$. Therefore, if $\chi\in X(M)$, its absolute value $\lvert \chi\rvert$ corresponds to an element of $\cA_M^*$ that we denote by $\Re(\chi)$. More generally, if $\sigma$ is an irreducible smooth representation of $M$ with central character $\omega_\sigma$, $\lvert \omega_\sigma\rvert$ extends uniquely to a positive valued character of $M$ and we set $\Re(\sigma)=\Re(\lvert \omega_\sigma\rvert)$.

\item  If $L\subset M$ is another Levi subgroup, there is a natural inclusion $\cA_M^*\subset \cA_L^*$ with a natural section $\cA_L^*\twoheadrightarrow \cA_M^*$ whose kernel we denote by $(\cA_L^M)^*$. The inclusion $T'\subset T$ induces an identification $\cA_{T'}^*=\cA_T^*$ and we just write $\cA^*$ for this real vector space.

\item Still for $M$ a Levi subgroup of $G$ (resp. of $G'$), we set $W(G,M)=\Norm_G(M)/M$ (resp. $W(G',M)=\Norm_{G'}(M)/M$) for the corresponding Weyl group and $W^M=W(M,T)$ (resp. $W^{M}=W(M,T')$) for the Weyl group of $T$ (resp. $T'$) in $M$. Then, $W(G,M)$ acts naturally on $\cA_M^*$. We have again a natural identification $W^{G'}=W^G$ and we simply write $W$ for this Weyl group. We fix on $\cA^*$ a $W$-invariant Euclidean norm $\lVert .\rVert$. Note that for every pair $L\subset M$ of semi-standard Levi subgroups, the subspaces $\cA_L^*$ and $(\cA_L^M)^*$ are orthogonal for the resulting Euclidean structure.

\item We denote by $\Irr(G)$ (resp. $\Irr(G')$) the set of isomorphism classes of smooth irreducible representations of $G$ (resp. $G'$) and by $\Irr^{\gen}(G)$, $\Temp(G)$, $\Pi_2(G)$, $\Pi_{2,\ess}(G)$, $\Pi_{\cusp}(G)$ (resp. $\Irr^{\gen}(G')$, $\Temp(G')$, $\Pi_2(G')$, $\Pi_{2,\ess}(G')$, $\Pi_{\cusp}(G')$) the subsets of generic, tempered, square-integrabl, essentially square-integrable and supercuspidal irreducible representations respectively.

\item If $P=MU$ is a parabolic subgroup of $G$ and $\tau$ a smooth representation of $M$, we denote by $I_P^G(\tau)$ the smooth unitarily normalized parabolic induction of $\tau$. If moreover $P$ is standard and $M$ decomposes in diagonal blocks as
$$\displaystyle M=\GL_{n_1}(E)\times \ldots \times \GL_{n_k}(E)$$
and $\tau$ is of the form $\tau=\tau_1\boxtimes\ldots\boxtimes \tau_k$ and we write
$$\displaystyle \tau_1\times\ldots\times \tau_k$$
for $I_P^G(\tau)$. Similar notation apply to representations of $G'$.

\end{itemize}

\subsection{Algebraic structure on $\Irr(G)$, the Bernstein center and base-change}\label{section structure of alg var}

Let $\Sqr(G)$ be the set of pairs $(M,\sigma)$ where $M$ is a semi-standard Levi of $G$ and $\sigma\in \Pi_{2,\ess}(M)$ is an irreducible essentially square-integrable representation of $M$. We equip $\Sqr(G)$ with its unique structure of algebraic variety over $\bC$ (with infinitely many components) such that for every $(M,\sigma)\in \Sqr(G)$, the map
$$\displaystyle X_{\unr}(M)\to \Sqr(G),\; \chi\mapsto (M,\sigma\otimes \chi)$$
is a finite covering over a connected component of $\Sqr(G)$. The Weyl group $W$ is acting on $\Sqr(G)$ by regular automorphisms and we denote by $\Sqr(G)/W$ the GIT quotient. By the special form of the Levi subgroups of $G$ and their associated Weyl groups, the connected components of $\Sqr(G)/W$ are all isomorphic to products of varieties of the form $(\bC^\times)^t/\mathfrak{S}_t$ where $\mathfrak{S}_t$ acts on $(\bC^\times)^t$ by permutation of the entries. This implies in particular that $\Sqr(G)/W$ is smooth.

To $(M,\sigma)\in \Sqr(G)$ we associate the unique irreducible quotient of $I_P^G(\sigma)$ where $P$ is any parabolic subgroup with Levi component $M$ such that $\Re(\sigma)$ is (non-strictly) dominant with respect to $P$. By the Langlands classification this induces a bijection $\Sqr(G)/W\simeq \Irr(G)$ and we use this bijection to transfer the structure of algebraic variety on $\Sqr(G)/W$ to $\Irr(G)$.

We will use this bijection to identify $\Sqr(G)/W$ and $\Irr(G)$, thus for $(M,\sigma)\in \Sqr(G)$ its image $[M,\sigma]\in \Sqr(G)/W$ is identified with the corresponding Langlands quotient in $\Irr(G)$. Also, for $(M,\sigma)\in \Sqr(G)$ we will write $\Irr_{M,\sigma}(G)$ for the image in $\Irr(G)$ of the subset
$$\displaystyle \left\{(M,\sigma\otimes \chi)\mid \chi\in X(M) \right\}$$
of $\Sqr(G)$. Setting
$$\displaystyle W'_\sigma=\{(\chi,w)\in X(M)\rtimes W(G,M)\mid w\sigma\simeq \sigma\otimes \chi  \}$$
(a finite group) the map $\chi\in X(M)\mapsto [M,\sigma\otimes \chi]$ induces a regular isomorphism $X(M)/W'_\sigma\simeq \Irr_{M,\sigma}(G)$. We emphasize here that, as $X(M)$ stands for the group of all {\em smooth} characters of $M$ (not necessarily unramified), $\Irr_{M,\sigma}(G)$ is only a countable union of connected components of $\Irr(G)$.

For $(M,\sigma)\in \Sqr(G)$, we also set
$$\displaystyle \Temp_{M,\sigma}(G)=\Irr_{M,\sigma}(G)\cap \Temp(G) \mbox{ and } \Irr^{\gen}_{M,\sigma}(G)=\Irr_{M,\sigma}(G)\cap \Irr^{\gen}(G).$$
Assuming that $\sigma$ is square-integrable (which we may up to a twist), $\Temp_{M,\sigma}(G)$ is the image of $X_{\unit}(M)$ by the surjective regular map $X(M)\to \Irr_{M,\sigma}(G)$, $\chi\mapsto [M,\sigma\otimes \chi]$. Since $X_{\unit}(M)$ is Zariski dense in $X(M)$ this shows:
\begin{equation}\label{density tempered}
\displaystyle \Temp(G) \mbox{ is Zariski-dense in } \Irr(G).
\end{equation}

Let $\pi=[M,\sigma]\in \Irr^{\gen}(G)$. Then, for every parabolic subgroup $P$ with Levi component $M$ we have $\pi\simeq I_P^G(\sigma)$ \cite[Theorem 9.7]{Ze}. Conversely, if $[M,\sigma]\in \Sqr(G)/W$ is such that for one parabolic subgroup $P$ with Levi component $M$, $I_P^G(\sigma)$ is irreducible then its image in $\Irr(G)$ is generic. Therefore, by \cite[Proposition VI.8.4]{Ren} we have
\begin{equation}\label{generic Zariski open}
\Irr^{\gen}(G)\mbox{ is Zariski open in }\Irr(G).
\end{equation}

Let $\cZ(G)$ be the ``finite'' Bernstein center (as defined in Section \ref{section PW theorem}) and let $\cB(G)$ be its maximal spectrum  which is an algebraic variety over $\bC$. Then, we have an identification $\cB(G)\simeq \Cusp(G)/W$ of algebraic varieties where $\Cusp(G)$ is the set of pairs $(L,\tau)$ with $L$ a semi-standard Levi subgroup and $\tau\in \Pi_{\cusp}(L)$ (the isomorphism class of) an irreducible supercuspidal representation of $L$ that we endow with a structure of algebraic variety the same way we did for $\Sqr(G)$. For $(L,\tau)\in \Cusp(G)$, we denote by $\cB_{L,\tau}(G)$ the subset
$$\displaystyle \left\{[L,\tau\otimes \chi]\mid \chi\in X(L) \right\}$$
of $\cB(G)$. As before, $\cB_{L,\tau}(G)$ is an union of connected component and the map $X(L)\to \cB_{L,\tau}(G)$, $\chi\mapsto [L,\tau\otimes \chi]$ induces an isomorphism $X(L)/W'_\tau\simeq \cB_{L,\tau}(G)$.

The natural inclusion $\Cusp(G)\subset \Sqr(G)$ descends to an open-closed immersion $\cB(G)\hookrightarrow \Irr(G)$ and in particular $\cB(G)$ is also smooth.
%Then the subsets $\Cusp_{L,\tau}(G)/W_\tau$ form an open partition of $\Cusp(G)/W$ and the bijections $X(L)/W'_\tau\simeq \Cusp_{L,\tau}(G)/W_\tau$, $[\chi]\mapsto [L,\tau\otimes \chi]$, are regular isomorphisms. Note that $\Cusp(G)$ and $\Cusp(G)/W$ are actually subsets of $\Sqr(G)$ and $\Sqr(G)/W$ respectively and that for $(L,\tau)\in \Cusp(G)$ we have $\Cusp_{L,\tau}(G)=\Sqr_{L,\tau}(G)$ and $\Cusp_{L,\tau}(G)/W_\tau\simeq \Sqr_{L,\tau}(G)/W_\tau$ (as algebraic varieties). In particular $\cB(G)=\Cusp(G)/W$ is also smooth and is an union of connected components of $\Irr(G)$.
This embedding admits a left-inverse
$$\displaystyle\lambda: \Irr(G)\to \cB(G)$$
which associates to $\pi\in \Irr(G)$ its supercuspidal support (i.e. the unique element $[L,\tau]\in \cB(G)$ such that $\pi$ is a subquotient of $I_Q^G(\tau)$ for one, or equivalently every, parabolic with Levi component $L$).

\begin{lem}\label{lem finite morphism}
$\lambda$ is a regular finite morphism.
\end{lem}

\begin{proof}
Let $(M,\sigma)\in \Sqr(G)$. It suffices to show that the restriction of $\lambda$ to $\Irr_{M,\sigma}(G)$ is regular and finite. Choose $(L,\tau)\in \Cusp(G)$ in the cuspidal support of $\sigma$. We have a commutative diagram
$$\displaystyle \xymatrix{X(M) \ar[r]^{\Res} \ar[d] & X(L) \ar[d] \\ X(M)/W'_\sigma\simeq \Irr_{M,\sigma}(G) \ar[r]^{\lambda} & \cB_{L,\tau}(G)\simeq X(L)/W'_\tau}$$
where the two vertical maps are $\chi\mapsto [M,\sigma\otimes \chi]$ and $\chi\mapsto [L,\tau\otimes \chi]$ respectively. Moreover, the restriction map $\Res: X(M)\to X(L)$ is a closed immersion and in particular finite. By the universal property of GIT quotients, the bottom map is therefore regular and finite.
\end{proof}

Let $(M,\sigma)\in\Sqr(G)$. We denote by $\Irr_{M,\sigma}(G)^\lambda$ and $\Irr^{\gen}_{M,\sigma}(G)^\lambda$ the respective images of $\Irr_{M,\sigma}(G)$ and $\Irr^{\gen}_{M,\sigma}(G)$ by $\lambda$. By the previous lemma, $\Irr_{M,\sigma}(G)^\lambda$ is closed in $\cB(G)$.

\begin{prop}\label{prop local isom lambda}
$\Irr^{\gen}_{M,\sigma}(G)^\lambda$ is open in $\Irr_{M,\sigma}(G)^\lambda$ and $\lambda: \Irr_{M,\sigma}(G)\to \Irr_{M,\sigma}(G)^\lambda$ restricts to an isomorphism over $\Irr^{\gen}_{M,\sigma}(G)^\lambda$.
\end{prop}

\begin{proof}
Without loss in generality, we may assume that $\sigma\in \Pi_2(M)$. First we prove
\begin{equation}\label{eq1 immersion}
\mbox{For }\pi\in \Irr^{\gen}_{M,\sigma}(G)\mbox{ and }\pi'\in \Irr_{M,\sigma}(G)\mbox{ if }\lambda(\pi)=\lambda(\pi')\mbox{ then }\pi=\pi'.
\end{equation}

Indeed, let $\pi\in \Irr^{\gen}_{M,\sigma}(G)$ and $\pi'\in \Irr_{M,\sigma}(G)$ and assume that $\lambda(\pi)=\lambda(\pi')$. There exist $\chi,\chi'\in X(M)$ and a parabolic subgroup $P$ with Levi component $M$ such that $\pi=I_P^G(\sigma\otimes \chi)$ and $\pi'$ is the Langlands quotient of $I_P^G(\sigma\otimes \chi')$. Since $\sigma$ is generic, $I_P^G(\sigma\otimes \chi')$ admits an irreducible generic subquotient \cite[Th\'eor\`eme 4]{Rod} with the same cuspidal support as $\pi'$. As there is an unique irreducible generic representation with a given cuspidal support, this shows that $\pi=I_P^G(\sigma\otimes \chi)$ is a subquotient of $I_P^G(\sigma\otimes \chi')$. Moreover, it follows from the geometric lemma of Bernstein-Zelevinsky and Casselman (see \cite[Geometric Lemma]{BZ} and \cite[\S 6.3]{Cas1}) that for every parabolic subgroup $Q\subset G$ the length of the supercuspidal parts of the Jacquet modules $J_QI_P^G(\sigma\otimes \chi)$ and $J_QI_P^G(\sigma\otimes \chi')$ are the same. By exactness of the Jacquet functor $J_Q$, this shows that if $\pi'\neq \pi$ then the supercuspidal part of the Jacquet module $J_Q\pi'$ is zero for every parabolic subgroup $Q$ but this is impossible by \cite[lemme VI.7.2 (iii)]{Ren}. Therefore $\pi=\pi'$.

%By \cite[Proposition 4.13]{BW}, this further implies that $\Re(\chi)$ belongs to the convex hull of the $W$-orbit $W\Re(\chi')$ of $\Re(\chi')$ in $\cA^*$ and moreover that $\Re(\chi)\in W\Re(\chi')$ if and only if $\pi=\pi'$. Let $(L,\tau)\in \Cusp(G)$ be in the cuspidal support of $\sigma$. Then, $(L,\tau\otimes \chi)$ and $(L,\tau\otimes \chi')$ are in the cuspidal supports of $\pi$ and $\pi'$ respectively hence are conjugate under $W(G,L)$. This shows in particular that $\Re(\tau\otimes \chi)=\Re(\tau)+\Re(\chi)$ and $\Re(\tau\otimes \chi')=\Re(\tau)+\Re(\chi')$ have the same norm. Since $\Re(\tau)\in (\cA_L^M)^*$ (as $\sigma$ is square-integrable) and $\Re(\chi),\Re(\chi')\in \cA_M^*$, it follows that $\lVert \Re(\chi)\rVert=\lVert \Re(\chi')\rVert$. Together with the fact that $\Re(\chi)$ belongs to the convex hull of $W\Re(\chi')$ this implies $\Re(\chi)\in W\Re(\chi')$ hence $\pi=\pi'$.

We now prove the proposition. As finite morphisms are closed, by \eqref{generic Zariski open}, Lemma \ref{lem finite morphism} and \eqref{eq1 immersion}, we see that $\Irr^{\gen}_{M,\sigma}(G)^\lambda$ is open in $\Irr_{M,\sigma}(G)^\lambda$ and moreover the restriction of $\lambda$ to $\Irr^{\gen}_{M,\sigma}(G)^\lambda$ is a finite bijective map $\Irr^{\gen}_{M,\sigma}(G)\to \Irr^{\gen}_{M,\sigma}(G)^\lambda$. Therefore, by \cite[\href{https://stacks.math.columbia.edu/tag/04XV}{Tag 04XV}]{stacks-project}, it only remains to check that $\lambda$ is unramified on $\Irr^{\gen}_{M,\sigma}(G)$.

Let $\chi_0\in X(M)$ be such that $[M,\sigma\otimes \chi_0]\in \Irr^{\gen}_{M,\sigma}(G)$ and $(L,\tau)\in \Cusp(G)$ be in the cuspidal support of $\sigma_0=\sigma\otimes \chi_0$. Let $W^0_{\sigma_0}\subset W(G,M)$ and $W^0_\tau\subset W(G,L)$ be the stabilizers of $\sigma_0$ and $\tau$ respectively. We have:
\begin{num}
\item\label{eq1 restriction} The restriction map $\Res:X_{\unr}(M)\to X_{\unr}(L)$ descends to a regular morphism
$$\displaystyle X_{\unr}(M)/W^0_{\sigma_0}\to X_{\unr}(L)/W^0_\tau.$$
\end{num}
Indeed, let $w\in W^0_{\sigma_0}$ and take any lift $\widetilde{w}\in \Norm_G(M)$. The pair $(\widetilde{w}L\widetilde{w}^{-1},\widetilde{w}\tau)$ is also in the cuspidal support of $\sigma_0$ and so, up to multiplying $\widetilde{w}$ by an element of $M$ we have $\widetilde{w}\in \Norm_G(L)$ and $\widetilde{w}\tau\simeq \tau$. Then, denoting by $w'$ the image of $\widetilde{w}$ in $W^0_\tau$, we have $\Res(w\chi)=w'\Res(\chi)$ for every $\chi\in X_{\unr}(M)$ and \eqref{eq1 restriction} follows.

%It is easy to see that the image of such a lift $\widetilde{w}$ in $W_\tau^0$ only depends on $w$ and that this gives a natural embedding $W^0_{\sigma_0}\subset W^0_\tau$ which is compatible with the actions of $W^0_{\sigma_0}$, $W^0_\tau$ on $X_{\unr}(M)$, $X_{\unr}(L)$ through the restriction morphism $X_{\unr}(M)\to X_{\unr}(L)$. Therefore, we get a regular map $X_{\unr}(M)/W^0_{\sigma_0}\to X_{\unr}(L)/W^0_\tau$. 

The maps $\chi\in X_{\unr}(M)\mapsto [M,\sigma_0\otimes \chi]\in \Irr(G)$ and $\chi\in X_{\unr}(L)\mapsto [L,\tau\otimes \chi]\in \cB(G)$ descend to regular morphisms $X_{\unr}(M)/W^0_{\sigma_0}\to \Irr(G)$ and $X_{\unr}(L)/W^0_\tau\to \cB(G)$ which are local isomorphisms near $1$ and such that the following diagram commutes
$$\displaystyle \xymatrix{ X_{\unr}(M)/W^0_{\sigma_0} \ar[r] \ar[d] & X_{\unr}(L)/W^0_\tau \ar[d] \\ \Irr(G) \ar[r]^{\lambda} & \cB(G).}$$
Consequently, it only remains to prove that $X_{\unr}(M)/W^0_{\sigma_0}\to X_{\unr}(L)/W^0_\tau$ is unramified at $1$. Actually, we are going to show that this map is a closed immersion.

We may decompose $M$ as
$$\displaystyle M=\GL_{n_1}(E)\times \ldots\times \GL_{n_k}(E),$$
where $n_1,\ldots,n_k$ are positive integers such that $n_1+\ldots+n_k=n$, and we may accordingly decompose $\sigma_0$ as a tensor product
$$\displaystyle \sigma_0=\nu_1\boxtimes \ldots\boxtimes \nu_k$$
where, for each $1\leqslant i\leqslant k$, $\nu_i$ is an essentially square-integrable representation of $\GL_{n_i}(E)$. Let $\Sigma$ be the set of all isomorphism classes among $\nu_1,\ldots,\nu_k$ and for each $\nu\in \Sigma$ set
$$\displaystyle m(\nu)=\left\lvert \{1\leqslant i\leqslant k\mid \nu\simeq\nu_i \}\right\rvert.$$
Regrouping the $\nu_i$'s according to their isomorphism classes, we get an isomorphism $X_{\unr}(M)\simeq \prod_{\nu\in \Sigma} (\bC^\times)^{m(\nu)}$ which descends to an isomorphism
\begin{equation}\label{isom1}
\displaystyle X_{\unr}(M)/W_{\sigma_0}^0\simeq \prod_{\nu\in \Sigma} (\bC^\times)^{m(\nu)}/\mathfrak{S}_{m(\nu)}.
\end{equation}

According to the classification by Bernstein and Zelevinsky of the essentially square-integrable representations of general linear groups \cite[Theorem 9.3]{Ze}, for each $\nu \in \Sigma$ there is a {\em segment} $\Delta_{\nu}$, that is a set of the form $\Delta_{\nu}=\{ \rho_{\nu}\lvert \det\rvert_E^{a_{\nu}}, \rho_{\nu}\lvert \det\rvert_E^{a_{\nu}+1},\ldots, \rho_{\nu}\lvert \det\rvert_E^{b_{\nu}}\}$ where $\rho_{\nu}$ is (the isomorphism class of) a supercuspidal representation of some $\GL_{d_{\nu}}(E)$ and $a_{\nu}$, $b_{\nu}$ are real numbers with $b_{\nu}-a_{\nu}\in \mathbb{N}$, such that $\nu$ is isomorphic to the unique irreducible quotient of
$$\displaystyle \rho_{\nu}\lvert \det\rvert_E^{a_{\nu}}\times \rho_{\nu}\lvert \det\rvert_E^{a_{\nu}+1}\times \ldots \times \rho_{\nu}\lvert \det\rvert_E^{b_{\nu}}.$$
Set $T=\bigcup_{\nu\in \Sigma} \Delta_\nu$ and for each $\rho\in T$ let
$$\displaystyle \ell(\rho)=\sum_{\nu\in \Sigma; \rho\in \Delta_\nu} m(\nu).$$
Then, up to the ordering, $\tau$ is isomorphic to $\bigboxtimes_{\rho\in T} \rho^{\boxtimes \ell(\rho)}$. Therefore, there is an isomorphism $X_{\unr}(L)\simeq \prod_{\rho\in T} (\bC^\times)^{\ell(\rho)}$ that descends to an isomorphism
$$\displaystyle X_{\unr}(L)/W^0_\tau\simeq \prod_{\rho\in T} (\bC^\times)^{\ell(\rho)}/\mathfrak{S}_{\ell(\rho)}$$
such that combined with the isomorphism \eqref{isom1}, the map $X_{\unr}(M)/W^0_{\sigma_0}\to X_{\unr}(L)/W^0_\tau$ becomes
\begin{equation}\label{immers1}
\displaystyle \prod_{\nu\in \Sigma} (\bC^\times)^{m(\nu)}/\mathfrak{S}_{m(\nu)}\to \prod_{\rho\in T} (\bC^\times)^{\ell(\rho)}/\mathfrak{S}_{\ell(\rho)},
\end{equation}
$$(z_{\nu})_{\nu\in \Sigma}\mapsto (\bigtimes_{\nu\in \Sigma; \rho\in \Delta_{\nu}} z_\nu)_{\rho\in T}$$
where $\bigtimes_{\nu\in \Sigma; \rho\in \Delta_{\nu}} z_\nu$ denotes the ``concatenation'' of the $z_\nu$ with $\rho\in \Delta_\nu$ (whose image in $(\bC^\times)^{m(\nu)}/\mathfrak{S}_{m(\nu)}$ does not depend on the ordering).

%If we can show that
%\begin{num}
%\item\label{segment1} For each $\nu\in \Sigma$, there exists $\rho_\nu\in T$ such that for every $\nu'\in \Sigma$, $\rho_\nu\in \Delta_{\nu'}$ implies $\nu=\nu'$;
%\end{num}
%then \eqref{immers1} would admit the left inverse $(z_{\rho})_{\rho\in T}\mapsto (z_{\rho_\nu})_{\nu\in \Sigma}$ and would therefore be a closed immersion.

Therefore, it only remains to show that \eqref{immers1} is a closed immersion. By Zelevinsky's classification of generic representations of $\GL_n(E)$ \cite[Theorem 9.7]{Ze}, for every $\nu, \nu'\in \Sigma$, if $\Delta_{\nu}\cup \Delta_{\nu'}$ is again a segment then $\Delta_\nu\subseteq \Delta_{\nu'}$ or $\Delta_{\nu'}\subseteq \Delta_\nu$. In particular, it follows that for $\nu\in \Sigma$ the union
$$\displaystyle \bigcup_{\nu'\in \Sigma; \Delta_{\nu'}\subsetneq \Delta_\nu} \Delta_{\nu'}$$
is strictly smaller than $\Delta_\nu$. Let $\rho_\nu\in \Delta_\nu$ be in the complement of this subset. Then, for every $\nu,\nu'\in \Sigma$, $\rho_\nu\in \Delta_{\nu'}$ implies $\Delta_\nu\subseteq \Delta_{\nu'}$. Moreover, for each $\nu\in \Sigma$ the map
$$\displaystyle \prod_{\nu'\in \Sigma; \Delta_\nu\subseteq \Delta_{\nu'}} (\bC^\times)^{m(\nu')}/\mathfrak{S}_{m(\nu')}\to \prod_{\nu'\in \Sigma; \Delta_\nu\subsetneq \Delta_{\nu'}} (\bC^\times)^{m(\nu')}/\mathfrak{S}_{m(\nu')} \times (\bC^\times)^{\ell(\rho_\nu)}/\mathfrak{S}_{\ell(\rho_\nu)},$$
$$\displaystyle (z_{\nu'})_{\Delta_\nu\subseteq \Delta_{\nu'}}\mapsto \left((z_{\nu'})_{\Delta_\nu\subsetneq \Delta_{\nu'}}, \bigtimes_{\Delta_\nu\subseteq \Delta_{\nu'}} z_{\nu'} \right)$$
is a closed immersion e.g. because it admits a left inverse. Therefore, that the map \eqref{immers1} is a closed immersion follows from the next lemma.

\begin{lem}
Let $I$, $J$ be finite sets and $(X_i)_{i\in I}$, $(Y_j)_{j\in J}$ be families of algebraic varieties over $\bC$. Let $f: \prod_{i\in I} X_i\to \prod_{j\in J} Y_j$ be a regular morphism. Let also $i\mapsto j_i\in J$ be an injective map and $\preceq$ be an order on $I$ such that the following condition is satisfied:
\begin{num}
\item\label{condition} For each $i_0\in I$, the composition of $f$ with the projection $\prod_{j\in J} Y_j\to Y_{j_{i_0}}$ factorizes through the projection $\prod_{i\in I}X_i\to \prod_{i_0\preceq i} X_i$ and the product
$$\displaystyle \prod_{i_0\preceq i} X_i\to \prod_{i_0\prec i} X_i\times Y_{j_{i_0}}$$
of the induced morphism $\prod_{i_0\preceq i} X_i\to Y_{j_{i_0}}$ with the projection $\prod_{i_0\preceq i}X_i\to \prod_{i_0\prec i} X_i$ is a closed immersion.
\end{num}
Then, $f$ is a closed immersion.
\end{lem}

\begin{proof}
It is easy to see that the condition \eqref{condition} is still satisfied for any order finer than $\preceq$. In particular, we may assume that $\preceq$ is a total order. Then, we can write $I=\{i_1,\ldots,i_d \}$ such that $i_k\preceq i_l$ if and only if $k\leqslant l$. By descending induction on $1\leqslant k\leqslant d$, \eqref{condition} implies that the morphism $\prod_{l\geqslant k} X_{i_l}\to \prod_{l\geqslant k} Y_{j_{i_l}}$ is a closed immersion. In particular, for $k=1$ we get that the map $\prod_{i\in I} X_i\to \prod_{i\in I} Y_{j_i}$ is a closed immersion from which it follows that so does $f$.
\end{proof}
\end{proof}

%we either have $\Delta_i\cap \Delta_j=\emptyset$, $\Delta_i\subset \Delta_j$ or $\Delta_j\subset \Delta_i$. On the other hand, decomposing $L$ as a product of general linear groups, $\tau$ is the tensor product (in some order) of the representations $\tau_i\lvert \det \rvert_E^{c}$ for $1\leqslant i\leqslant k$ and $a_i\leqslant c\leqslant b_i$ such that $c-a_i\in \mathbb{N}$. Partitioning the segments $\Delta_1,\ldots,\Delta_k$ into chains for the inclusion, we see that $X_{\unr}(M)/W^0_{\sigma_0}\to X_{\unr}(L)/W^0_\tau$ is a product of regular maps of the form
%$$\displaystyle (\mathbb{C}^\times)^{d_1}/\mathfrak{S}_{d_1}\times\ldots (\mathbb{C}^\times)^{d_r}/\mathfrak{S}_{d_r}\to \left( (\mathbb{C}^\times)^{d_1}/\mathfrak{S}_{d_1}\right)^{e_1}\times \left( (\mathbb{C}^\times)^{d_1+d_2}/\mathfrak{S}_{d_1+d_2}\right)^{e_2}\ldots \times \left( (\mathbb{C}^\times)^{d_1+\ldots+d_r}/\mathfrak{S}_{d_1+\ldots+d_r}\right)^{e_r}$$
%$$\displaystyle ([z_1],\ldots,[z_r])\mapsto ([z_1]^{\times e_1}, [z_1,z_2]^{\times e_2},\ldots, [z_1,\ldots,z_r]^{\times e_r})$$
%for some positive integers $d_1,\ldots,d_r$ and $e_1,\ldots,e_r$. This last map is easily seen to be a closed immersion (e.g. because it admits a section) and this ends the proof of the proposition. $\blacksquare$

%%Note: a counterexample to the fact that $\lambda$ is everywhere unramified: take $n=4$ and $\pi=\lvert .\rvert^{-1/2}\times St\times \lvert .\rvert^{1/2}$, then $\lambda$ is ramified at $\pi$.

Of course, all the above constructions and results apply similarly to $G'$. Let $\BC:\Irr(G')\to \Irr(G)$ be the quadratic base-change map constructed by Arthur and Clozel \cite{AC}. By \cite[Lemma 6.10]{AC}, $\BC$ restricts to a map $\cB(G')\to \cB(G)$. Moreover, by \cite[Lemma 6.12]{AC}, the following diagram is commutative
\begin{equation}\label{diagram bc}
\displaystyle \xymatrix{\Irr(G')\ar[r]^{\BC} \ar[d]^{\lambda} & \Irr(G) \ar[d]^{\lambda} \\ \cB(G') \ar[r]^{\BC} & \cB(G).}
\end{equation}

\begin{lem}\label{lem BC regular}
\begin{enumerate}[(i)]
\item For each connected component $\Omega\subset \Irr(G)$, there exists $(M,\sigma)\in \Sqr(G')$ such that $\BC^{-1}(\Omega)\subseteq \Irr_{M,\sigma}(G')$. Moreover, for every connected components $\Omega,\Omega'\subset \Irr(G')$ we either have $\BC(\Omega)=\BC(\Omega')$ or that $\BC(\Omega)$ and $\BC(\Omega')$ lie in distinct connected components of $\Irr(G)$.
\item $\BC$ is a finite regular map which is flat over its image.
\end{enumerate}
\end{lem}

\begin{proof}
\begin{enumerate}[(i)]
\item This follows rather easily from the description of the fibers of the base-change map \cite[Proposition 6.7]{AC} and its compatibility with parabolic induction.
\item Let $(M,\sigma)\in \Sqr(G')$. By the compatibility between base-change and parabolic induction, there exist $(L,\tau)\in \Sqr(G)$ and a closed embedding $X(M)\to X(L)$ such that the following diagram commutes
$$\displaystyle \xymatrix{ X(M) \ar[r] \ar[d] & X(L) \ar[d] \\ \Irr_{M,\sigma}(G') \ar[r]^{\BC} & \Irr(G)}$$
where the two vertical maps are given by $\chi\mapsto [M,\sigma\otimes \chi]$ and $\chi\mapsto [L,\tau\otimes \chi]$ respectively. Since these two arrows are finite morphisms and the first one is a quotient map by a finite group of automorphisms, it follows that $\BC$ is both regular and finite. To show the flatness of $\BC$ over its image, we will use the ``miracle flatness theorem'' \cite[Exercise III.10.9]{Hart} which implies that a finite surjective morphism between smooth connected varieties is automatically flat. Indeed, by \cite[Theorem 6.2(b)]{AC} the image of $\BC$ is the set of fixed points of the automorphism $c$ of $\Irr(G)$ induced from the non-trivial Galois automorphism of $E/F$. This automorphism is easily seen to be algebraic, hence by \cite[Proposition 1.3]{Iv} the image of $\BC$ is smooth. Thus, by the second part of (i) the image by $\BC$ of a connected component of $\Irr(G')$ is also smooth (being the intersection of the full image with a component of $\Irr(G)$). Since the source is also smooth we can conclude by \cite[Exercise III.10.9]{Hart}.
\end{enumerate}
\end{proof}

\subsection{The result}\label{section result multiplicities}

For $V\in \cV$ and $\pi\in \Irr(G)$ we set
$$\displaystyle m_V(\pi)=\dim \Hom_{U(V)}(\pi,\bC)$$
where $\Hom_{U(V)}(\pi,\bC)$ denotes the space of $U(V)$-invariant functionals on (the space of) $\pi$. We define the {\em degree} of the base-change map to be the function
$$\deg \BC: \Irr(G)\to \mathbb{N}$$
$$\displaystyle \deg \BC(\pi)=\dim \bC[\Irr(G')]/\mathfrak{m}_\pi\bC[\Irr(G')]$$
where $\mathfrak{m}_\pi\subset \mathbb{C}[\Irr(G)]$ denotes the maximal ideal corresponding to $\pi$. By Lemma \ref{lem BC regular}(ii), $\deg \BC$ is locally constant on the image of the base-change map. Thus, to compute it we just need to consider the case where $\pi$ is in general position in the image in which case we simply have $\deg \BC(\pi)=\lvert \BC^{-1}(\pi)\rvert$. By the description of the image and fibers of $\BC$ and its compatibility with parabolic induction (see \cite[Theorem 6.2, Proposition 6.7]{AC}), we obtain the following explicit description: if $\pi\in \Irr(G)$ is the Langlands quotient of an induced representation of the form
$$\displaystyle \sigma_1\times\ldots\times \sigma_k$$
where for each $1\leqslant i\leqslant k$, $\sigma_i\in \Pi_{\ess,2}(\GL_{n_i}(E))$ for some positive integer $n_i$, then we have
\begin{align}\label{eq1 result multiplicities}
\displaystyle \deg \BC(\pi)=\left\{
    \begin{array}{ll}
        2^{\lvert \{1\leqslant i\leqslant k \mid \sigma_i^c\simeq \sigma_i \}\rvert} & \mbox{ if } \pi\simeq \pi^c, \\
        0 & \mbox{ otherwise.}
    \end{array}
\right.
\end{align}

The following result is proved by Feigon-Lapid-Offen in \cite[Theorem 0.2]{FLO}.

\begin{theo}[Feigon-Lapid-Offen]\label{theo mult FLO}
For every $\pi\in \Irr^{\gen}(G)$ and $V\in \cV$ we have
$$\displaystyle m_V(\pi)\geqslant \left\{
    \begin{array}{ll}
        \ceil{\frac{\deg \BC(\pi)}{2}} & \mbox{ if } U(V) \mbox{ is quasi-split}, \\
        \\
        \floor{\frac{\deg \BC(\pi)}{2}} & \mbox{ otherwise.}
    \end{array}
\right.$$
Moreover, equality holds whenever $\BC$ is unramified on the fiber of $\pi$.
\end{theo}

The goal of this chapter is to refine this result and prove the following.

\begin{theo}\label{theo mult}
For every $\pi\in \Irr^{\gen}(G)$ and $V\in \cV$ we have
$$\displaystyle m_V(\pi)= \left\{
    \begin{array}{ll}
        \ceil{\frac{\deg \BC(\pi)}{2}} & \mbox{ if } V \mbox{ is quasi-split}, \\
        \\
        \floor{\frac{\deg \BC(\pi)}{2}} & \mbox{ otherwise.}
    \end{array}
\right.$$
\end{theo}

\subsection{First step: Reduction to the tempered case}\label{section first step multiplicities}

For $\pi \in \Irr(G)$ we set
$$\displaystyle m(\pi):=\sum_{V\in \cV} m_V(\pi).$$
Note that, since we are in the $p$-adic case, the above sum contains only two terms. Moreover, if $n$ is odd every $V\in \cV$ is quasi-split whereas, if $n$ is even one of the Hermitian spaces in $\cV$ is quasi-split and the other is not. Using \eqref{eq1 result multiplicities}, we readily check that if $n$ is odd then $\deg \BC(\pi)$ is always even. Therefore, by Theorem \ref{theo mult FLO}, Theorem \ref{theo mult} is equivalent to
\begin{align}\label{eq1 first step}
\displaystyle m(\pi)\leqslant \deg \BC(\pi)
\end{align}
for every $\pi\in \Irr^{\gen}(G)$.

Let $\pi\in \Irr^{\gen}(G)$. It can be written as
$$\displaystyle \pi=\tau_1\lvert \det\rvert^{\lambda_1}\times\ldots\times \tau_t \lvert \det \rvert^{\lambda_t}$$
where, for each $1\leqslant i\leqslant t$, $\tau_i\in \Temp(\GL_{n_i}(E))$ for some positive integer $n_i$ and $\lambda_1,\ldots,\lambda_t$ are real numbers satisfying $\lambda_1>\lambda_2>\ldots>\lambda_t$. For every $1\leqslant i\leqslant t$, we define $m(\tau_i)$ and $\deg \BC(\tau_i)$ similarly to $m(\pi)$ and $\deg \BC(\pi)$ (just replacing $n$ by $n_i$). The proposition below will allow to reduce the proof of Theorem \ref{theo mult} to the case where $\pi$ is tempered.

\begin{prop}\label{prop first reduction}
We have
\begin{enumerate}[(i)]
\item $\displaystyle \deg \BC(\pi)=\deg \BC(\tau_1)\ldots \deg \BC(\tau_t);$
\item $\displaystyle m(\pi)=m(\tau_1)\ldots m(\tau_t)$.
\end{enumerate}
\end{prop}

\begin{proof}
(i) can be inferred directly from the description \eqref{eq1 result multiplicities} of $\deg \BC(\pi)$. The proof of (ii) essentially follows from the analysis performed in \cite[\S 6]{FLO} but is not explicitely stated there. Therefore, we shall now explain carefully this deduction. Let
$$\displaystyle M=\GL_{n_1}(E)\times\ldots \times \GL_{n_t}(E)$$
be the standard Levi subgroup of $G$ from which $\pi$ is induced as a standard module and
$$\displaystyle \tau=\tau_1\lvert \det\rvert^{\lambda_1}\boxtimes\ldots\boxtimes \tau_t\lvert \det \rvert^{\lambda_t}\in \Irr(M)$$
so that $\pi\simeq I_P^G(\tau)$ where $P$ is the standard parabolic subgroup with Levi $M$. By \cite[Lemma 6.7]{FLO}, we just need to check that the ``unitary periods of $\pi$ are supported on open $P$-orbits" with the terminology of {\em loc. cit.} (see \cite[Definition 6.6]{FLO}). Here, the $P$-orbits refer to the action of $P$ on $X$. Given the explicit description of $P$-orbits from \cite[\S 6.1]{FLO} and of the ``unitary periods" supported on each of these $P$-orbit from \cite[Lemma 6.4]{FLO}, we just need to show the following: if $n_i=n_{i,t}+\ldots+n_{i,1}$ are partitions of the $n_i$'s satisfying $n_{i,j}=n_{j,i}$ for every $1\leqslant i,j\leqslant t$ which are not all trivial (i.e. there exist $1\leqslant i\neq j\leqslant t$ with $n_{i,j}\neq 0$), $P_i$ stands for the standard parabolic subgroup of $\GL_{n_i}(E)$ associated to this partition of $n_i$ with standard Levi
\begin{align}\label{eq2 first step}
\displaystyle M_i=\GL_{n_{i,t}}(E)\times\ldots\times \GL_{n_{i,1}}(E)
\end{align}
and $J_{P_i}(\tau_i\lvert \det\rvert^{\lambda_i})$ denotes the normalized Jacquet module with respect to this parabolic, there is no irreducible subquotients
$$\displaystyle \rho_i=\rho_{i,t}\boxtimes\ldots\boxtimes \rho_{i,1}\in \Irr(M_i)$$
of the $J_{P_i}(\tau_i)$, $1\leqslant i\leqslant t$, such that $\rho_{ij}\simeq \rho_{ji}^c$ for every $1\leqslant i\neq j\leqslant t$. Assume, by way of contradiction, that there exist such partitions and irreducible subquotients of the Jacquet modules. Let $1\leqslant i\leqslant t$ be the smallest index such that the partition of $n_i$ is non-trivial and $1\leqslant j\leqslant t$ be the largest index such that $n_{ij}\neq 0$. Note that $j>i$ as the partition of $n_i$ is non-trivial and $n_{ik}=n_{ki}=0$ for every $k<i$ by minimality of $i$. Let $\mu$ be the real exponent of the central character $\omega_{\rho_{ij}}=\omega_{\rho_{ji}}^c$. As $\GL_{n_{ij}}(E)$ is the first non-trivial group in the product decomposition \eqref{eq2 first step} of $M_i$, by Casselman's criterion of temperedness \cite[Proposition III.2.2]{Wald1} we have $\lambda_i\leqslant \frac{\mu}{n_{ij}}$. Similarly, since $n_{jk}=n_{kj}=0$ for $k<i$ (again by minimality of $i$), by Casselman's criterion of temperedness we have $\lambda_j\geqslant \frac{\mu}{n_{ji}}=\frac{\mu}{n_{ij}}$. But $j>i$ implies that $\lambda_i>\lambda_j$ and therefore we have a contradiction.
\end{proof}

\subsection{Second step: relation between multiplicities and FLO functionals}\label{section second step multiplicities}

For $\pi\in \Irr(G)$, we let $C_c^\infty(X)_\pi$ be the $\pi^\vee$-isotypic quotient of $C_c^\infty(X)$ i.e. the maximal quotient which is $G$-isomorphic to a direct sum of copies of $\pi^\vee$. Note that by Frobenius reciprocity, since $X=\bigsqcup_{V\in \mathcal{V}} U(V)\backslash G$ (see Section \ref{S:symm X}), we have
$$\displaystyle C_c^\infty(X)_\pi\simeq (\pi^\vee)^{\oplus m(\pi)},\;\;\; \pi\in \Irr(G).$$
Therefore, the following lemma is just a consequence of the unicity of Whittaker models.

\begin{lem}\label{lem1 1st step mult}
For $\pi\in \Irr^{\gen}(G)$, we have
$$\displaystyle m(\pi)=\dim \Hom_N(C_c^\infty(X)_\pi, \psi_n).$$
\end{lem}

Recall the FLO relative character $J_\sigma$ associated to each $\sigma\in \Temp(G')$ introduced in Section \ref{section FLO funct}. Note that $J_\sigma\in \Hom_N(C_c^\infty(X), \psi_n)$ for every $\sigma\in \Temp(G')$. Let $\pi\in \Temp(G)$, $\Omega_\pi\subseteq \Irr(G)$ be the connected component of $\pi$ and $\Omega_\pi^t=\Omega_\pi\cap \Temp(G)$. We equip $\Hom_N(C_c^\infty(X), \psi_n)$ with the weak topology (that is the topology of pointwise convergence). Set
$$\displaystyle \mathcal{J}(\pi):=\overline{\langle J_\sigma\mid \sigma\in \BC^{-1}(\Omega_\pi^t)\rangle}$$
for the closure of the subspace of $\Hom_N(C_c^\infty(X), \psi_n)$ generated by the FLO relative characters $J_\sigma$ with $\sigma\in \BC^{-1}(\Omega_\pi^t)$. The main result of this section is the following proposition.

\begin{prop}\label{prop1 1st step mult}
We have
$$\displaystyle \Hom_N(C_c^\infty(X)_\pi, \psi_n)\subseteq \mathcal{J}(\pi).$$
\end{prop}

\begin{proof}
Let $J\in \Hom_N(C_c^\infty(X)_\pi, \psi_n)$. We need to show that for every $\varphi\in C_c^\infty(X)$ such that $J_\sigma(\varphi)=0$ for all $\sigma\in \BC^{-1}(\Omega_\pi^t)$ we have $J(\varphi)=0$. By Lemma \ref{lem1 abstract rc}, $J$ and the relative characters $J_\sigma$, for $\sigma\in \Temp(G')$, extend continuously to $\cC(X)$ and we will prove that the previous property holds more generally for $\varphi\in \cC(X)$.

The application $f\in \cC(G)\mapsto (\pi'\in \Temp(G)\mapsto \pi'(f))$ is injective and its image was described by Harish-Chandra \cite[Th\'eor\`emes VII.2.5 et VIII.1.1]{Wald1}. A consequence of this description is that there exists a projector $f\in \cC(G)\mapsto e_{\Omega_\pi^t}\ast f\in \cC(G)$ which is equivariant with respect to both left and right convolutions such that for every $f\in \cC(G)$ and $\pi'\in \Temp(G)$ we have \footnote{The existence of such a projector can also be deduced from the description of the {\em tempered Bernstein center} by Schneider and Zink \cite{SZ}}
\begin{equation}\label{eq1 tempered projector}
\displaystyle \pi'(e_{\Omega_\pi^t}\ast f)=\left\{
\begin{array}{ll}
\pi'(f) & \mbox{ if } \pi'\in \Omega_\pi^t, \\
0 & \mbox{ otherwise.}
\end{array}
\right.
\end{equation}
By Proposition \ref{prop HCS}(ii), we can define a similar projector $\varphi\in \cC(X)\mapsto e_{\Omega_\pi^t}\ast \varphi\in \cC(X)$: for $\varphi\in \cC(X)$, choose any $f\in C_c^\infty(G)$ such that $\varphi=R(f)\varphi$ (e.g. $\vol(K_0)^{-1}\mathbf{1}_{K_0}$ for a sufficiently small compact-open subgroup $K_0$) and set $e_{\Omega_\pi^t}\ast \varphi=R(e_{\Omega_\pi^t}\ast f)\varphi$: the fact that $e_{\Omega_\pi^t}\ast.$ is equivariant with respect to right convolution ensures that the result does not depend on the choice of $f$.

Let $\pi'\in \Temp(G)$ and $T_{\pi'}:\cC(X)\to \pi'$ be a continuous $G$-equivariant linear map where {\em continuous} here means that for every compact-open subgroup $K_0$ of $G$, the restriction $\cC(X)^{J}\to (\pi^\vee)^{K_0}$ is continuous. Then, $T_{\pi'}(R(f)\varphi)=\pi'(f)T_{\pi'}(\varphi)$ for every $(f,\varphi)\in \cC(G)\times \cC(X)$ and therefore, by \eqref{eq1 tempered projector} and the definition of $e_{\Omega_\pi^t}\ast \varphi$, it follows that:
\begin{equation}
\displaystyle T_{\pi'}(e_{\Omega_\pi^t}\ast \varphi)=\left\{
\begin{array}{ll}
T_{\pi'}(\varphi) & \mbox{ if } \pi'\in \Omega_\pi^t, \\
0 & \mbox{ otherwise}
\end{array}
\right.
\end{equation}
for all $\varphi\in \cC(X)$.

By Frobenius reciprocity, $J$ and $J_{\sigma}$, for $\sigma\in \Temp(G')$, induce continuous $G$-equivariant linear maps $\cC(X)\to \cW(\pi^\vee,\psi_n)$ and $\cC(X)\to \cW(\BC(\sigma)^\vee,\psi_n)$ respectively. Thus, by the above, we have
$$\displaystyle J(e_{\Omega_\pi^t}\varphi)=J(\varphi)$$
and
$$\displaystyle J_\sigma(e_{\Omega_\pi^t}\varphi)=\left\{
\begin{array}{ll}
J_\sigma(\varphi) & \mbox{ if } \sigma\in \BC^{-1}(\Omega_\pi^t), \\
0 & \mbox{ otherwise,}
\end{array}
\right.$$
for every $\varphi\in \cC(X)$ and $\sigma\in \Temp(G')$.

As a consequence, up to replacing $\varphi$ by $e_{\Omega_\pi^t}\varphi$, we only need to show that:
\begin{equation}\label{eq1 1st step mult}
\mbox{For every } \varphi\in \cC(X)\mbox{ such that }J_\sigma(\varphi)=0\mbox{ for every }\sigma\in \Temp(G'),\mbox{ we have }J(\varphi)=0.
\end{equation}
We henceforth fix a function $\varphi\in \cC(X)$ satisfying $J_\sigma(\varphi)=0$ for every $\sigma\in \Temp(G')$.

By Lemma \ref{lem1 abstract rc}, there exists $F\in \cC^w(X)$ such that
$$\displaystyle J(\varphi)=\int_N^* \langle R(u)\varphi,F\rangle_X \psi_n(u)^{-1}du.$$
Let $(X_k)_{k\geqslant 1}$ be an increasing and exhausting sequence of $K$-invariant compact subsets of $X$ and set $F_k=\mathbf{1}_{X_k} F$ for every $k\geqslant 1$. We can show, by the same argument as for \eqref{density HCS}, that the sequence $(F_k)_{k\geqslant 1}$ converges to $F$ in $\cC^w(X)$. Hence, by Proposition \ref{prop HCS}(i), we have
$$\displaystyle J(\varphi)=\lim\limits_{k\to \infty} \int_N^* \langle R(u)\varphi,F_k\rangle_X \psi_n(u)^{-1}du=\lim\limits_{k\to \infty}J(\varphi,F_k)$$
with the notation of Section \ref{section JY TF geom}. Therefore, by Theorem \ref{theo JY TF spec2} and the hypothesis made on $\varphi$, we have
$$\displaystyle J(\varphi)=\lim\limits_{k\to \infty}\int_{\Temp(G')} J_\sigma(\varphi)\overline{J_\sigma(F_k)}d\mu_{G'}(\sigma)=0.$$
This shows \eqref{eq1 1st step mult} and ends the proof of the proposition.
\end{proof}

\subsection{End of the proof of Theorem \ref{theo mult}}\label{section end of proof multiplicities}

For convenience, here we normalize the action of the Bernstein center $\cZ(G)$ on $C_c^\infty(X)$ such that $z\in \cZ(G)$ acts on the coinvariant space $C_c^\infty(X)_{\pi}$ by the scalar $z(\lambda(\pi))$ for every $\pi\in \Irr(G)$.

Let $\pi\in \Temp(G)$ and $\Omega_\pi\subseteq \Irr(G)$ be the connected component of $\pi$. Set $\Omega_\pi^t=\Omega_\pi\cap\Temp(G)$, $\Omega_\pi'=\BC^{-1}(\Omega_\pi)$, $\Omega_\pi^\lambda=\lambda(\Omega_\pi)\subseteq \cB(G)$  and $(\Omega_\pi')^\lambda=\lambda(\Omega'_\pi)\subseteq \cB(G')$. %%Note that, by Lemma \ref{lem finite morphism}, $(\Omega_\pi')^\lambda$ is Zariski closed in $\cB(G')$. %%By Lemma \ref{lem BC regular} (i), there exists $(M,\sigma)\in \Sqr(G')$ such that $\Omega_\pi'$ is a finite union of components of $\Irr_{M,\sigma}(G')$. 
%%Moreover, since the base-change of an irreducible representation $\pi'$ of $G'$ is tempered if and only if $\pi'$ is itself tempered, we have $BC^{-1}(\Omega_\pi^t)=\Omega_\pi'\cap \Temp(G')$.
Let $V$ be the space of functions of the form
$$\displaystyle \sigma\in \BC^{-1}(\Omega_\pi^t)\mapsto J_\sigma(\varphi)$$
where $\varphi\in C_c^\infty(X)$. Then, $V$ is a quotient of $C_c^\infty(X)$ by a $\cZ(G)$-submodule. Moreover, by the definition of FLO functionals (Theorem \ref{theo FLO funct}) and the existence of the Jacquet-Ye transfer (Theorem \ref{theo JY transfer}), $V$ is also the space of functions of the form
$$\displaystyle \sigma\in \BC^{-1}(\Omega_\pi^t)\mapsto I_\sigma(f')$$
where $f'\in C_c^\infty(G')$. Note that, by Lemma \ref{lem finite morphism}, $(\Omega_\pi')^\lambda$ is Zariski closed in $\cB(G')$. Therefore, by Theorem \ref{theo Whittaker-Paley-Wiener theorem}, $V$ is the space of restrictions to $\BC^{-1}(\Omega_\pi^t)$ of the algebra of regular functions $\bC[(\Omega_\pi')^\lambda]$ on $(\Omega_\pi')^\lambda$ through the map $\lambda$. As $\BC^{-1}(\Omega_\pi^t)=\Omega_\pi'\cap \Temp(G')$ is Zariski dense in $\Omega_\pi'$ by \eqref{density tempered}, this gives an isomorphism
\begin{equation}\label{isom V}
\displaystyle V\simeq \bC[(\Omega_\pi')^\lambda]
\end{equation}
through which the action of $\cZ(G)$ is given by the pullback $\BC^*: \cZ(G)=\bC[\cB(G)]\to \bC[\cB(G')]$.

Let $\mathfrak{m}_{\lambda(\pi)}\subseteq \cZ(G)$ be the maximal ideal corresponding to $\lambda(\pi)\in \cB(G)$. Then, by Proposition \ref{prop1 1st step mult}, each element of $\Hom_N(C_c^\infty(X)_\pi,\psi_n)$ factorizes through the quotient $C_c^\infty(X)\to V$ and therefore, by the theory of the Bernstein center and the isomorphism \eqref{isom V}, also through
$$\displaystyle V/\mathfrak{m}_{\lambda(\pi)}V\simeq \bC[(\Omega_\pi')^\lambda]/\mathfrak{m}_{\lambda(\pi)}\bC[(\Omega_\pi')^\lambda].$$
Consequently, by Lemma \ref{lem1 1st step mult}, we have
\begin{equation}\label{1st ineq multiplicity}
\displaystyle m(\pi)\leqslant \dim(\bC[(\Omega_\pi')^\lambda]/\mathfrak{m}_{\lambda(\pi)}\bC[(\Omega_\pi')^\lambda]).
\end{equation}

Consider the following commutative diagram (coming from restriction of \eqref{diagram bc})
$$\displaystyle \xymatrix{ \Omega_\pi' \ar[r]^{\BC} \ar[d]^{\lambda} & \Omega_\pi \ar[d]^{\lambda} \\ (\Omega_\pi')^\lambda \ar[r]^{\BC} & \Omega_\pi^\lambda.}$$
By Proposition \ref{prop local isom lambda} and Lemma \ref{lem BC regular} (i), the two vertical arrows are isomorphisms when restricted to suitable Zariski open neighborhood of $\lambda(\pi)$ and $\BC^{-1}(\lambda(\pi))\cap (\Omega'_\pi)^{\lambda}=\lambda(\BC^{-1}(\pi))$. Therefore,
$$\displaystyle \bC[(\Omega_\pi')^\lambda]/\mathfrak{m}_{\lambda(\pi)}\bC[(\Omega_\pi')^\lambda]\simeq \bC[\Omega_\pi']/\mathfrak{m}_{\pi}\bC[\Omega_\pi'].$$
Combining this with \eqref{1st ineq multiplicity}, we obtain
$$\displaystyle m(\pi)\leqslant \dim(\bC[\Omega_\pi']/\mathfrak{m}_{\pi}\bC[(\Omega_\pi')]) =\deg \BC(\pi).$$

We have just proven that \eqref{eq1 first step} holds for every $\pi\in \Temp(G)$ and therefore, by Proposition \ref{prop first reduction}, also for every $\pi\in \Irr^{\gen}(G)$. This ends the proof of Theorem \ref{theo mult}.

\section{A Plancherel formula for $X$ and relation to factorization of global periods}\label{chapter Plancherel formula}

In this chapter, we keep the notation introduced in the Chapters \ref{chapter invt funct} and \ref{chapter JY TF} and we don't assume anymore that $F$ is a $p$-adic field (i.e. we allow $F=\mathbb{R}$). The goal of this part is to establish an explicit Plancherel formula for $X$. More precisely, we will prove that the $L^2$-inner product $\langle .,.\rangle_X$ on $X$ decomposes as an integral of certain $G$-invariant semi-positive Hermitian forms $\langle .,.\rangle_{X,\sigma}$ that are indexed by $\sigma\in \Temp(G')$ and ``living on $\BC(\sigma)$'' in the sense that they factorize through the $\BC(\sigma)^\vee$-coinvariant space $C_c^\infty(X)_{\BC(\sigma)}$ (see Theorem \ref{theo PLancherel formula}). The Hermitian forms $\langle .,.\rangle_{X,\sigma}$ are defined through the FLO functionals $\alpha^\sigma$ of Section \ref{section FLO funct} and moreover the underlying spectral measure is the Plancherel measure $d\mu_{G'}$ of $G'$. According to Bernstein \cite{Ber3}, such a decomposition induces an isomorphism of unitary representations
\begin{equation}\label{isom L2X}
\displaystyle L^2(X)\simeq \int_{\Temp(G')}^{\oplus}\BC(\sigma) d\mu_{G'}(\sigma)
\end{equation}
and it is actually also equivalent to a certain Plancherel inversion formula expressing any test function $\varphi\in C_c^\infty(X)$ as an integral of ``generalized eigenfunctions'' $\varphi_\sigma$ (see Theorem \ref{Theorem Plancherel inversion}). The isomorphism \eqref{isom L2X} can be seen as a particular case of a general conjecture of Sakellaridis-Venkatesh on the $L^2$-spectrum of spherical varieties \cite[Conjecture 16.2.2]{SV}. More precisely, in \cite{SV} a dual group is associated to any spherical variety\footnote{This construction actually only works well under a suitable extra technical condition (namely that the spherical variety has no root of `type N') for which we refer the reader to {\em loc. cit.}} which for the case at hand is the group $\check{G}_X=\GL_n(\bC)=\check{G}'$ coming with a natural ``distinguished morphism'' $\check{G}_X\to \check{G}$ to the dual group of $G$. Here, this morphism extends naturally to the base-change map between $L$-groups ${}^L G'\to {}^L G$ and \cite[Conjecture 16.2.2]{SV}, suitably interpreted, predicts exactly a decomposition of the $G$-unitary representation $L^2(X)$ of the form \eqref{isom L2X}. A concrete consequence of this Plancherel decomposition is a description of the so-called {\em relative discrete series} of $X$ (see Corollary \ref{cor discrete series}).

The precise statement of the Plancherel formula is given in the next section. The proof, which is relatively short and builds upon the local Jacquet-Ye trace formula of Chapter \ref{chapter JY TF} together with the Fourier inversion formula \eqref{eq1 Fourier inversion}, occupies Section \ref{sect proof Plancherel}. In the final Section \ref{sect global periods}, we revisit the work of Feigon-Lapid-Offen \cite{FLO} on the factorization of unitary periods (generalizing previous work of Jacquet \cite{Jac01}) to make the relation to the local Plancherel decomposition we have obtained more transparent. That there is such a relation is of course not surprising, since the FLO functionals we use to compute the Plancherel decomposition are also the main local input in {\em loc. cit.} to the global period factorization, but once properly reformulated we find this connection to be in striking accordance with general speculations of Sakellaridis-Venkatesh on the factorization of global spherical periods \cite[\S 17]{SV} which is why we have included such a discussion here.

\subsection{The statement}\label{section Planch formula statement}

Let $\sigma\in \Temp(G')$. Recall from Section \ref{section FLO funct} that to $\sigma$ is associated a functional $\alpha^\sigma\in \mathcal{E}_G(X,\cW(\pi,\psi_n)^*)$ where $\pi=\BC(\sigma)$. For $\varphi\in C_c^\infty(X)$, we construct as in Section \ref{section FLO funct} a smooth functional $\varphi\cdot \alpha^\sigma\in \cW(\pi,\psi_n)^\vee$ that we identify with an element of $\cW(\pi^\vee,\psi_n^{-1})$ through the invariant inner product $\langle .,.\rangle_{\Whitt}$ \eqref{def pairing Whittaker}. For every $\varphi_1,\varphi_2\in C_c^\infty(X)$, we set
$$\displaystyle \langle \varphi_1,\varphi_2\rangle_{X,\sigma}:= \langle\varphi_1\cdot \alpha^\sigma, \varphi_2 \cdot \alpha^\sigma\rangle_{\Whitt}.$$
Obviously, $\langle .,.\rangle_{X,\sigma}$ is a $G$-invariant positive semi-definite Hermitian form that factorizes through the $\pi^\vee$-coinvariants $C_c^\infty(X)\to C_c^\infty(X)_\pi$.

Finally, recall that $\langle .,.\rangle_X$ stands for the $L^2$-scalar product on $X$ and $d\mu_{G'}$ denotes the Plancherel measure on $G'$.

\begin{theo}\label{theo PLancherel formula}
For every $\varphi_1,\varphi_2\in C_c^\infty(X)$, we have
$$\displaystyle \langle \varphi_1,\varphi_2\rangle_X=\int_{\Temp(G')} \langle \varphi_1,\varphi_2\rangle_{X,\sigma} d\mu_{G'}(\sigma)$$
where the right hand side is absolutely convergent.
\end{theo}

%According to Bernstein \cite{Ber3}, such a formula induces a decomposition as a direct integral
%$$\displaystyle L^2(X)\simeq \int_{\Temp(G')}^{\oplus}\BC(\sigma) d\mu_{G'}(\sigma)$$
%of the right regular representation of $G$ on $L^2(X)$. Such a decomposition is in accordance with a general conjecture of Sakellaridis-Venkatesh \cite[Conjecture 16.2.2]{SV} on the $L^2$-decomposition of local spherical varieties (in the case at hand the dual group of $X$ is $\check{G}_X=\GL_n(\bC)={}^L G'$ and the ``distinguished morphism'' ${}^L G'\to {}^L G$ should be taken to be the base-change homomorphism).

Note that the action of the center $Z(G)=E^\times$ on $X$ factorizes through the quotient $E^\times\to N(E^\times)$. Let $\chi:N(E^\times)\to \mathbb{S}^1$ be an unitary character and $L^2(X,\chi)$ be the space of functions $f:X\to \bC$ satisfying $f(xz)=\chi(z)f(x)$ for every $(x,z)\in X\times Z(G)$ and which are square-integrable on $X/Z(G)$. Let $L^2(X,\chi)_{\disc}$ the subspace generated by all the irreducible smooth submodules of $L^2(X,\chi)$ (the so-called {\em relative discrete series}) and $\Pi_{2,\chi}(G')$ be the subset of representations $\sigma\in \Pi_2(G')$ whose central character restricted to $N(E^\times)\subset Z(G')$ is equal to $\chi$. The above decomposition of $L^2(X)$ admits the following concrete representation-theoretic corollary.

\begin{cor}\label{cor discrete series}
There is a $G$-isomorphism
$$\displaystyle L^2(X,\chi)_{\disc}\simeq \bigoplus_{\sigma\in \Pi_{2,\chi}(G')}\BC(\sigma).$$
\end{cor}

Let $x\in X$. The value of $\alpha^\sigma$ at $x$ is a $G_x$-invariant functional $\alpha^\sigma_x: \cW(\pi,\psi_n)\to \bC$. Identifying its complex conjugate $\overline{\alpha^\sigma_x}$ with a functional on $\overline{\cW(\pi,\psi_n)}=\cW(\pi^\vee,\psi_n^{-1})$, for every $\varphi\in C_c^\infty(X)$ we set
$$\displaystyle \varphi_\sigma(x)=\langle \varphi\cdot \alpha^\sigma, \overline{\alpha_x^\sigma}\rangle.$$
Note that the function $\varphi_\sigma$ generates (by right-translation) a representation isomorphic to $\pi^\vee=\BC(\sigma)^\vee$. In this sense, it is a ``generalized eigenfunction''. The following explicit ``Plancherel inversion formula'' follows from Theorem \ref{theo PLancherel formula} by specializing it to the case where $\varphi_1=\varphi$ and $\varphi_2=\mathbf{1}_{xK_0}$ for $K_0$ a sufficiently small compact-open subgroup of $G$ in the $p$-adic case. In the Archimedean case, we can argue in a similar way using the Dixmier-Malliavin theorem (details are left to the reader).

\begin{theo}\label{Theorem Plancherel inversion}
For every $\varphi\in C_c^\infty(X)$ and $x\in X$, we have
$$\displaystyle \varphi(x)=\int_{\Temp(G')} \varphi_\sigma(x) d\mu_{G'}(\sigma)$$
where the right hand side is absolutely convergent.
\end{theo}

\subsection{Proof of Theorem \ref{theo PLancherel formula}}\label{sect proof Plancherel}

Note that, for every $\sigma\in \Temp(G')$ and $\varphi_1,\varphi_2\in C_c^\infty(X)$ and since the scalar product $\langle .,.\rangle_{X,\sigma}$ is $G$-invariant and factorizes through the $\pi^\vee=\BC(\sigma^\vee)$-coinvariants $C_c^\infty(X)_\pi$, the function $g\in G\mapsto \langle R(g)\varphi_1,\varphi_2\rangle_{X,\sigma}$ is a finite sum of matrix coefficients of $\pi^\vee$ hence belongs to $\cC^w(G)$. In particular, we can apply to it the regularized integral $\displaystyle \int_N^* . \psi_n(u)^{-1}du$ of Section \ref{section HCS}.

\begin{lem}\label{lem rel chars and scalar products}
For every $\sigma\in \Temp(G')$ and $\varphi_1,\varphi_2\in C_c^\infty(X)$, we have
$$\displaystyle J_\sigma(\varphi_1)\overline{J_\sigma(\varphi_2)}=\int_N^* \langle R(u) \varphi_1,\varphi_2\rangle_{X,\sigma}\psi_n(u)^{-1}du.$$
\end{lem}

\begin{proof}
By \eqref{eq2 Fourier inversion} and the definition of $\langle .,.\rangle_{X,\sigma}$ and $J_\sigma$, we have
\[\begin{aligned}
\displaystyle \int_N^* \langle R(u) \varphi_1,\varphi_2\rangle_{X,\sigma}\psi_n(u)^{-1}du & =\int_N^* \langle R(u^{-1})(\varphi_1\cdot \alpha^\sigma), \varphi_2\cdot \alpha^\sigma\rangle_{\Whitt}\psi_n(u)^{-1}du \\
 & =\langle \varphi_1\cdot \alpha^\sigma,\lambda^\vee_1\rangle \overline{\langle \varphi_2\cdot \alpha^\sigma,\lambda^\vee_1\rangle} =J_\sigma(\varphi_1)\overline{J_\sigma(\varphi_2)}
\end{aligned}\]
where we recall that $\lambda^\vee_1$ stands for the functional $W^\vee\in \cW(\pi^\vee,\psi^{-1}_n)\mapsto W^\vee(1)$.
\end{proof}

We can now finish the proof of Theorem \ref{theo PLancherel formula}. Since both $\langle .,.\rangle_X$ and $\langle .,.\rangle_{X,\sigma}$, $\sigma\in \Temp(G')$, are positive semi-definite Hermitian forms, by Cauchy-Schwarz and the polarization formula, it suffices to prove the theorem when $\varphi_1=\varphi_2=\varphi\in C_c^\infty(X)$. By \eqref{eq2 X tempered}, \eqref{eq1 Fourier inversion}, the definition \eqref{def J} of $J(\varphi,\varphi)$ and Theorem \ref{theo JY TF spec}, we have
\[\begin{aligned}
\displaystyle \langle \varphi,\varphi\rangle_X & =\int_{N\backslash P} J(R(p)\varphi,R(p)\varphi)dp \\
 & =\int_{N\backslash P}\int_{\Temp(G')} \lvert J_\sigma (R(p)\varphi)\rvert^2 d\mu_{G'}(\sigma)dp.
\end{aligned}\]
Since the integrand in the last expression above is nonnegative, this expression is absolutely convergent. By Lemma \ref{lem rel chars and scalar products} and the inversion formula \eqref{eq1 Fourier inversion}, we have
$$\displaystyle \int_{N\backslash P} \lvert J_\sigma (R(p)\varphi)\rvert^2dp=\langle \varphi,\varphi\rangle_{X,\sigma}$$
for every $\sigma\in \Temp(G')$. Hence, we get
\[\begin{aligned}
\displaystyle \langle \varphi,\varphi\rangle_X =\int_{\Temp(G')} \int_{N\backslash P} \lvert J_\sigma (R(p)\varphi)\rvert^2 dpd\mu_{G'}(\sigma) =\int_{\Temp(G')} \langle \varphi,\varphi\rangle_{X,\sigma} d\mu_{G'}(\sigma)
\end{aligned}\]
showing at once the identity and the convergence of the right-and side of Theorem \ref{theo PLancherel formula} when $\varphi_1=\varphi_2=\varphi$.

\subsection{Relation to the factorization of global periods}\label{sect global periods}

\textbf{In this section, we assume that $n$ is odd.}

Recall that there is a natural left $F^\times$-action on $X$. We denote the corresponding diagonal action by left translation of $F^\times$ on $C_c^\infty(X\times X)$ by $L^\Delta$ (that is $L^\Delta(\lambda)\Phi=\Phi(\lambda^{-1}.,\lambda^{-1}.)$ for $\Phi\in C_c^\infty(X\times X)$ and $\lambda\in F^\times$). Let $C_c^\infty(X\times X)_G$ be the $G$-coinvariant space of $C_c^\infty(X\times X)$ for the diagonal action by right translation of $G$. Then, we say that a function $\Phi\in C_c^\infty(X\times X)$ is {\em $F^\times$-stable} if for every $\lambda\in F^\times$, $\Phi-L^\Delta(\lambda)\Phi$ maps to $0$ in $C_c^\infty(X\times X)_G$. By \eqref{equivariance FLO functionals center}, we readily check that if $\Phi=\varphi_1\otimes \varphi_2$ is $F^\times$-stable then for every $\sigma\in \Temp(G')$, we have
%%\[\begin{aligned}
%%\displaystyle & \langle \overline{\varphi_2\cdot \alpha^\sigma},\varphi_1\cdot \alpha^{\sigma\otimes \eta}\rangle_{\Whitt}=\langle \overline{L(\lambda)\varphi_2\cdot \alpha^\sigma},L(\lambda)\varphi_1\cdot \alpha^{\sigma\otimes \eta}\rangle_{\Whitt} \\
%% & =\overline{\omega_\sigma(\lambda)} \omega_{\sigma\otimes \eta}(\lambda) \langle \overline{\varphi_2\cdot \alpha^\sigma},\varphi_1\cdot \alpha^{\sigma\otimes \eta}\rangle_{\Whitt}=\eta(\lambda)\langle \overline{\varphi_2\cdot \alpha^\sigma},\varphi_1\cdot \alpha^{\sigma\otimes \eta}\rangle_{\Whitt}
%%\end{aligned}\]
%%for all $\lambda\in F^\times$ (here we have used that $n$ is odd), hence
\begin{equation}\label{eq annulation stable}
\displaystyle \langle\varphi_1\cdot \alpha^{\sigma\otimes \eta}, \varphi_2\cdot \alpha^\sigma\rangle_{\Whitt}=0.
\end{equation}

We now move to a global setting and consider a quadratic extension $k/k'$ of number fields. We write $\bA$ for the adele ring of $k'$, $\eta:\bA^\times/(k')^\times\to \{\pm 1 \}$ for the idele class character associated to the extension and for every place $v$ of $k'$, we denote by $k'_v$ the corresponding completion, by $\cO_v$ its ring of integers in case it is non-Archimedean and by $k_v$ the tensor product $k\otimes_{k'} k'_v$. We also change slightly notation to denote by $G'$ the group $\GL_n$ over $k'$, by $G=\Res_{k/k'}\GL_n$ the algebraic group obtained by restriction of scalar of $\GL_n$ from $k$ to $k'$ and by $X$ the algebraic variety (over $k'$) of non-degenerate Hermitian forms on $k^n$. There is a natural right action of $G$ on $X$ and for each place $v$ of $k'$ inert in $k$, the groups $G'_v=G'(k'_v)$, $G_v=G(k'_v)$ and the variety $X_v=X(k'_v)$ are what we have denoted $G'$, $G$ and $X$ so far for $F=k'_v$ and $E=k_v$.

When $v$ is inert in $k$, for every $\sigma_v\in \Temp(G'_v)$ we denote by $\langle .,.\rangle_{X_v,\sigma_v}$ the inner product on $C_c^\infty(X_v)$ defined in Section \ref{section Planch formula statement}. When $v$ splits in $k$, we define an inner product $\langle .,.\rangle_{X_v,\sigma_v}$ on $C_c^\infty(X_v)$ for every $\sigma_v\in \Temp(G'_v)$ as follows: choosing a place of $k$ above $v$ we get an identification $k_v\simeq k'_v\times k'_v$ and projection on the first component induces an isomorphism $X_v\simeq \GL_n(k'_v)=G'_v$, then we set
$$\displaystyle \langle \varphi_{1,v},\varphi_{2,v}\rangle_{X_v,\sigma_v}=\Tr(\sigma_v(\varphi_{1,v}\star \varphi_{2,v}^*)),\;\; \varphi_1,\varphi_2\in C_c^\infty(X_v),$$
where $\displaystyle (\varphi_{1,v}\star \varphi_{2,v}^*)(x)=\int_{X_v} \varphi_{1,v}(xy)\overline{\varphi_{2,v}(y)}dy$ (for $x\in X_v$) and $\displaystyle \sigma_v(\varphi_v)=\int_{G'_v} \varphi_v(h)\sigma_v(h)dh$ (for $\varphi_v\in C_c^\infty(X_v)=C_c^\infty(G'_v)$). Note that for these inner products, the analog of Theorem \ref{theo PLancherel formula} holds by Harish-Chandra Plancherel formula for $G'_v$.

When the place $v$ is split, by the above definition, it is clear that the inner product $\langle .,.\rangle_{X_v,\sigma_v}$ only depends on the choice of invariant measures on $X_v$ and $G'_v$. It is also true when $v$ is inert as follows from the identity of Theorem \ref{theo PLancherel formula} (the Plancherel measure $d\mu_{G'_v}(\sigma_v)$ is inversely proportional to the Haar measure on $G'_v$). This can alternativey be checked (slightly painfully) by tracing back all the constructions and normalizations of this paper (More precisely, we have made two auxilliary choices in the construction: a Haar measure on $T'$ and a nontrivial additive character $\psi'$).

We now normalize the local measures on $X_v$ and $G'_v$ so that they factorize the global invariant Tamagawa measures on $X(\bA)$ and $G'(\bA)$ and give, for almost all places $v$, volume $1$ to the subsets of integral points $X(\cO_v)$, $G'(\cO_v)$.

Let $\Phi=\varphi_1\otimes \varphi_2\in C_c^\infty(X(\bA))\otimes C_c^\infty(X(\bA))$ and assume that the functions $\varphi_1,\varphi_2$ are products $\varphi_1=\prod_v \varphi_{1,v}$, $\varphi_2=\prod_v \varphi_{2,v}$ where $\varphi_{1,v},\varphi_{2,v}\in C_c^\infty(X_v)$ for each place $v$ of $k'$. Let $\sigma=\bigotimes_v' \sigma_v$ be a cuspidal automorphic representation of $G'(\bA)$ such that for each place $v$, the local representation $\sigma_v$ is tempered. We denote by $L(s,\sigma, \Ad)$ (resp. $L(s,\sigma,\Ad\otimes \eta)$) the adjoint $L$-function $L(s,\sigma\times \sigma^\vee)$ (resp. the twisted adjoint $L$-function $L(s,\sigma\eta\times \sigma^\vee)$) of $\sigma$. For any finite set $S$ of places (resp. place $v$), we write $L^S(s,\sigma,\Ad)$ and $L^S(s,\sigma,\Ad\otimes \eta)$ (resp. $L(s,\sigma_v,\Ad)$ and $L(s,\sigma_v,\Ad\otimes \eta)$) for the corresponding partial $L$-functions (resp. local $L$-factors) and we set $L^{*,S}(1,\sigma,\Ad)=\Res_{s=1} L^S(s,\sigma,\Ad)$. Since $n$ is odd, $\sigma\not\simeq \sigma\otimes \eta$ and the partial $L$-function $L^S(s,\sigma,\Ad\otimes \eta)$ is regular at $s=1$ (for any $S$). Moreover, by the unramified computations of \cite[Lemma 3.9]{FLO} and \cite[Proposition 2.3]{JS}, for almost all places $v$ of $k'$ we have
$$\displaystyle \langle \varphi_{1,v},\varphi_{2,v}\rangle_{X_v,\sigma_v}=\frac{L(1,\sigma_v,\Ad\otimes \eta)}{L(1,\sigma_v,\Ad)}.$$
(Note that when $v$ is split, the right-hand side is simply $1$). Therefore, for any sufficiently large finite set of places $S$ of $k'$, we can set
$$\displaystyle \langle \varphi_1,\varphi_2\rangle_{X,\sigma}=\frac{L^S(1,\sigma,\Ad\otimes \eta)}{L^{*,S}(1,\sigma,\Ad)}\prod_{v\in S} \langle \varphi_{1,v},\varphi_{2,v}\rangle_{X_v,\sigma_v}.$$

Let $\varphi\in C_c^\infty(X(\bA))$. We denote by $\Sigma \varphi$ the function on $[G]=G(k')\backslash G(\bA)$ defined by
$$\displaystyle (\Sigma \varphi)(g)=\sum_{x\in X(k')} \varphi(xg),\;\;\; g\in [G].$$
Let $\pi$ be a cuspidal automorphic representation of $G(\bA)$. We equip it with the Petersson inner product
$$\displaystyle \langle \phi,\phi\rangle_{\Pet}=\int_{G(k')\backslash G(\bA)^1} \lvert \phi(g)\rvert^2 dg$$
where $G(\bA)^1$ is the subgroup of matrices $g\in G(\bA)=\GL_n(\bA_{k})$ ($\bA_{k}$ denoting the adele ring of $k$) such that $\lvert \det(g)\rvert=1$ and $dg$ is the Tamagawa measure (i.e. the one giving $G(k')\backslash G(\bA)^1$ volume $1$). We then write $(\Sigma \varphi)_\pi$ for the $\pi$-projection of $\Sigma \varphi$ that is
$$\displaystyle (\Sigma \varphi)_\pi=\sum_{\phi} \langle \Sigma \varphi, \phi\rangle_{[G]} \phi$$
where the sum runs over an orthonormal basis of $\pi$ and $\langle .,.\rangle_{[G]}$ stands for the $L^2$-inner product on $[G]$ (again with respect to the Tamagawa measure).

For any cuspidal automorphic representation $\sigma$ of $G'(\bA)$, we let $\BC(\sigma)$ be the automorphic base-change of $\sigma$ to $G(\bA)$ \cite{AC}.% By \cite[Theorem 4.2]{AC} and since $n$ is odd, $\BC(\sigma)$ is a cuspidal automorphic representation of $G(\bA)$.

The following result is simply a reformulation of a theorem of Feigon-Lapid-Offen \cite[Theorem 10.2]{FLO} on the factorization of unitary periods of cuspidal automorphic representations of $G$ (following an approach of Jacquet who has established a similar result when $n=3$ for quasi-split unitary groups \cite{Jac01}). The main reason to restate the result in the form below, is to make the relation to the explicit local Plancherel decomposition of Theorem \ref{theo PLancherel formula} more transparent. In particular, we find this formulation to be pleasantly aligned with certain speculations of Sakellaridis-Venkatesh on the factorization of general spherical periods \cite[\S 17]{SV}.

\begin{theo}[Feigon-Lapid-Offen, Jacquet (n=3)]\label{theo global}
Assume that $n$ is odd. Let $\Phi=\varphi_1\otimes \varphi_2\in C_c^\infty(X(\bA))\otimes C_c^\infty(X(\bA))$ be a factorizable test function $\Phi=\prod_v \Phi_v$ and let $\pi$ be a cuspidal automorphic representation of $G(\bA)$. Assume that for at least one inert place $v$, the function $\Phi_v$ is ${k'_v}^\times$-stable and that for every place $v$, the representation $\pi_v$ is tempered. Then, we have
\begin{equation}\label{reformulation global FLO result}
\displaystyle \langle (\Sigma \varphi_1)_\pi, (\Sigma \varphi_2)_\pi\rangle_{\Pet}=\sum_{\BC(\sigma)=\pi} \langle \varphi_1,\varphi_2\rangle_{X,\sigma}
\end{equation}
where the sum runs over cuspidal automorphic representations $\sigma$ of $G'(\bA)$ such that $\BC(\sigma)=\pi$.

%%$$\displaystyle \langle (\Sigma \varphi_1)_\pi, (\Sigma \varphi_2)_\pi\rangle_{\Pet}=\left\{ \begin{array}{ll} \langle \varphi_1,\varphi_2\rangle_\sigma+\langle \varphi_1,\varphi_2\rangle_{\sigma\otimes \eta} \mbox{ if } \pi=BC(\sigma) \mbox{ for some } \sigma\hookrightarrow \cA_{\cusp}(G'); \\
%%0 \mbox{ otherwise.}
%%\end{array}
%%\right.$$
\end{theo}

\begin{proof}
Unfolding all the definitions, we arrive at
\begin{equation}\label{eq1 global}
\displaystyle \langle (\Sigma \varphi_1)_\pi, (\Sigma \varphi_2)_\pi\rangle_{\Pet}=\sum_{\phi} \frac{\langle \Sigma \varphi_1, \phi\rangle_{[G]} \langle \phi, \Sigma \varphi_2\rangle_{[G]}}{\langle \phi,\phi\rangle_{\Pet}}
\end{equation}
the sum being over an orthogonal basis of $\pi$ and
\begin{equation}\label{eq2 global}
\displaystyle \langle \Sigma \varphi_i, \phi\rangle_{[G]}=\sum_{x\in X(k')/G(k')} \int_{G_x(\bA)\backslash G(\bA)} \varphi_i(xg) \overline{P_{G_x}(R(g)\phi)} dg
\end{equation}
for $i=1,2$, where $\displaystyle P_{G_x}:\phi\mapsto \int_{[G_x]} \phi(h)dh$ denotes the period integral over $G_x$ and the measure on $G_x(\bA)$ is again the Tamagawa measure.

We now fix a global nontrivial additive character $\psi':\bA/k'\to \bC^\times$ and we set $\psi=\psi'\circ \Tra_{k/k'}:\bA_{k}/k\to \bC^\times$. For each place $v$ of $k$, we normalize the right Haar measures on the mirabolic subgroups $P'_v=P_n(k'_v)$ and $P_v=P_n(k_v)$ so that the Fourier inversion formulas \eqref{eq1 Fourier inversion} are satisfied for the local additive characters $\psi'_v$ and $\psi_v$. We also set $N'=N_{n,k'}$, $N=\Res_{k/k'} N_{n,k}$ and we equip $N'(\bA)$, $N(\bA)$ with the Haar measures giving $N'(k')\backslash N'(\bA)$, $N(k')\backslash N(\bA)$ volume $1$. With these normalizations, we can define local FLO functionals as in Section \ref{section FLO funct} by using Haar measures on the local groups $N'_v=N'(k'_v)$, $N_v=N(k'_v)$ that factorize the global ones. Finally, we define a generic character $\psi_n$ of $N(\bA)$ using the character $\psi$ as in the local case (see Section \ref{section groups, measures and the symmetric space X}).

Let $x\in X(k')$. By \cite[Theorem 10.2]{FLO}, $P_{G_x}$ vanishes on $\pi$ unless it is the base-change of some cuspidal automorphic representation $\sigma$ of $G'(\bA)$ in which case for any factorizable vector $\phi\in \pi$, we have
\begin{equation}
\displaystyle P_{G_x}(\phi)=2\alpha^\sigma_x(W_\phi)
\end{equation}
where $\displaystyle W_\phi(g)=\int_{[N]} \phi(ug)\psi_n(u)^{-1}du=\prod_v W_{\phi,v}$ is the Whittaker function associated to $\phi$ and $\alpha^\sigma_x(W_\phi)$ is defined by
$$\displaystyle \alpha_x^\sigma(W_\phi)=L(1,\sigma, \Ad\otimes \eta)\prod_{v} L(1,\sigma_v,\Ad\otimes \eta)^{-1}\alpha^{\sigma_v}_x(W_{\phi,v}).$$
From now on we assume that $\pi=\BC(\sigma)$ for some cuspidal automorphic representation $\sigma$ of $G'(\bA)$ (as otherwise the just quoted result of Feigon-Lapid-Offen implies that both sides of \eqref{reformulation global FLO result} are zero). Plugging this into \eqref{eq2 global}, we obtain
$$\displaystyle \langle \Sigma \varphi_i, \phi\rangle_{[G]}=2\sum_{x\in X(k')/G(k')}(\varphi_{i,x}\cdot\alpha^{\sigma})(W_{\phi})$$
for $i=1,2$ where $\varphi_{i,x}$ denotes the restriction of $\varphi_{i}$ to the $G(\bA)$-orbit of $x$ and we have set
$$\displaystyle (\varphi\cdot\alpha^{\sigma})(W_{\phi})=\int_{X(\bA)} \varphi(x) \alpha_{x}^\sigma(W_\phi)dx$$
for every $\varphi\in C_c^\infty(X(\bA))$ and $\phi\in \pi$. Together with \eqref{eq1 global}, this gives
\begin{align}\label{eq3 global}
\displaystyle \langle (\Sigma \varphi_1)_\pi, (\Sigma \varphi_2)_\pi\rangle_{\Pet}=4\sum_{x\in X(k')/G(k')}\sum_{\phi}\frac{(\varphi_{1,x}\cdot\alpha^{\sigma})(W_{\phi})\overline{(\varphi_{2,x}\cdot\alpha^{\sigma})(W_{\phi})}}{\langle \phi,\phi\rangle_{\Pet}}.
\end{align}
For any factorizable vector $\phi\in \pi$, we set
$$\displaystyle \langle W_\phi,W_\phi\rangle_{\Whitt}=L^*(1,\pi,\Ad)\prod_v L(1,\pi_v,\Ad)^{-1}\langle W_{\phi,v},W_{\phi,v}\rangle_{\Whitt}.$$
Then, by \cite[\S 4]{JS} (see also \cite[Eq. (10.1) p.265]{FLO} or \cite[Proposition 3.1]{Zha}\footnote{Note that the normalization of the Petterson inner product in {\em loc. cit.} is different from ours. Namely, there it is normalized as the $L^2$-inner product on $[\PGL_n]$ for the Tamagawa measure (thus giving $[\PGL_n]$ volume $n$).}), we have $\langle \phi,\phi\rangle_{\Pet}=\langle W_\phi,W_\phi\rangle_{\Whitt}$ so that \eqref{eq3 global} can be rewritten as
\begin{align}\label{eq4 global}
\displaystyle \langle (\Sigma \varphi_1)_\pi, (\Sigma \varphi_2)_\pi\rangle_{\Pet}=4\sum_{x\in X(k')/G(k')}\sum_{\phi}\frac{(\varphi_{1,x}\cdot\alpha^{\sigma})(W_{\phi})\overline{(\varphi_{2,x}\cdot\alpha^{\sigma})(W_{\phi})}}{\langle W_\phi,W_\phi\rangle_{\Whitt}}.
\end{align}
Let $\disc: X\to \mathbb{G}_m$ be the regular map that sends $x\in X$ to its discriminant in the standard basis of $k^n$. Then, by global class field theory, the natural map $X(k')/G(k')\to X(\bA)/G(\bA)$ is injective with image the set of orbits $x\in X(\bA)/G(\bA)$ such that $\eta(\disc(x))=1$. On the other hand, by \cite[Lemma 3.5]{FLO}, we have $\varphi_x\cdot \alpha^{\sigma\otimes \eta}=\eta(\disc(x)) \varphi_x\cdot \alpha^\sigma$ for $\varphi\in C_c^\infty(X(\bA))$ and $x\in X(\bA)$. This allows to rewrite the identity \eqref{eq4 global} as
\begin{align}\label{eq5 global}
\displaystyle & \langle (\Sigma \varphi_1)_\pi, (\Sigma \varphi_2)_\pi\rangle_{\Pet}=\sum_{x\in X(\bA)/G(\bA)}\sum_{\phi}\frac{(\varphi_{1,x}\cdot\alpha^{\sigma}+\varphi_{1,x}\cdot\alpha^{\sigma\otimes \eta})(W_{\phi})\overline{(\varphi_{2,x}\cdot\alpha^{\sigma}+\varphi_{2,x}\cdot\alpha^{\sigma\otimes \eta})(W_{\phi})}}{\langle W_\phi,W_\phi\rangle_{\Whitt}} \\
\nonumber & =\langle \varphi_1\cdot \alpha^\sigma,\varphi_2\cdot \alpha^\sigma\rangle_{\Whitt}+\langle \varphi_1\cdot \alpha^{\sigma\otimes \eta},\varphi_2\cdot \alpha^{\sigma\otimes \eta}\rangle_{\Whitt}+\langle \varphi_1\cdot \alpha^{\sigma}, \varphi_2\cdot \alpha^{\sigma\otimes \eta}\rangle_{\Whitt}+\langle \varphi_1\cdot \alpha^{\sigma\otimes \eta}, \varphi_2\cdot \alpha^{\sigma}\rangle_{\Whitt}
\end{align}
where as in the local case for every $\varphi\in C_c^\infty(X(\bA))$ we have identified $\varphi\cdot \alpha^{\sigma}$ and $\varphi\cdot \alpha^{\sigma\otimes \eta}$ with elements of the global Whittaker model $\cW(\pi^\vee,\psi_n^{-1})$ through the inner product $\langle .,.\rangle_{\Whitt}$. From the definitions it is clear that
$$\displaystyle \langle \varphi_1\cdot \alpha^\sigma,\varphi_2\cdot \alpha^\sigma\rangle_{\Whitt}=\langle \varphi_1,\varphi_2\rangle_{X,\sigma} \mbox{ and } \langle \varphi_1\cdot \alpha^{\sigma\otimes \eta},\varphi_2\cdot \alpha^{\sigma\otimes \eta}\rangle_{\Whitt}=\langle \varphi_1,\varphi_2\rangle_{X,\sigma\otimes \eta}$$
whereas the hypothesis that $\Phi_v$ is ${k'_v}^\times$-stable for at least one inert place $v$ implies (by \eqref{eq annulation stable}) that
$$\displaystyle \langle \varphi_1\cdot \alpha^{\sigma}, \varphi_2\cdot \alpha^{\sigma\otimes \eta}\rangle_{\Whitt}=\langle \varphi_1\cdot \alpha^{\sigma\otimes \eta}, \varphi_2\cdot \alpha^{\sigma}\rangle_{\Whitt}=0.$$
Together with \eqref{eq5 global} and the fact that the only cuspidal automorphic representations of $G'(\bA)$ with base-change $\pi$ are $\sigma$ and $\sigma\otimes \eta$ \cite[Theorem 4.2]{AC}, this gives identity \eqref{reformulation global FLO result}.
\end{proof}

\begin{finalremark}
To finish this paper, we would like to offer a word of explanation on the assumption in the theorem above and its relation to the (author's interpretation of) speculations made by Sakellaridis-Venkatesh in \cite[\S 17]{SV}\footnote{Strictly speaking, the situation considered here is not even covered in {\em loc. cit.} since they assume local multiplicity one. Therefore, our discussion should be seen as a kind of ``speculation over a speculation''.}. Namely, we can see the formal (non-convergent) expression $\RTF_{X\times X/G}(\Phi)=\langle \Sigma \varphi_1,\Sigma \varphi_2\rangle_{[G]}$ as a version of Jacquet's relative formula for the variety $X$. This expression decomposes (again formally) as a sum of orbital integrals of $\Phi$ for the diagonal action of $G$ on $X\times X$. Note that, in the case at hand, there is a stability issue: different rational orbits for this action may become the same over the algebraic closure. Therefore, a natural expectation would be that a stabilization process, similar to the one for the Arthur-Selberg trace formula, can lead to a stable version $\STF_{X\times X/G}(\Phi)$ of this trace formula. Now, we interpret\footnote{We of course try to follow the general spirit of Sakellaridis-Venkatesh's vision but any error or misinterpretation is the author's responsability only.} the speculations in \cite[\S 17]{SV} as saying that $\STF_{X\times X/G}(\Phi)$ should decompose as an integral over the $L^2$-automorphic spectrum of $G'$ (for a suitable canonical spectral measure) of the scalar product $\langle \varphi_1,\varphi_2\rangle_{X,\sigma}$. Of course, all of this is based on many formal statements that the author cannot make precise here (In particular, the scalar products $\langle .,.\rangle_{X,\sigma}$ have only been defined when $\sigma$ is tempered. The definition naturally extends to generic $\sigma$ but e.g. it is not obvious how to make sense of them for the residual representations.) but this at least can be used as a rationale for the statement of Theorem \ref{theo global}: the assumption of being ${k'_v}^\times$-stable should be seen as a weak version of stability in this context and the result roughly says that (when $n$ is odd) it is nevertheless enough to get the correct stable cuspidal contributions.
\end{finalremark}

\begin{flushleft}
Rapha\"el Beuzart-Plessis \\
Aix Marseille Univ, CNRS \\
Centrale Marseille \\
I2M  \\
Marseille \\
France
\medskip
	
email:\\
raphael.beuzart-plessis@univ-amu.fr \\
\end{flushleft}

\end{document}